\pdfoutput=1
\RequirePackage{ifpdf}
\ifpdf 
\documentclass[pdftex]{sigma}
\else
\documentclass{sigma}
\fi

\usepackage{tikz}
\usepackage{pict2e}

\usepackage{curves}
\usepackage{enumerate}

\usepackage{accents}
\usetikzlibrary{calc,decorations.markings}
\newcommand{\R}{\mathbb{R}}
\newcommand{\Z}{\mathbb{Z}}
\newcommand{\N}{\mathbb{N}}
\newcommand{\C}{\mathbb{C}}
\newcommand{\Id}{I}

\newcommand{\cK}{\mathcal{K}}
\newcommand{\cE}{\mathcal{E}}
\newcommand{\cR}{\mathcal{R}}
\newcommand{\cW}{\mathcal{W}}

\newcommand{\hw}{\widehat w}
\newcommand{\tw}{\widetilde w}

\newcommand{\bQ}{ Q}
\newcommand{\hQ}{{\widehat Q}}
\newcommand{\tQ}{{\widetilde Q}}

\newcommand{\sbQ}{Q^\star}
\newcommand{\shQ}{{\widehat Q^\star}}
\newcommand{\stQ}{{\widetilde Q^\star}}

\newcommand{\smu}{\mu^\star}
\newcommand{\shmu}{\widehat \mu^\star}

\newcommand{\bmu}{{\mu}}
\newcommand{\hmu}{{\widehat \mu}}
\newcommand{\tmu}{{\widetilde \mu}}

\newcommand{\bP}{P}
\newcommand{\hP}{{\widehat P}}
\newcommand{\tP}{{\widetilde P}}

\newcommand{\bv}{v}
\newcommand{\hv}{{\widehat v}}
\newcommand{\tv}{{\widetilde v}}

\newcommand{\sbv}{v^\star}
\newcommand{\shv}{{\widehat v^\star}}
\newcommand{\stv}{{\widetilde v^\star}}

\newcommand{\stu}{{\widetilde u^\star}}

\newcommand{\hF}{\widehat F}

\newcommand{\Fw}{\frac{F^2}{w}}
\newcommand{\hFw}{\frac{\hF^2}{\hw}}

\newcommand{\Cau}{\mathcal C_{\Sigma}}

\newcommand{\Cautv}{\mathcal C_{\tv}}

\newcommand{\sCau}{\mathcal C_{\Sigma^\star}}
\newcommand{\sCauv}{\mathcal C_{\bv^\star}}
\newcommand{\sCauhv}{\mathcal C_{\hv^\star}}
\newcommand{\sCautv}{\mathcal C_{\stv}}

\newcommand{\Ltwo}{L^2(\Sigma)}

\newcommand{\Sigl}{\Sigma^\ell}
\newcommand{\Sigr}{\Sigma^r}

\newcommand{\I}{\mathrm{i}}
\newcommand{\E}{\mathrm{e}}

\newcommand{\im}{\mathop{\mathrm{Im}}}

\newcommand{\dist}{\mathop{\mathrm{dist}}}

\newtheorem{Theorem}{Theorem}[section]
\newtheorem{Corollary}[Theorem]{Corollary}
\newtheorem{Lemma}[Theorem]{Lemma}
\newtheorem{Proposition}[Theorem]{Proposition}
 { \theoremstyle{definition}

\newtheorem{Remark}[Theorem]{Remark} }

\numberwithin{equation}{section}

\begin{document}

\allowdisplaybreaks

\renewcommand{\thefootnote}{}

\newcommand{\arXivNumber}{2307.09277}

\renewcommand{\PaperNumber}{004}

\FirstPageHeading

\ShortArticleName{Recurrence Coefficients for Orthogonal Polynomials with a Logarithmic Weight Function}

\ArticleName{Recurrence Coefficients for Orthogonal Polynomials\\ with a Logarithmic Weight Function\footnote{This paper is a~contribution to the Special Issue on Evolution Equations, Exactly Solvable Models and Random Matrices in honor of Alexander Its' 70th birthday. The~full collection is available at \href{https://www.emis.de/journals/SIGMA/Its.html}{https://www.emis.de/journals/SIGMA/Its.html}}}

\Author{Percy DEIFT~$^{\rm a}$ and Mateusz PIORKOWSKI~$^{\rm b}$}

\AuthorNameForHeading{P.~Deift and M.~Piorkowski}

\Address{$^{\rm a)}$~Department of Mathematics, Courant Institute of Mathematical Sciences,\\
\hphantom{$^{\rm a)}$}~New York University, 251 Mercer Str., New York, NY 10012, USA}
\EmailD{\href{mailto:deift@cims.nyu.edu }{deift@cims.nyu.edu}}

\Address{$^{\rm b)}$~Department of Mathematics, Katholieke Universiteit Leuven,\\
\hphantom{$^{\rm b)}$}~Celestijnenlaan 200B, 3001 Leuven, Belgium}
\EmailD{\href{mailto:mateusz.piorkowski@kuleuven.be}{mateusz.piorkowski@kuleuven.be}}

\ArticleDates{Received July 19, 2023, in final form January 01, 2024; Published online January 10, 2024}

\Abstract{We prove an asymptotic formula for the recurrence coefficients of orthogonal polynomials with orthogonality measure $\log \bigl(\frac{2}{1-x}\bigr) {\rm d}x$ on $(-1,1)$. The asymptotic formula confirms a special case of a conjecture by Magnus and extends earlier results by Conway and one of the authors. The proof relies on the Riemann--Hilbert method. The main difficulty in applying the method to the problem at hand is the lack of an appropriate local parametrix near the logarithmic singularity at $x = +1$.}


\Classification{42C05; 34M50; 45E05; 45M05}

\renewcommand{\thefootnote}{\arabic{footnote}}
\setcounter{footnote}{0}

\section{Introduction}

\subsection{Background}
\begin{figure}[t]
 \centering
 \includegraphics[scale=0.4]{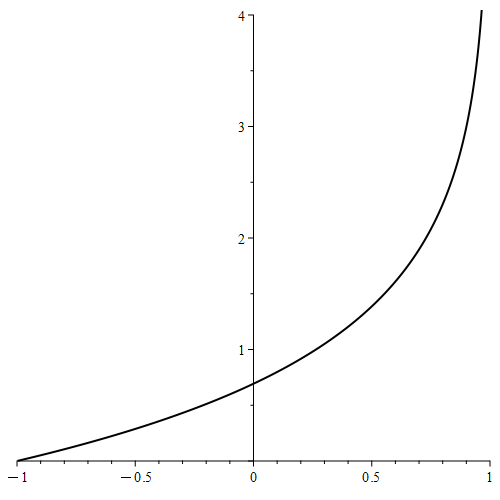}
 \caption{Plot of the weight function $w(x)$.} \label{log_weight}
\end{figure}

In this paper, we study orthogonal polynomials with orthogonality measure $w(x){\rm d}x$ given by
\begin{align}\label{wlog}
 w(x){\rm d}x = \log\left( \frac{2}{1-x}\right){\rm d}x, \qquad x \in [-1,1).
\end{align}
Note that $w(x)$ has a logarithmic singularity for $x \to +1$ and a simple zero at $x = -1$, see Figure~\ref{log_weight}.
Denote by $\lbrace p_n \rbrace_{n=0}^\infty$ the corresponding orthonormal polynomials,
\begin{align*}
 \int_{-1}^1 p_m(x) p_n(x) w(x) {\rm d}x = \delta_{mn}, \qquad m,n \in \N.
\end{align*}
The polynomials $\lbrace p_n \rbrace_{n=0}^\infty$ satisfy the three terms recurrence relation given by
\begin{align*}
 x p_n(x) = b_n p_{n+1}(x) + a_n p_n(x)+ b_{n-1} p_{n-1}(x), \qquad n \geq 1,
\end{align*}
where $a_n \in \R$ and $b_n > 0$.
 Note that our notation for $a_n$, $b_n$ is the same as in \cite{CD,PD} but opposite to the one in \cite{KMVV, Mag18}. Listed below are the first few recurrence coefficients for the weight function~$w$.
The large $n$ asymptotics of the recurrence coefficients $a_n$, $b_n$ are the main focus of the present work. To be precise we will prove the following result.
\begin{Theorem}\label{MainTheorem}
The recurrence coefficients $\lbrace a_n\rbrace_{n=0}^\infty$, $\lbrace b_n \rbrace_{n=0}^\infty$ of the orthogonal polynomials with orthogonality measure $w(x){\rm d}x$, with $w$ given in \eqref{wlog}, satisfy
\begin{align}\label{FinalEq1}
 a_n = \frac{1}{4n^2} -\frac{3}{16n^2\log^2 n} + O\left( \frac{1}{n^2\log^3 n}\right) \qquad \text{as}\quad n \to \infty
\end{align}
and
\begin{align}\label{FinalEq2}
 b_n = \frac{1}{2}-\frac{1}{16n^2}-\frac{3}{32n^2\log^2 n} + O\left(\frac{1}{n^2\log^3 n}\right) \qquad \text{as} \quad n \to \infty.
\end{align}
\end{Theorem}
Comparing \eqref{FinalEq1} and \eqref{FinalEq2} with Table~\ref{tableCoefficients}, we see that, already for $n = 4$, $a_n$ and $b_n^2$ are close to their limiting values, $0$ and $\frac{1}{4}$ respectively, up to the second digit.

\begin{table}[t]\centering
\renewcommand{\arraystretch}{1.4}
\scalebox{0.95}{
\begin{tabular}{ c|c|c|c|c|c }
 \hline
 $n$ & $0$ & $1$ & $2$ & $3$ & $4$\\
 \hline
 $a_n$ & $\frac{1}{2}$ & $\frac{1}{14}$ & $\frac{263}{9058}$ \scriptsize $\approx 0.029$ & $\frac{1995511}{126347454}$ \scriptsize $\approx 0.016$& $\frac{436364251361}{43886567673522}$ \scriptsize $\approx 0.010$\\
 \hline
 $b_n^2$ & $\frac{7}{36}$ & $\frac{2588}{11025}$ & $\frac{71180289}{293026300}$ \scriptsize $\approx 0.243$& $\frac{1329399823424}{5405644687527}$ \scriptsize $\approx 0.246$& $\frac{39672481023099631594375}{160381475127054568640484}$ \scriptsize $\approx 0.247$
\end{tabular}}
\caption{First recurrence coefficients for the weight function $w$.}\label{tableCoefficients}
\end{table}

Theorem \ref{MainTheorem} is a special case of a conjecture by Magnus, who analyzed in~\cite{Mag18} a few examples of continuous weight functions with logarithmic singularities at the edge and in the bulk of the support, among others $-\log(t)$ with $t \in (0, 1]$ which is equivalent to~\eqref{wlog} after the affine change of variables $x = 1-2t$ (see the remark below). More generally, for the case of a logarithmic singularity at the edge of the support, Magnus considered $w_M(x)$, now supported, without loss of generality, on $x \in (-1,1)$, satisfying the following two conditions:
\begin{itemize}\itemsep=0pt
 \item $w_M(x)/(1+x)^{\beta}$ has a positive finite limit for $x \to -1$,
 \item $w_M(x)/[-(1-x)^{\alpha}\log(1-x)]$ has a positive finite limit for $x \to +1$,
\end{itemize}
where $\alpha, \beta > -1$. He conjectured based on numerical evidence that the recurrence coefficients~$a_{M,n}$ and $b_{M,n}$ of the corresponding orthogonal polynomials satisfy for $n \to \infty$,
\begin{align}\label{MagA}
 a_{M,n} = \frac{\beta^2-\alpha^2}{4n^2} + \frac{2B}{n^2\log n} + \frac{2C}{n^2 \log^2 n} + o\big((n \log n)^{-2}\big)
\end{align}
and
\begin{align}\label{MagB}
 b_{M,n} = \frac{1}{2} - \frac{\alpha^2 +\beta^2-\frac{1}{2}}{8n^2} + \frac{B}{n^2 \log n} + \frac{C}{n^2 \log^2 n} + o\big((n \log n)^{-2}\big).
\end{align}
Additionally, it was conjectured that for the special case $w_M = w$ from \eqref{wlog}, $B = 0$ and $C = -\frac{3}{32}$ holds, which is confirmed by Theorem~\ref{MainTheorem}. Note that for $w$, we have $\alpha = 0$ and $\beta = 1$.
\begin{Remark}
 For a general orthogonality measure ${\rm d}\mu(x)$ supported on $x\in (-1,1)$ with recurrence coefficients $A_n$, $B_n$, the orthogonality measure defined by ${\rm d}\widetilde \mu(t) := {\rm d}\mu(1-2t)$ with~${t \in (0,1)}$ leads to recurrence coefficients given by
 \begin{align*}
 \widetilde A_n = \frac{1}{2}-\frac{A_n}{2}, \qquad \widetilde B_n = \frac{B_n}{2}.
 \end{align*}
\end{Remark}

\subsection{State of the art}
Earlier work on Magnus' conjecture was done by Conway and one of the authors in \cite{CD} using Riemann--Hilbert (RH) techniques. The weight function considered therein had the form
\begin{align}\label{wk}
 w_k(x)= \log \left(\frac{2k}{1-x}\right), \qquad x \in [-1,1), \quad k > 1.
\end{align}
The authors prove the conjecture for this special case corresponding to $\alpha = \beta = 0$, and also obtain $B = 0$ and $C = -\frac{3}{32}$, further suggesting that these constants do not depend on the behaviour of the weight function away from the logarithmic singularity.

\begin{Theorem}[{\cite{CD}}]\label{CDTheorem}
 The recurrence coefficients $\big\lbrace a_n^{(k)}\big\rbrace_{n=0}^\infty$, $\big\lbrace b_n^{(k)} \big\rbrace_{n=0}^\infty$ of the orthogonal polynomials with orthogonality measure $w_k(x){\rm d}x$, with $w_k$ given in \eqref{wk}, satisfy
\begin{align*}
 a_n^{(k)} = -\frac{3}{16n^2\log^2 n} + O\left( \frac{1}{n^2\log^3 n}\right) \qquad \text{as} \quad n \to \infty
\end{align*}
and
\begin{align*}
 b_n^{(k)} = \frac{1}{2}+\frac{1}{16n^2}-\frac{3}{32n^2\log^2 n} + O\left(\frac{1}{n^2\log^3 n}\right) \qquad \text{as} \quad n \to \infty.
\end{align*}
\end{Theorem}
Note that for $k > 1$, the weight function $w_k(x)$ has a positive finite value at $x = -1$, while $\lim_{k\to 1+} w_k(x) = w(x)$ for $x \in [-1,1)$. Hence, we see different $1/n^2$-terms in Theorem \ref{MainTheorem} compared to Theorem \ref{CDTheorem}, consistent with the conjectural \eqref{MagA} and \eqref{MagB}. Also observe that terms of order $1/\big(n^2 \log^2 n\big)$ are identical in both theorems, in particular they do not depend on~$k$. These terms can be interpreted as contributions from the logarithmic singularity at $x = +1$, which does not depend on $k \geq 1$.

For $k < 1$, the weight function $w_k$ would have a simple zero in the interior of $(-1,1)$ and the uniqueness of the corresponding polynomials is no longer guaranteed. We will not deal with this case in the present paper.

The main difficulty encountered in \cite{CD} was the lack of a known parametrix in the vicinity of the logarithmic singularity, a key ingredient in the usual nonlinear steepest descent analysis. Hence, the authors relied instead on a technically involved comparison to the Legendre problem with weight function $w_{Leg}(x) \equiv 1$, $x \in (-1,1)$. Surprisingly, their argument could not be generalized in an obvious way to the weight function $w = \lim_{k\to 1} w_k$, due to the appearance of a~simple zero of $w$ at $x = -1$. While, an analogous comparison to the Jacobi problem with weight function~$w_{Jac}(x) = 1+x$ (or something similar) seems suggestive, significant challenges remain due to the presence of the simple zero. In particular, the crucial Theorem 4.7 in \cite{CD} requires a different proof in this case: now one needs to control the behaviour of the Cauchy operator acting on spaces with Muckenhoupt weights.

Orthogonal polynomials with logarithmic weight functions, have applications to both pure mathematics and physics. In particular, apart from logarithmic singularities, these weight functions also tend to have zeros (see \cite[Section~8]{Mag18} and \cite{VA}). Such applications motivate us in the present paper to extend the results in~\cite{CD} to the weight function $w$ in~\eqref{wlog}, having both a~logarithmic singularity and a simple zero.

\begin{Remark}
 One might think, a priori, that the vanishing of a weight $w(x)$ at a point should not give rise to serious technical difficulties. Naively, it would appear that only singularities in the weight, and not zeros, should present obstacles. In this regard, we recall the hope and the prophecy of Lenard's 1972 paper~\cite{Lenard}:
 \begin{quote}
 ``It is the author’s hope that a
rigorous analysis will someday carry the results to the point where the true role of the zeros of the
generating function will be understood. When that day comes a capstone will have been put on a
beautiful edifice to whose construction many contributed and whose foundations lie in the studies
of Gabor Szeg\H{o} half a century ago''.
\end{quote}
\end{Remark}

\subsection[Relation of the present work to cite\{CD\}]{Relation of the present work to \cite{CD}}

Significant parts of the analysis performed in the present paper are based on the analysis introduced in \cite{CD}. Hence, we will repeatedly refer to that paper for proofs of certain statements. This is justified by the fact that the majority of estimates found in \cite{CD} do not depend on the distinction $k > 1$ and $k = 1$ in \eqref{wk}. Thus, the proofs of many of the results will also hold for the weight function \eqref{wlog} that we are interested in. There are however certain propositions which have their analogs in~\cite{CD}, but still deserve a separate proof due to some minor differences. These are Proposition~\ref{Proprr} which is the analog of Proposition~2.5 in~\cite{CD}, and Proposition~\ref{IntegralEst2} which is the analog of Propositions~5.3 and~5.4 in~\cite{CD}. In the case of Proposition~\ref{Proprr}, it is necessary to prove a slightly more general result than Proposition~2.5 in \cite{CD}. Meanwhile, Proposition~6.5 contains an application of Proposition~\ref{Proprr} in its more general form, and uses different RH solutions than Propositions~5.3 and 5.4 in \cite{CD}. Both results are proven in Appendix~\ref{appendixA}.

Finally, let us reiterate that the analog of Theorem~4.7 in \cite{CD} concerning the uniform boundedness of the inverse of a certain singular integral operator requires a completely different approach to the case $k = 1$, due to the appearance of a simple zero in the weight function $w$ at $x = -1$. This necessitates the construction of an appropriate local parametrix in the vicinity of the zero, which is then used to invert the RH problem locally. While the local parametrix is well-known and can be expressed in terms of Bessel functions, see \cite[Section~6]{KMVV}, its appearance significantly complicates the analysis that follows. Crucially, the method of proving Theorem~4.7 in~\cite{CD} is no longer sufficient in this new setting. In fact, the material found in Sections \ref{Section_LegendreResolvent}--\ref{ModifiedRHP} is entirely devoted to formulating and proving Theorems~\ref{UniformBoundedness} and~\ref{TheoremUniInv}, which are the analogs of Theorem~4.7 in \cite{CD}. Here, a key role is played by Proposition~\ref{ExpForm}, which in a sense localizes the effect that the logarithmic singularity has on the uniform invertibility of the associated singular integral operator. Interestingly, Theorem~4.7 in~\cite{CD} itself plays a crucial role in the proofs of these results. Sections \ref{Section_LegendreResolvent}--\ref{ModifiedRHP} contain the main novelties of the present work.

\subsection{Outline of the paper}

In the following, we will briefly summarize the content of each section.
\begin{itemize}\itemsep=0pt
 \item In Section~\ref{AuxiliaryFunctions}, we introduce two auxiliary weight functions, the model and the Legendre weight function, together with related quantities, which will be relevant for the RH analysis. We also list certain estimates and asymptotic results which will be used in later sections.
 \item In Section~\ref{RHProblems}, we introduce the Fokas--Its--Kitaev RH problem for orthogonal polynomials. We proceed to perform the necessary conjugation and deformation steps to arrive at three distinct RH problems amenable to asymptotics analysis: the logarithmic, model and Legendre RH problems. We then state the relation between solutions of these RH problems and the corresponding recurrence coefficients.
 \item In Section~\ref{Section_LegendreResolvent}, we derive an explicit formula for the Legendre resolvent, that is, the inverse of a singular integral operator associated to the Legendre RH problem. This formula is not used in \cite{CD}.
 \item In Section~\ref{LocPar}, we perform a detailed analysis of the known local parametrices for the RH problem near the point $-1$, which can be constructed explicitly using Bessel functions, as in \cite[Section~6]{KMVV}. This leads, in particular, to uniform asymptotics on their growth as~$n \to \infty$.
 \item In Section~\ref{ModifiedRHP}, we introduce modified versions of the three previously mentioned RH problems using the appropriate local parametrices from Section~\ref{RHProblems} around $z = -1$.
 These modified RH problems are better suited when comparing their associated resolvents.
 Thus, by showing the uniform invertibility of the modified Legendre resolvent, we obtain the uniform invertibility of the modified logarithmic and model resolvents, thereby proving an analog of Theorem~4.7 in~\cite{CD}.
 \item In Section~\ref{SectAsymptoticAnalysis}, we derive an asymptotic formula expressing the difference between the recurrence coefficients for the logarithmic weight and the recurrence coefficients for the model weight. The aforementioned uniform invertibility of the associated resolvents plays a crucial part in this argument. As the asymptotics of the recurrence coefficients for the model weight is known, Theorem~\ref{MainTheorem} follows.
 \item In Appendix~\ref{appendixA}, we provide some proofs of more technical nature that are omitted from the main text.
\end{itemize}
It is also worth mentioning that a common step in the RH analysis -- the construction of a local parametrix -- is not performed in the vicinity of the logarithmic singularity at $+1$. The reason is, quite simply, as in \cite{CD}, that we were unable to find the local parametrix in the presence of such a singularity. Similar instances in which the local parametrix was not constructed explicitly can be found in \cite[Section~5]{DKMVZ1} and~\cite{KMcL99}, where non-constructive Fredholm methods were used instead. For a discussion of RH problems without explicitly solvable local parametrices, see~\cite{MP21}.

\subsection{Notation}

Throughout this paper, all contours that arise are finite unions of smooth and oriented arcs, with a finite number of points of (self)intersection. More details can be found in the book~\cite{BK} which treats a more general class of so-called Carleson contours.

Let $\Gamma \subset \C$ be such a contour and $m$ an analytic function on $\C \setminus \Gamma$. For $s \in \Gamma$, we will denote by $m_{\pm}(s)$ the limit of $m(z)$ as $z \to s\pm$, provided this limit exists. The notation $z \to s\pm$ denotes a nontangential limit in $\C \setminus \Gamma$ to $s \in\Gamma$, from the $+$, resp.\ $-$, side of the contour, see Figure~\ref{FigContour}. Recall that as $\Gamma$ is taken to be oriented, this notion is well-defined away from the points of intersection. Everything generalizes to matrix-valued functions $m$ in a straightforward manner.

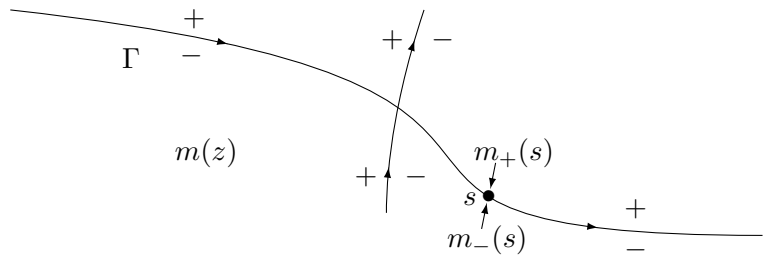
\begin{figure}[th]\centering\vspace*{2mm}
\begin{tikzpicture}
\draw (-5,2) .. controls (4,1) and (-2,-1) .. (5,-1);
\draw (0,-0.7) .. controls (0,-0.5) and (0,0.5) .. (0.5,2);
\put(-100pt,35pt){$\Gamma$}
\put(-80pt,0pt){$m(z)$}

\put(-60pt,44pt){\vector(4,-1){0.2pt}}
\put(-77pt,51pt){$+$}
\put(-78pt,37pt){$-$}

\put(80pt,-25.5pt){\vector(8,-1){0.2pt}}
\put(90pt,-21pt){$+$}
\put(90pt,-36.5pt){$-$}

\put(10.5pt,45pt){\vector(1,3){0.2pt}}
\put(-2pt,43pt){$+$}
\put(17pt,43pt){$-$}

\put(0.7pt,-4pt){\vector(1,8){0.2pt}}
\put(-12pt,-9pt){$+$}
\put(7pt,-9pt){$-$}


\put(23pt,-33pt){$m_-(s)$}
\put(33pt,0pt){$m_+(s)$}
\put(36pt,-16pt){$\bullet$}
\put(36pt,-26pt){\vector(0.4,2){2pt}}
\put(41.3pt,-1pt){\vector(-0.4,-2){2pt}}
\put(29pt,-17pt){$s$}
\end{tikzpicture}
\caption{A contour with a point of intersection.}\label{FigContour}
\end{figure}

We will denote by $\C_\pm$ the upper, resp.~lower open half plane of $\C$ and use the notation $\overline{\C}$ for the Riemann sphere $\C \cup \lbrace \infty \rbrace$. Unless specified otherwise, $z^{1/2}$, $z \in \C \setminus (-\infty, 0)$ will denote the principal branch of the square root.

For two sequences $A_n$ and $B_n$ in a normed space, we will use the notation $A_n \lesssim B_n$ if there exists a $c > 0$ and $N \in \mathbb{N}$ such that $\Vert A_n \Vert\leq c\Vert B_n \Vert$ for all $n \geq N$. Equivalently, we will sometimes use the notation $A_n = O(B_n)$. A similar definition holds if $n$ is substituted for a continuous variable, e.g., $f(x) \lesssim g(x)$ (equivalently $f(x) = O(g(x))$) for $x \to x_0$ means $\Vert f(x) \Vert \leq c \Vert g(x) \Vert$ for some $c > 0$ and all $x$ satisfying $|x-x_0| \leq \varepsilon$.

Finally, for a $d\times d$ dimensional measurable matrix-valued function $f(s)$, $s \in \Gamma$, we write $f \in L^p(\Sigma)$, with $p \in [1,\infty)$, if and only if \begin{align*}
 \Vert f \Vert_{L^p(\Gamma)} := \Biggl(\int_\Gamma \sum_{i,j=1}^d |f_{ij}(s)|^p |{\rm d}s|\Biggr)^{\frac{1}{p}} < \infty,
\end{align*}
where $|{\rm d}s|$ denotes the arc length measure of $\Gamma$, and $f \in L^\infty(\Gamma)$ if and only if
\begin{align*}
 \Vert f \Vert_{L^\infty(\Gamma)} := \max_{i,j = 1 \dots d}\lbrace \Vert f_{ij} \Vert_{L^\infty(\Gamma)}\rbrace < \infty.
\end{align*}
In particular, $\Vert f \Vert_{L^2(\Gamma)}$ denotes the $L^2$-norm on $\Gamma$ of the Hilbert--Schmidt norm of $f(s)$.
Generalizations to weighted $L^p$-spaces are introduced in Section~\ref{Section_LegendreResolvent}.


\section{Auxiliary functions used in the RH analysis}
\label{AuxiliaryFunctions}
\subsection{The model and Legendre weight functions}
To obtain the recurrence coefficients related to the weight function $w$, we will have to compare with a different \emph{model} weight function given by
\begin{align*}
 \hw(x) = (1+x)\E^{d_0x}, \qquad x \in [-1,1],
\end{align*}
where $d_0 \in \R$ is determined via \eqref{d0}. As we will see, the choice of $d_0$ gives an error estimate in \eqref{3/2decay} of order $O\big(|1+z|^{3/2}\big)$, rather that $O\big(|1+z|^{1/2}\big)$, as in part 3 of Proposition~2.3 in \cite{CD}. This extra decay as $z \to -1$ considerably simplifies the proof of the key Lemma \ref{FinalLemma}.

Note that $\hw$ can be analytically extended to an entire function. Moreover, $\hw(x)$ has a~simple zero at $x = -1$ and a finite positive value at $x = +1$. As it lies in the class of weight functions considered in \cite{KMVV}, the corresponding RH analysis is well understood. Observe that our requirements on $\hw$ do not specify it uniquely, hence one could have performed the comparison argument with other choices of model weight functions as well.

As we will heavily rely on the arguments found in \cite{CD}, we will also introduce the Legendre weight function
\begin{align*}
 \tw(x) = 1, \qquad x \in [-1,1].
\end{align*}
As shown in \cite[Theorem 4.7]{CD} (see also Theorem \ref{UnifLeg}), the weight $\tw$ gives rise to a singular integral operator defined in Section~\ref{Section_LegendreResolvent}, with an inverse that is uniformly bounded as $n \to \infty$. However, $\tw$ does not approximate the logarithmic weight function $w$ for $x\to -1$, due to the presence of a simple zero at that point. In contrast, the weight $\hw$ approximates the logarithmic weight function $w$ for $x \to -1$, however for $\hw$ the analogous singular integral operator is not invertible in $L^2$ (see property (iv) in RH problem (see Section~\ref{ModRHP}) showing that the solution will not be square integrable). This is the essential technical difficulty that we face in this paper.

\subsection[The Szego functions]{The Szeg\H{o} functions}
To perform the nonlinear steepest descent analysis, we need to define the Szeg\H{o} function $F$ associated to the logarithmic weight function $w$:
\begin{align}\label{SzegoF}
 F(z) = \exp \Bigg( \dfrac{\big(z^2-1\big)^{1/2}}{2 \pi}\int_{-1}^1 \dfrac{\log w(s)}{\sqrt{1-s^2}} \dfrac{{\rm d}s}{z-s} \Bigg), \qquad z \in \C \setminus [-1,1].
\end{align}
Here $\big(z^2-1\big)^{1/2}$ is uniquely specified as an analytic function having a branch cut along $(-1,1)$, and \smash[b]{$\big(z^2-1\big)^{1/2} \approx z$} for $z\to \infty$. Analogously, we define \smash{$\hF$} to be the Szeg\H{o} function associated to the model weight function $\hw$:
\begin{align*}
 \hF(z) = \exp \Bigg( \dfrac{\big(z^2-1\big)^{1/2}}{2 \pi}\int_{-1}^1 \dfrac{\log \hw(s)}{\sqrt{1-s^2}} \dfrac{{\rm d}s}{z-s} \Bigg), \qquad z \in \C \setminus [-1,1].
\end{align*}
Note that as $\log \tw \equiv 0$, the Szeg\H{o} function for the Legendre weight $\tw$ is trivial: $\widetilde F \equiv 1$.

The Szeg\H{o} functions $F$, $\hF$ satisfy the following properties which are crucial for the RH analysis.

\begin{Proposition}\label{PropertiesF}
The functions $F, \hF \colon \C \setminus [-1,1]\to \C$ satisfy the following properties:
\begin{enumerate}[$(i)$]\itemsep=0pt
 \item $F(z)$, $\hF(z)$ are analytic for $z \in \C \setminus [-1,1]$, with $F(\overline{z}) = \overline{F(z)}$ and $\hF(\overline{z}) = \overline{\hF(z)}$,
 \item $\lim_{z\to\infty} F(z) = F_\infty \in \R_+$, $\lim_{z\to\infty} \hF(z) = \widehat F_\infty \in \R_+$,
 \item $F_+(x)F_-(x) = w(x)$, $\hF_+(x)\hF_-(x) = \hw(x)$ for $x \in [-1,1)$,
 \item $|F_\pm(x)|^2 = w(x)$, $\big|\hF_\pm(x)\big|^2 = \hw(x)$ for $x \in [-1,1)$.
\end{enumerate}
\end{Proposition}
Note that at $x = -1$, the limits of $F(z)$, $\hF(z)$ are equal to $0$. In particular, they are independent of the path $z \to -1$, and we write $F_\pm(-1) = \lim_{z\to -1} F(z) = \lim_{x \downarrow -1} F_\pm(x) = 0$ and $\hF_\pm(-1) = \lim_{z\to -1} \hF(z) = \lim_{x \downarrow -1} \hF_\pm(x) = 0$ in this case.
We also introduce the function~$\phi$ given by (see \cite[Proposition~2.1]{CD})
\begin{align*}
 \phi(z) = z + \big(z^2-1\big)^{1/2}, \qquad z \in \C \setminus [-1,1].
\end{align*}
With this choice, we see that $\phi$ defines a biholomorphism between $\C \cup\lbrace \infty \rbrace \setminus [-1,1]$ and ${\C \cup \lbrace \infty \rbrace \setminus \lbrace z\colon |z| \leq 1 \rbrace}$, mapping $\infty$ to itself. In particular, the function $\phi$ satisfies
\begin{align}\label{phiGreater1}
 |\phi(z)| > 1, \qquad z \in \C \setminus [-1,1],
\end{align}
and
\begin{align*}
 \phi(z) = 2z + O\left(\frac{1}{z}\right) \qquad \mbox{as}\quad z \to \infty.
\end{align*}
Moreover, the inequality \eqref{phiGreater1} holds uniformly away from the interval $[-1,1]$, while on the interval we have
\begin{align}\label{PhionInterval}
 \lim_{z\to x}|\phi(z)| = 1, \qquad x\in[-1,1].
\end{align}
Additionally, as $\phi(\overline{z}) = \overline{\phi(z)}$ for $z \in \C \setminus [-1,1]$, it follows from \eqref{PhionInterval} that
\begin{align*}
 \phi_+(x)\phi_-(x) = 1, \qquad x \in (-1,1).
\end{align*}
Near the points $z = \pm 1$, we have
\begin{align}\label{phiEndpoints}
\begin{split}
 \phi(z) &= 1 + \sqrt{2}(z-1)^{1/2} + O(|z -1|), \qquad z \to +1,
 \\
 \phi(z) &= -1 \pm \sqrt{2}\I(z+1)^{1/2} + O(|z+1|), \qquad z \to -1, \quad z \in \C_\pm.
\end{split}
\end{align}

For the subsequent analysis, it is necessary to understand the behaviour of the functions $\frac{F^2}{w}(z)$, $\frac{\hF^2}{\hw}(z)$ near the points $z = \pm 1$, as these functions show up in the jump matrices of the corresponding RH problems, see Section~\ref{RHProblems}.
\begin{Proposition} \label{F^2/w}
The function $\frac{F^2}{w} \colon \C \setminus [-1,\infty) \to \C$ satisfies
\begin{align}\label{F^2/w at +1}
 \frac{F^2}{w}(z) = 1 \mp \frac{\I \pi}{w(z)}-\frac{\pi^2}{2w^2(z)} +O\left(\frac{1}{\log^3(z-1)}\right)
\end{align}
uniformly for any path $z\to +1$ in $\C \setminus [-1,\infty)$, where the $-$ is taken for~$\im(z)> 0$ and $+$ for~$\im(z) < 0$, and also
\begin{align}\label{F^2/w at -1}
 \frac{F^2}{w}(z) = \phi(z)^{-1}\exp\bigl\lbrace-\big(z^2-1\big)^{1/2}\big(d_0+O(|z+1|)\big)\bigr\rbrace,
\end{align}
uniformly for any path $z\to-1$ in $\C \setminus [-1,\infty)$, where
\begin{align}\label{d0}
 d_0 = \frac{1}{2\pi\I} \int_\gamma \frac{\log\big(w(\zeta)/(1+\zeta)\big)}{(\zeta^2-1)^{1/2}}\frac{{\rm d}\zeta}{\zeta+1}
\end{align}
and $\gamma$ is an oriented contour originating from the point $\zeta = 1$, going anticlockwise around the interval $[-1,1]$ and ending again at the point $\zeta = 1$ as depicted in Figure~{\rm \ref{ContourGamma}}.
\end{Proposition}

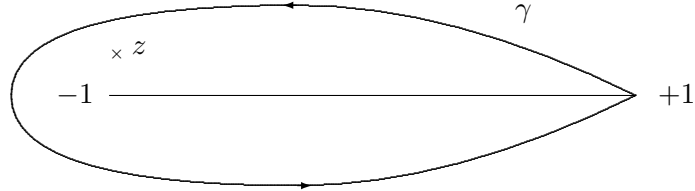
\begin{figure}[th]\vspace*{8mm}
 \centering
 \hspace*{5mm}\begin{picture}(0,0)
 \put(-3.5,0){\line(1,0){7.0}}
 \put(-1,1.2){\vector(-1,0){0.2}}
 \put(-1,-1.2){\vector(1,0){0.2}}

 \curve(3.5,0, -1,1.2, -4.8,0, -1,-1.2, 3.5,0)

 \put(3.8,-0.1){$+1$}
 \put(-4.2,-0.1){$-1$}
 \put(-3.5,0.5){\text{\tiny{$\times$}}}
 \put(-3.2,0.55){$z$}

 \put(1.9,1.05){$\gamma$}
 \end{picture}
 \vspace{10mm}
 \caption{The contour $\gamma$ encircling the point $z$.} \label{ContourGamma}
\end{figure}

\begin{proof}
For the proof of statement \eqref{F^2/w at +1}, see \cite[Proposition~A.1]{CD}. Statement \eqref{F^2/w at -1} is a special case of \cite[Lemma~6.6]{KMVV}, but with an additional restriction on the choice of the contour $\gamma$ stemming from the logarithmic singularity of $w$. Due to this technicality, we repeat the proof found therein.

First, let us note that if $F_1$ is the Szeg\H{o} function of a weight $w_1$ and $F_2$ the Szeg\H{o} function of a weight $w_2$, then the Szeg\H{o} function $F_{12}$ of the product $w_1w_2$ is given by the product of the individual Szeg\H{o} functions: $F_{12} = F_1 F_2$. As the Szeg\H{o} function of a Jacobi weight $w_{\alpha,\beta}(x) = (1-x)^\alpha (1+x)^\beta$ with $\alpha, \beta > -1$, is given by (see \cite[Remark 5.1]{KMVV})
\begin{align*}
 F_{\alpha,\beta}(z) = \frac{(z-1)^{\alpha/2}(z+1)^{\beta/2}}{\phi(z)^{\frac{\alpha+\beta}{2}}},
\end{align*}
we can conclude that
\begin{align}
 F(z) ={}& \frac{(z+1)^{1/2}}{\phi(z)^{1/2}}\nonumber\\
 &\times\exp \Bigg( \dfrac{\big(z^2-1\big)^{1/2}}{2 \pi}\int_{-1}^1 \dfrac{\log (w(s)/(1+s))}{\sqrt{1-s^2}} \dfrac{{\rm d}s}{z-s} \Bigg),
 \qquad z \in \C \setminus [-1,1].\label{SzegoFact}
\end{align}
Again, $(z+1)^{1/2}$ and $\phi(z)^{1/2}$ denote the principal branches (a simple calculation shows that $\phi(z) \not \in (-\infty, -1]$ for $z \in \C \setminus (-\infty, 1]$). Moreover, for $x < -1$, $\big(\phi(x)^{1/2}\big)_+ = -\big(\phi(x)^{1/2}\big)_-$ as~$\phi(z) \sim 2z$ for $z \to \infty$. Thus $(z+1)^{1/2}/\phi(z)^{1/2}$ is indeed analytic in $\C\setminus [-1,1]$.

To analyze the argument of the exponential in \eqref{SzegoFact}, we choose for a fixed $z \in \C \setminus [-1,1]$ a~contour $\gamma$ as shown in Figure~\ref{ContourGamma}. Then a residue calculation shows that
\begin{align}\nonumber
 &\frac{1}{\pi} \int_{-1}^1 \frac{\log(w(s)/(1+s))}{\sqrt{1-s^2}} \frac{{\rm d}s}{z-s}
 \\\label{ResiCalc}
 &\qquad= \frac{\log(w(z)/(1+z))}{\big(z^2-1\big)^{1/2}} - \frac{1}{2\pi\I} \int_\gamma\frac{\log(w(\zeta)/(1+\zeta))}{\big(\zeta^2-1\big)^{1/2}} \frac{{\rm d}\zeta}{\zeta-z}.
\end{align}
Here we note the $w(z)/(1+z)$ is analytic and non-zero in the simply connected region $\C \setminus [1,\infty)$, and hence $\log(w(z)/(1+z))$ exists and is analytic in this region. Moreover, the integrand in~\eqref{ResiCalc} is integrable along $\gamma$ as long as $z \not \in \gamma$.

Note also that as $\log(w(\zeta)/(1+\zeta))$ has an iterated logarithmic singularity at $\zeta = 1$, we cannot deform $\gamma$ away from this point. Plugging \eqref{SzegoFact} and \eqref{ResiCalc} into the definition of $\Fw$, we obtain
\begin{align*}
 \Fw(z) = \phi(z)^{-1} \exp \Biggl( - \frac{\big(z^2-1\big)^{1/2}}{2\pi\I} \int_\gamma\frac{\log(w(\zeta)/(1+\zeta))}{\big(\zeta^2-1\big)^{1/2}} \frac{{\rm d}\zeta}{\zeta-z} \Biggr).
\end{align*}
We compute the Taylor expansion around $\zeta = -1$ to obtain
\begin{align*}
 \frac{1}{2\pi\I} \int_\gamma\frac{\log(w(\zeta)/(1+\zeta))}{\big(\zeta^2-1\big)^{1/2}} \frac{{\rm d}\zeta}{\zeta-z} = \sum_{k=0}^\infty d_k (z+1)^k,
\end{align*}
where $d_k$ is given by
\begin{align*}
 d_k = \frac{1}{2\pi\I} \int_\gamma \frac{\log\big(w(\zeta)/(1+\zeta)\big)}{(\zeta^2-1)^{1/2}}\frac{{\rm d}\zeta}{(\zeta+1)^{k+1}}.
\end{align*}
This finishes the proof.
\end{proof}

The analog of Proposition~\ref{F^2/w} for $\hFw$ is more elementary.

\begin{Proposition}
The function $\frac{\hF^2}{\hw} \colon \C \setminus [-1,1] \to \C$ satisfies
\begin{align}\label{hat F^2/w at +1}
 \frac{\hF^2}{\hw}(z) = 1 + O\big(|z-1|^{1/2}\big)
\end{align}
uniformly for any path $z \to +1$ in $\C \setminus [-1,1]$ and can be written as
\begin{align}\label{hat F^2/w at -1}
 \frac{\hF^2}{\hw}(z) = \phi(z)^{-1}\exp\bigl\lbrace-\big(z^2-1\big)^{1/2}d_0\bigr\rbrace.
\end{align}
\end{Proposition}
\begin{proof}
Statement \eqref{hat F^2/w at +1} follows directly from \cite[Lemma~6.4]{KMVV}, while statement \eqref{hat F^2/w at -1} can be obtained with the same line of reasoning as in Proposition~\ref{F^2/w}, but now the integral simplifies:
\begin{align*}
 \frac{1}{2\pi\I} \int_\gamma\frac{\log(\hw(\zeta)/(1+\zeta))}{\big(\zeta^2-1\big)^{1/2}} \frac{{\rm d}\zeta}{\zeta-z} = \frac{d_0}{2\pi\I} \int_\gamma\frac{\zeta}{\big(\zeta^2-1\big)^{1/2}} \frac{{\rm d}\zeta}{\zeta-z}.
\end{align*}
In particular, we can deform the contour $\gamma$ to infinity leading to $\int_\gamma\frac{\zeta}{(\zeta^2-1)^{1/2}} \frac{{\rm d}\zeta}{\zeta-z} = 2\pi\I$ and finishing the proof.
\end{proof}

From now, on estimates of the type \eqref{hat F^2/w at +1} are always understood to be uniform for any path taken in the appropriate domain. The following two corollaries contain important estimates related to the behaviour of the functions $\Fw$ and $\hFw$ near the critical points $\pm 1$.
\begin{Corollary}[{\cite[Proposition~2.4]{CD}}]\label{CorF+1} For $x \to 1$ such that $x > 1$, we have
\begin{align}\label{Diff F^2/w at +1}
\frac{F^2}{w_+}(x)+\frac{F^2}{w_-}(x) - 2 = - \frac{3\pi^2}{\log^2 \frac{2}{x-1}} + O\left(\frac{1}{\log^3(x-1)}\right).
\end{align}
\end{Corollary}

\begin{Corollary}\label{CorF-1}
 For $z \to -1$, we have
\begin{align}\label{3/2decay}
 \Fw(z)-\hFw(z) = O\big(|z+1|^{3/2}\big).
\end{align}
\end{Corollary}
\begin{proof}
The statement follows from \eqref{F^2/w at -1} and \eqref{hat F^2/w at -1}.
\end{proof}

\begin{Remark}
As noted earlier, the estimate \eqref{3/2decay} is the motivation for considering the particular weight function $\hw(x)$ instead of the simpler Jacobi weight function $1+x$.
\end{Remark}

For later analysis, we will also need the following technical result.

\begin{Proposition} \label{Proprr}
Fix $R > 0$. Let $r_n, \tilde r_n \in (0, 1)$, $n \in \N$ be two sequences satisfying $r_n, \tilde r_n \to 0$, such that $n\big|\frac{r_n}{\tilde r_n}-1|< R$. Then
\begin{align*}\nonumber
 &\frac{F^2}{w_+}(1+r_n)-\frac{F^2}{w_+}(1+\tilde r_n)+\frac{F^2}{w_-}(1+r_n)-\frac{F^2}{w_-}(1+\tilde r_n)
 \\
 &\qquad= O(r_n \log |\log r_n |) + O\left(\frac{1}{n\log^3 r_n}\right) + O\big(n^{-2}\big),
\end{align*}
where the implied constants in the $O$-terms depend only on $R$.
\end{Proposition}
\begin{proof}
For the proof see Proposition~\ref{ProprrAp}.
\end{proof}
\begin{Remark}
The original formulation found in \cite[Proposition~A.4]{CD} assumes that $r_n, \tilde r_n = \smash{O\big(\frac{1}{n^2}\big)}$, but does not take into account the fact that the error term will additionally depend on the convergence rate of the sequences $r_n$, $\tilde r_n$, i.e., on the bound $C_{r, \tilde r} > 0$ such that $r_n, \tilde r_n < \frac{C_{r,\tilde r}}{n^2}$. This fact however becomes crucial in the proof of Proposition~\ref{QQ} \cite[Proposition~C.4]{CD}. To capture the dependence of the error term on the precise decay rate of $r_n$, $\tilde r_n$, we have chosen to express the error in terms of $r_n$ and $n$, though one could have used $C_{r,\tilde r}$ and $n$ instead. Luckily, the gap in the original formulation can be easily filled in as shown in Proposition~\ref{ProprrAp}.
\end{Remark}

\section{The Riemann--Hilbert problem for orthogonal polynomials}\label{RHProblems}

In the following section, we recall the celebrated Fokas--Its--Kitaev characterization of orthogonal polynomials via RH problems \cite{FIK}. We will state the problem explicitly in the case where the weight function is $w(x)$, $x\in[-1,1]$, but similar characterizations hold for the other weight functions $\hw$ and $\tw$.

\subsection{Fokas--Its--Kitaev RH problem for the logarithmic weight}
Find a $2\times2$ matrix-valued function $Y = Y^{(n)} \colon \C \setminus [-1,1] \to \C^{2\times2}$ satisfying the following properties:
\begin{enumerate}[(i)]\itemsep=0pt
 \item $Y(z)$ is analytic for $z \in \C \setminus [-1,1]$,
 \item $Y$ satisfies the jump condition
 \begin{align*}
 Y_+(s) = Y_-(s) \begin{pmatrix}
 1 & w(s)
 \\
 0 & 1
 \end{pmatrix}, \qquad s\in [-1,1],
 \end{align*}
\item $Y(z)\begin{pmatrix}
z^{-n} & 0
\\
0 & z^n
\end{pmatrix} = \Id + O\big(z^{-1}\big)$, as $z \rightarrow \infty$,
\item $Y$ is bounded away from the points $\pm 1$, and has the following behaviours near the points~$\pm1$:
\begin{align*}
 Y(z) = \begin{pmatrix}
 O(1) & O\big(\log^2|z-1|\big)
 \\
 O(1) & O\big(\log^2|z-1|\big)
 \end{pmatrix}, \qquad z\to +1,
\end{align*}
and
\begin{align*}
 Y(z) = \begin{pmatrix}
 O(1) & O(1)
 \\
 O(1) & O(1)
 \end{pmatrix}, \qquad z \to -1.
\end{align*}
\end{enumerate}
The condition (iv) for $z \to +1$ has been shown in \cite[Section~3.1]{CD} for the case of the weight function having a logarithmic singularity, while the behaviour for $z \to -1$ is a special case of the algebraic-type singularity of the weight function treated in \cite[Section~2]{KMVV}.

To understand the connection between the RH problem for $Y$ and orthogonal polynomials with respect to the orthogonality measure $w(x){\rm d}x$ on $[-1,1]$, we need to introduce the Cauchy operator $\mathcal C_{[-1,1]}$ on the interval $[-1,1]$:
\begin{align*}
 \mathcal C_{[-1,1]}\colon\ L^2([-1,1])&\to \mathcal O(\C \setminus [-1,1]),
 \\
 f(s) &\mapsto \mathcal C_{[-1,1]}(f)(z) = \frac{1}{2\pi\I} \int_{-1}^1 \frac{f(s)}{s-z} {\rm d}s.
\end{align*}
Here $\mathcal O(\C \setminus [-1,1])$ denotes the space of functions holomorphic in $\C \setminus [-1,1]$. When $\mathcal C_{[-1,1]}$ is applied to matrix-valued functions, it is understood to act componentwise. Cauchy operators can also be defined on different contours, as in Section~\ref{Section_LegendreResolvent}.

We can now state the seminal result by Fokas, Its and Kitaev which characterizes orthogonal polynomials via a RH problem:
\begin{Theorem}[\cite{FIK}]
The RH problem for $Y$ is solved uniquely by
\begin{align*}
Y(z) =
\begin{pmatrix}
\pi_n(z) & \mathcal{C}_{[-1,1]}(\pi_n w)(z)
\\
-2\pi\I\gamma^2_{n-1} \pi_{n-1}(z) & -2\pi\I\gamma^2_{n-1} \mathcal{C}_{[-1,1]}(\pi_{n-1} w)(z)
\end{pmatrix},
\end{align*}
where $\pi_n(z)$ is the $n$-th monic orthogonal polynomial with respect to the weight function $w(x)$ and $\gamma_n > 0$ is the leading coefficients of the orthonormal polynomial $p_n$, meaning $p_n = \gamma_n \pi_n$.
\end{Theorem}

The fact that the Fokas--Its--Kitaev RH problem characterizes the corresponding orthogonal polynomials has lead to numerous new results, in particular in the case where the associated weight function satisfies local analyticity properties (see, e.g.,~\cite{BI,PD,DKMVZ1,DKMVZ2,KMVV}), but also in the case of nonanalytic weights (see \cite{BY10,KMcL99,McM06,McM08,MY23}). Instrumental in the derivation of those results has been the nonlinear steepest descent method, first presented in \cite{DZ93} to study the long-time asymptotics of the mKdV equation and later generalized to the Fokas--Its--Kitaev RH problem in \cite{BI,DKMVZ1,DKMVZ2}. It turns out that the recurrence coefficients can also be simply expressed in term of~$Y$ (see \cite{PD}):
\begin{Proposition}\label{Yab}
 Let $Y_1^{(n)} \in \C^{2\times2}$ be given through the expansion
 \begin{align*}
 Y^{(n)}(z)z^{-n\sigma_3} = \Id + \frac{Y_1^{(n)}}{z} + O\left(\frac{1}{z^2}\right) \qquad \text{as} \quad z\to\infty.
 \end{align*}
Then the recurrence coefficients $a_n$ and $b_{n-1}$ can be extracted from $Y_1^{(n)}$ via the formulas
\begin{align*}
 a_n = \big(Y_1^{(n)}\big)_{11}-\big(Y_1^{(n+1)}\big)_{11}
\qquad \text{and}\qquad b_{n-1}^2 = \big(Y_1^{(n)}\big)_{12} \big(Y_1^{(n)}\big)_{21}.
\end{align*}
\end{Proposition}

The nonlinear steepest descent analysis for weight functions supported on a single finite interval has been performed in \cite{KMVV}, and in the following we shall repeat the conjugation and deformation steps found therein. In the first step one normalizes the RH problem at infinity through an appropriate conjugation, viz.,
\begin{align} \label{TY}
 T(z) := \left(\frac{2}{F_\infty}\right)^{n\sigma_3} Y(z) \left(\frac{F(z)}{\phi^n(z)}\right)^{\sigma_3}, \qquad z \in \C \setminus [-1,1].
\end{align}
Then $T$ turns out to be the unique solution of the following RH problem.

\subsection{Normalized Fokas--Its--Kitaev RH problem for the logarithmic weight}
Find a $2\times2$ matrix-valued function $T = T^{(n)} \colon \C \setminus [-1,1] \to \C^{2\times2}$ satisfying the following properties:
\begin{enumerate}[(i)]\itemsep=0pt
 \item $T(z)$ is analytic for $z \in \C \setminus [-1,1]$,
 \item $T$ satisfies the jump condition
 \begin{align}\label{Tjump}
 T_+(s) = T_-(s) \begin{pmatrix}
 \frac{F_+^2}{w}(s)\phi_+^{-2n}(s) & 1
 \\
 0 & \frac{F_-^2}{w}(s)\phi_-^{-2n}(s)
 \end{pmatrix}, \qquad s\in [-1,1],
 \end{align}
\item $T(z) = \Id + O\big(z^{-1}\big)$, as $z \rightarrow \infty$.
\item $T$ is bounded away from the points $\pm 1$, and has the following behaviours near the points~$\pm1$:
\begin{align*}
 T(z) = \begin{pmatrix}
 O\big(\log^{1/2}|z-1|\big) & O\big(\log^{3/2}|z-1|\big)
 \\
 O\big(\log^{1/2}|z-1|\big) & O\big(\log^{3/2}|z-1|\big)
 \end{pmatrix}, \qquad z\to +1
\end{align*}
and
\begin{align}\label{Tjump-1}
 T(z) = \begin{pmatrix}
 O\big(|z+1|^{1/2}\big) & O\big(|z+1|^{-1/2}\big)
 \\
 O\big(|z+1|^{1/2}\big) & O\big(|z+1|^{-1/2}\big)
 \end{pmatrix}, \qquad z \to -1.
\end{align}
\end{enumerate}
Note that we follow the convention found in \cite{CD} where the matrix $T$ is conjugated by the Szeg\H{o} function $F$, as in \eqref{TY}. Hence, the matrix $T$ found in \cite[equation~(3.1)]{KMVV} differs from ours in that respect. In our case the inclusion of the Szeg\H{o} function $F$ in the jump matrices \eqref{Tjump} will play a~crucial role in regularizing the jump matrices of the RH problems and thus enabling us to make the comparison argument in Section \ref{Comparison} effective.

However, as the weight function $w_k$ from \eqref{wk} is nonvanishing at $z = -1$, the matrix $T$ in~\eqref{TY} also differs crucially in its behaviour as $z\to-1$ from the one found in \cite[equation~(3.7)]{CD}. The reason is that our logarithmic weight function $w$ has a simple zero at $z=-1$, implying by item (iv) in Proposition~\ref{PropertiesF} that $|F(z)| = O\big(|z+1|^{1/2}\big)$ as $z \to -1$. This induces the $O\big(|z+1|^{-1/2}\big)$-behaviour in $T$ as $z\to-1$. Crucially, the entries of $T_\pm$ (and later $Q_\pm$) will not be square integrable, meaning that the $L^2$-theory used in \cite{CD} will not be applicable directly. We will circumvent this difficulty by defining certain ${}^\star$RH problems in Section~\ref{LocPar} through inverting locally by appropriate local parametrix solutions near the endpoint $z = -1$.

In the second step, we will use the following factorization of the jump matrix \eqref{Tjump}:
\begin{align*}
 \begin{pmatrix}
 \frac{F_+^2}{w}(s)\phi_+^{-2n}(s) & 1
 \\
 0 & \frac{F_-^2}{w}(s)\phi_-^{-2n}(s)
 \end{pmatrix} = \begin{pmatrix}
 1 & 0
 \\
 \frac{F_-^2}{w}(s)\phi_-^{-2n}(s) & 1
 \end{pmatrix}&
 \begin{pmatrix}
 0 & 1
 \\
 -1 & 0
 \end{pmatrix}
 \begin{pmatrix}
 1 & 0
 \\
 \frac{F_+^2}{w}(s)\phi_+^{-2n}(s) & 1
 \end{pmatrix}.
\end{align*}
Next, we introduce the oriented \emph{lens} jump contour $\Sigma = \Sigma_1 \cup \Sigma_2 \cup (-1,1+\delta)$, see Figure~\ref{lensJumpContour}, where~$\delta > 0$ is fixed and $n$-independent.

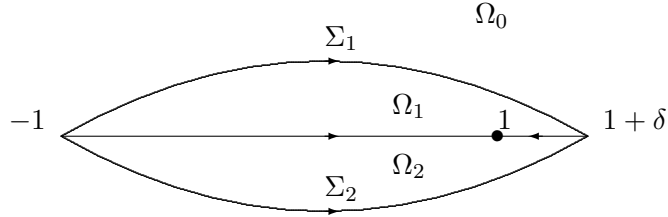
\begin{figure}[th]\centering
\vspace*{-8mm}
\begin{picture}(7,5.2)
\put(0,2.5){\line(1,0){7.0}}
\put(3.5,2.5){\vector(1,0){0.2}}

\put(3.5,3.5){\vector(1,0){0.2}}
\put(3.5,1.5){\vector(1,0){0.2}}
\put(6.4,2.5){\vector(-1,0){0.2}}

\curve(0,2.5, 3.5,1.5, 7,2.5)
\curve(0,2.5, 3.5,3.5, 7,2.5)

\put(3.5, 3.7){$\Sigma_1$}
\put(3.5, 1.7){$\Sigma_2$}

\put(5.5, 4){$\Omega_0$}

\put(4.4, 2.8){$\Omega_1$}

\put(4.4, 2){$\Omega_2$}

\put(7.2, 2.6){$1+\delta$}
\put(-0.7, 2.6){$-1$}
\put(5.8,2.6){$1$}
\put(5.7,2.41){$\bullet$}
\end{picture}
\vspace{-16mm}
\caption{Lens-shaped jump contour $\Sigma$.}\label{lensJumpContour}
\end{figure}

Note that the matrices
\begin{align*}
 \begin{pmatrix}
 1 & 0
 \\
 \frac{F_+^2}{w}(s)\phi_+^{-2n}(s) & 1
 \end{pmatrix}, \qquad \begin{pmatrix}
 1 & 0
 \\
 \frac{F_-^2}{w}(s)\phi_-^{-2n}(s) & 1
 \end{pmatrix}, \qquad s \in (-1,1)
\end{align*}
can be analytically continued to $z \in \Omega_1$ and $z\in \Omega_2$, respectively. Hence, we can define the following matrix-valued function $\bQ \colon \C \setminus \Sigma \to \C^{2\times2}$:
\begin{align}\label{TQRel}
 \bQ(z) =
\begin{cases}
T(z), & z \in \Omega_0,
\\
T(z)
\begin{pmatrix}
1 & 0
\\
-\frac{F^2}{w}(z) \phi^{-2n}(z) & 1
\end{pmatrix},
& z \in \Omega_1,\vspace{1mm}
\\
T(z)
\begin{pmatrix}
1 & 0
\\
\frac{F^2}{w}(z) \phi^{-2n}(z) & 1
\end{pmatrix}, & z \in \Omega_2.
\end{cases}
\end{align}
Then $\bQ$ will be the solution to the following RH problem.

\subsection{Logarithmic RH problem}\label{logRHP} Find a $2\times2$ matrix-valued function $\bQ \colon \C \setminus \Sigma \to \C^{2\times2}$ satisfying the following properties:
\begin{enumerate}[(i)]\itemsep=0pt
 \item $\bQ(z)$ is analytic for $z \in \C \setminus \Sigma$,
 \item $\bQ$ satisfies the jump condition
 \begin{align*}
 \bQ_+(s) = \bQ_-(s) \bv(s), \qquad s\in \Sigma,
 \end{align*}
 where
 \begin{align} \label{jumpMatrixV}
\bv(s) =
\begin{cases}
\begin{pmatrix}
1 & 0
\\
\frac{F^2}{w}(s) \phi^{-2n}(s) & 1
\end{pmatrix}, & \mbox{for} \ s \in \Sigma_1 \cup \Sigma_2,\vspace{1mm}
\\
\begin{pmatrix}
1 & 0
\\
\big(\frac{F^2}{w_+}(s)+\frac{F^2}{w_-}(s)\big) \phi^{-2n}(s) & 1
\end{pmatrix}, & \mbox{for}\  s\in(1,1+\delta),\vspace{1mm}
\\
\begin{pmatrix}
0 & 1
\\
-1 & 0
\end{pmatrix}, & \mbox{for}\ s \in (-1,1),
\end{cases}
\end{align}
\item $\bQ(z) = \Id + O\big(z^{-1}\big)$, as $z \rightarrow \infty$,
\item $\bQ$ is bounded away from the points $\pm 1$, and has the following behaviours near the points~$\pm1$:
\begin{align*}
 Q(z) = \begin{pmatrix}
 O\big(\log^{3/2}|z-1|\big) & O\big(\log^{3/2}|z-1|\big)
 \\
 O\big(\log^{3/2}|z-1|\big) & O\big(\log^{3/2}|z-1|\big)
 \end{pmatrix}, \qquad z\to +1
\end{align*}
and
\begin{align*}
 Q(z) = \begin{pmatrix}
 O\big(|z+1|^{-1/2}\big) & O\big(|z+1|^{-1/2}\big)
 \\
 O\big(|z+1|^{-1/2}\big) & O\big(|z+1|^{-1/2}\big)
 \end{pmatrix}, \qquad z \to -1.
\end{align*}
\end{enumerate}
Note that $\bQ_\pm \not \in L^2(\Sigma)$ due to its behaviour as $z\to -1$, which is caused by the simple zero of $w$ at that point.

For the asymptotic analysis in Section~\ref{SectAsymptoticAnalysis}, we will need the analog of the logarithmic RH problem stated for the model weight function $\hw$. The derivation from the Fokas--Its--Kitaev formulation remains unchanged except for the use of the functions $\hw$ and $\hF$ instead of $w$ and~$F$. The behaviour near $z \to \pm1$ can be read off from \cite[Section~4]{KMVV}, after taking into account the behaviour of the Szeg\H{o} function $\hF$.

\subsection{Model RH problem}\label{ModRHP} Find a $2\times2$ matrix-valued function $\hQ \colon \C \setminus \Sigma \to \C^{2\times2}$ satisfying the following properties:
\begin{enumerate}[(i)]\itemsep=0pt
 \item $\hQ(z)$ is analytic for $z \in \C \setminus \Sigma$,
 \item $\hQ$ satisfies the jump condition
 \begin{align*}
 \hQ_+(s) = \hQ_-(s) \hv(s), \qquad s\in \Sigma,
 \end{align*}
 where
 \begin{align}\label{jumpMatrixhV}
\hv(s) =
\begin{cases}
\begin{pmatrix}
1 & 0
\\
\frac{\hF^2}{\hw}(s) \phi^{-2n}(s) & 1
\end{pmatrix}, & \mbox{for}\ s \in \Sigma_1 \cup \Sigma_2,\vspace{1mm}
\\
\begin{pmatrix}
1 & 0
\\
2\frac{\hF^2}{\hw}(s) \phi^{-2n}(s) & 1
\end{pmatrix}, & \mbox{for}\ s\in(1,1+\delta),\vspace{1mm}
\\
\begin{pmatrix}
0 & 1
\\
-1 & 0
\end{pmatrix}, &\mbox{for}\ s \in (-1,1),
\end{cases}
\end{align}
\item $\hQ(z) = \Id + O\big(z^{-1}\big)$, as $z \rightarrow \infty$,
\item $\hQ$ is bounded away from the points $\pm 1$, and has the following behaviours near the points~$\pm1$:
\begin{align*}
 \hQ(z) = \begin{pmatrix}
 O(\log|z-1|) & O(\log|z-1|)
 \\
 O(\log|z-1|) & O(\log|z-1|)
 \end{pmatrix}, \qquad z\to +1
\end{align*}
and
\begin{align*}
 \hQ(z) = \begin{pmatrix}
 O\big(|z+1|^{-1/2}\big) & O\big(|z+1|^{-1/2}\big)
 \\
 O\big(|z+1|^{-1/2}\big) & O\big(|z+1|^{-1/2}\big)
 \end{pmatrix}, \qquad z \to -1.
\end{align*}
\end{enumerate}

Note that the jump matrix $\hv$ simplifies compared to $\bv$, as the weight function $\hw$ is continuous (in fact analytic) across $(1, 1+\delta)$. As for $Q$, due to the simple zero of $\hw$ at $z = -1$, $\hQ$ will not be square integrable near that point.

Note that the weight function $\hw$ falls into the class of \emph{modified Jacobi weight functions} considered in \cite{KMVV}. As such, an asymptotic series expansion in powers of $n^{-1}$ for the recurrence coefficients $\widehat a_n$, $\widehat b_n$ can be explicitly computed. Note however that we use a different convention here than in \cite{KMVV}, i.e., the roles of $\widehat a_n$, $\widehat b_n$ are interchanged. We write down the expansion up to the $n^{-2}$-term:
\begin{Corollary}[{\cite[Theorem.~1.10]{KMVV}}]\label{AsymptoticsHatRC}
The recurrence coefficients $\widehat a_n$, $\widehat b_n$ associated to the weight function $\hw$ satisfy
\begin{align*}
 \widehat a_n= \frac{1}{4n^2} + O\left(\frac{1}{n^3}\right),\qquad \widehat b_n = \frac{1}{2}-\frac{1}{16n^2} + O\left(\frac{1}{n^3}\right).
\end{align*}
\end{Corollary}
To compute the asymptotics of the recurrence coefficients $a_n$, $b_n$, we will first compute the asymptotics of $a_n -\widehat a_n$, $b_n^2-\widehat b_n^2$ and then use Corollary~\ref{AsymptoticsHatRC}. Hence, we will need an analog of Proposition~\ref{Yab} above for the differences of recurrence coefficients, which we express in terms of~$Q$ and $\hQ$.
\begin{Proposition}[{\cite[Proposition~3.6]{CD}}]\label{PropRecQ}
 Let $Q_1^{(n)}$ and $\hQ_1^{(n)}$ be given through the expansion
 \begin{align*}
 Q^{(n)} = \Id + \frac{Q_1^{(n)}}{z} + O\left(\frac{1}{z^2}\right) \qquad \text{as} \quad z\to\infty,
 \end{align*}
 and
 \begin{align*}
 \hQ^{(n)} = \Id + \frac{\hQ_1^{(n)}}{z} + O\left(\frac{1}{z^2}\right) \qquad \text{as} \quad z\to\infty.
 \end{align*}
Then the differences $a_n - \widehat a_n$ and $b_n^2 -\widehat b_n^2$ can be expressed via
\begin{align}\label{a_n-widehat a_n}
 a_n - \widehat a_n = \big(Q_1^{(n)}\big)_{11}-\big(\hQ_1^{(n)}\big)_{11} - \big(\big(Q_1^{(n+1)}\big)_{11}-\big(\hQ_1^{(n+1)}\big)_{11}\big),
\end{align}
and
\begin{align}\nonumber
 b_{n-1}^2-\widehat b_{n-1}^2 ={}& \big(\big(\bQ_1^{(n)}\big)_{12}-\big(\hQ_1^{(n)}\big)_{12}\big)\big(\big(\bQ_1^{(n)}\big)_{21}-\big(\bQ_1^{(n+1)}\big)_{21}\big)
 \\\label{b_n^2-widehat b_n^2}
 &+ \big( \hQ_1^{(n)}\big)_{12}\big[\big(\big(\bQ_1^{(n)}\big)_{21}-\big(\bQ_1^{(n+1)}\big)_{21}\big)-\big(\big(\hQ_1^{(n)}\big)_{21}-\big(\hQ_1^{(n+1)}\big)_{21}\big)\big].
\end{align}
\end{Proposition}
The usefulness of Proposition~\ref{PropRecQ} comes from the fact that $\bQ^{(n)}_1- \hQ^{(n)}_1$ has a simple integral representation.
\begin{Proposition}[{\cite[Proposition~4.9]{CD}}]
The following formula holds:
\begin{align}\label{QIntForm}
 \bQ^{(n)}_1- \hQ^{(n)}_1 = -\frac{1}{2\pi\I}\int_\Sigma \bQ^{(n)}_-(s)\big(\bv^{(n)}(s)-\hv^{(n)}(s)\big)\big[\hQ^{(n)}_-(s)\big]^{-1} {\rm d}s.
\end{align}
\end{Proposition}
\begin{proof}
In the following, we drop the subscript $(n)$ for better readability, i.e., write $\bQ_1$ for $\bQ^{(n)}_1$ and so on. Let us define the matrix-valued function $X(z) = \bQ(z) \big[\hQ(z)\big]^{-1}$ for $z \in \C \setminus \Sigma$. Note that
\begin{align}\label{Rformula}
 X(z) = \Id + \frac{\bQ_1-\hQ_1}{z} + O\left(\frac{1}{z^2}\right) \qquad \text{as } z \to \infty.
\end{align}
Moreover, $X$ satisfies the jump condition
\begin{align*}
 X_+(s) = X_-(s)v_X(s), \qquad s\in \Sigma,
\end{align*}
where $v_X = \hQ_- \bv \hv^{-1} \hQ_-^{-1}$. We claim that $X_\pm \in L^1(\Sigma)$. To see this let us first introduce the analog of $T$ for the weight function $\hw$:
\begin{align*}
 \widehat T(z) = \left(\frac{2}{\hF_\infty}\right)^{n\sigma_3} \widehat Y(z) \left(\frac{\hF(z)}{\phi^n(z)}\right)^{\sigma_3}.
\end{align*}
Here $\hat Y$ is the solution to the Fokas--Its--Kitaev problem for the weight function $\hw$. Then for~${z \in \Omega_0}$ (cf.\ Figure~\ref{lensJumpContour}), we have
\begin{align*}
 X(z) = T(z)\big[\widehat T(z)\big]^{-1} = O(1), \qquad z\in \Omega_0, \quad z\to-1,
\end{align*}
where we have used \eqref{Tjump-1} and its analog (see \cite[Section~2]{KMVV})
\begin{align}\label{hTjump-1}
 \widehat T(z) = \begin{pmatrix}
 O\big(|z+1|^{1/2}\big) & O\big(|z+1|^{-1/2}\big)
 \\
 O\big(|z+1|^{1/2}\big) & O\big(|z+1|^{-1/2}\big)
 \end{pmatrix}, \qquad z \to -1.
\end{align}
Note that the analog of \eqref{TQRel} remains valid
\begin{align*}
 \hQ(z) =
\begin{cases}
\widehat T(z), & z \in \Omega_0,
\\
\widehat T(z)
\begin{pmatrix}
1 & 0
\\
-\frac{\hF^2}{\hw}(z) \phi^{-2n}(z) & 1
\end{pmatrix},
& z \in \Omega_1,\vspace{1mm}
\\
\widehat T(z)
\begin{pmatrix}
1 & 0
\\
\frac{\hF^2}{\hw}(z) \phi^{-2n}(z) & 1
\end{pmatrix},& z \in \Omega_2.
\end{cases}
\end{align*}
Thus for $z\in \Omega_1\cup \Omega_2$, we have
\begin{align*}
 &X(z) = T(z)
 \begin{pmatrix}
 1 & 0
 \\
 \mp\big(\frac{F^2}{w}(z)-\frac{\hF^2}{\hw}(z)\big) \phi^{-2n}(z) & 1
 \end{pmatrix}
 \big[\widehat T(z)\big]^{-1} = O(1),
 \\
 &z\in \Omega_1\cup\Omega_2, \qquad z\to-1,
\end{align*}
where the $+$, resp.\ $-$ sign refer to $\Omega_1$ and $\Omega_2$. Apart from \eqref{Tjump-1} and \eqref{hTjump-1}, we have used in addition \eqref{3/2decay}. As $X$ can have only logarithmic singularities for $z \to 1$ and otherwise the limits to the contour are analytic, it follows that $X_\pm \in L^1(\Sigma)$.

Next, we can use the Sokhotski--Plemelj theorem to conclude that $X$ can be represented as
\begin{align*}
 X(z) &= \Id + \frac{1}{2\pi \I}\int_\Sigma \frac{X_+(s)-X_-(s)}{s-z} {\rm d}s
 \\
 &= \Id + \frac{1}{2\pi \I}\int_\Sigma \frac{Q_-(s)\big(\bv(s)[\hv(s)]^{-1}-\Id\big)\big[\hQ_-(s)\big]^{-1}}{s-z} {\rm d}s
 \\
 &= \Id + \frac{1}{2\pi \I}\int_\Sigma \frac{Q_-(s)(\bv(s)-\hv(s))[\hQ_-(s)]^{-1}}{s-z} {\rm d}s,
\end{align*}
where we made use of the particular form of the jump matrices $\bv$ and $\hv$. This representation together with \eqref{Rformula} implies \eqref{QIntForm}.
\end{proof}

Finally, we recall the Legendre RH problem taken from \cite[Section~3.4]{CD}. While this RH problem does not approximate the logarithmic RH problem globally, it does so near the logarithmic singularity and additionally gives rise to a singular integral operator whose inverse is uniformly bounded as $n \to \infty$ (see Theorem~\ref{UnifLeg}). The existence of a RH problem with these properties will be crucial for the proof of Theorem~\ref{TheoremUniInv}. Note that the Szeg\H{o} function for the Legendre weight is just given by $\widetilde F \equiv 1$.

\subsection{Legendre RH problem}\label{LegRHP} Find a $2\times2$ matrix-valued function $\tQ \colon \C \setminus \Sigma \to \C^{2\times2}$ satisfying the following properties (see \cite[Section~3.4]{CD}):
\begin{enumerate}[(i)]\itemsep=0pt
 \item $\tQ(z)$ is analytic for $z \in \C \setminus \Sigma$,
 \item $\tQ$ satisfies the jump condition
 \begin{align*}
 \tQ_+(s) = \tQ_-(s) \tv(s), \qquad s\in \Sigma,
 \end{align*}
 where
 \begin{align*}
\tv(s) =
\begin{cases}
\begin{pmatrix}
1 & 0
\\
\phi^{-2n}(s) & 1
\end{pmatrix}, & \mbox{for}\ s \in \Sigma_1 \cup \Sigma_2,\vspace{1mm}
\\
\begin{pmatrix}
1 & 0
\\
2\phi^{-2n}(s) & 1
\end{pmatrix}, & \mbox{for}\ s\in(1,1+\delta),\vspace{1mm}
\\
\begin{pmatrix}
0 & 1
\\
-1 & 0
\end{pmatrix}, & \mbox{for}\ s \in (-1,1),
\end{cases}
\end{align*}
\item $\tQ(z) = \Id + O\big(z^{-1}\big)$ as $z \rightarrow \infty$,
\item $\tQ(z)$ is bounded away from the points $\pm 1$, and has the following behaviours near the points~$\pm1$:
\begin{align*}
 \tQ(z) = \begin{pmatrix}
 O(\log|z-1|) & O(\log|z-1|)
 \\
 O(\log|z-1|) & O(\log|z-1|)
 \end{pmatrix}, \qquad z\to +1
\end{align*}
and
\begin{align*}
 \tQ(z) = \begin{pmatrix}
 O(\log|z+1|) & O(\log|z+1|)
 \\
 O(\log|z+1|) & O(\log|z+1|)
 \end{pmatrix}, \qquad z \to -1.
\end{align*}
\end{enumerate}

The weight function $\tw(x) = 1$, $x\in(-1,1)$ lies in the class of weight functions considered in~\cite{KMVV}, and it follow from the calculations therein that the solution $\tQ$ can be globally approximated with arbitrary small errors. Moreover, unlike the logarithmic and model RH problems found in Sections~\ref{logRHP} and \ref{ModRHP}, the Legendre RH problem can be stated with a weaker $L^2$-condition instead of condition (iv) (cf.\ \cite[Proposition~3.2]{CD}), as follows.

\begin{Proposition}\label{LegUniq}
 The matrix-valued function $\tQ$ is the unique solution of the Legendre RH problem with the condition \emph{(iv)} being replaced by the condition $\tQ_\pm \in L^2([-1,1])$.
\end{Proposition}
\begin{proof}
Let $L$ be another solution of the Legendre RH problem, but with $L_\pm\in L^2([-1,1])$ instead of condition (iv). Then $\det L$ will be a holomorphic function in $\C \setminus [-1,1]$ with
$\det L_+ = \det L_-$ on $\Sigma$ and $\det L_\pm \in L^1(\Sigma)$. By Morera's theorem, it follows that $\det L$ is in fact an entire function with $\lim_{z\to\infty} \det L(z) = 1$. Hence by Liouville's theorem $\det L \equiv 1$ in $\C$.

We conclude that $L$ is invertible in $\C \setminus \Sigma$ and we can define $\tQ[L]^{-1}$. As with the determinant, $\tQ[L]^{-1}$ will have no jump across the contour~$\Sigma$: \smash{$\big(\tQ[L]^{-1}\big)_+=\big(\tQ[L]^{-1}\big)_-$}. Moreover, by the $L^2$-condition for $\tQ$ and $L$, we have that \smash{$\big(\tQ[L]^{-1}\big)_\pm \in L^1(\Sigma)$}. It follows that $\tQ[L]^{-1}$ can be extended to an entire function. By Liouville's theorem, we have that $\tQ[L]^{-1} \equiv \lim_{z\to\infty} \tQ(z)[L(z)]^{-1} = \Id$, hence $\tQ = L$.
\end{proof}

\section{An explicit formula for the Legendre resolvent}\label{Section_LegendreResolvent}
In this section, we will associate to the Legendre RH problem a singular integral operator. First, let us define the Cauchy operator on $\Sigma$ by
\begin{align*}
 \Cau \colon\ L^2(\Sigma) \to \mathcal O(\C \setminus \Sigma), \qquad f(s) \mapsto \Cau(f)(z) = \frac{1}{2\pi \I }\int_\Sigma \frac{f(s)}{s-z} {\rm d}s,
\end{align*}
where $\mathcal O(\C \setminus \Sigma)$ denotes the set of analytic functions on the open set $\C \setminus \Sigma$. For $f \in L^2(\Sigma)$, we further define the two Cauchy boundary operators by
\begin{align}\label{DefCau}
 \Cau^\pm(f)(s) = \lim_{z\to s\pm} \Cau(f)(z).
\end{align}
In our setting, the curve $\Sigma$ is clearly a Carleson curve and hence the limit in \eqref{DefCau} exists for almost all $s \in \Sigma$
and satisfies $\Cau^\pm(f) \in L^2(\Sigma)$, see \cite{BK} for more details. We define the two Cauchy boundary operators via
\begin{align*}
\Cau^\pm \colon\ L^2(\Sigma) \to L^2(\Sigma).
\end{align*}
These are bounded operators on $\Ltwo$, cf.\ Theorem \ref{Muckenhoupt} below. Note that $C^\pm_{\Sigma}$ satisfy the important identity $C^+_\Sigma - C^-_\Sigma = 1$.

More generally, the mapping in \eqref{DefCau} induces a bounded operator on certain weighted $L^p$-spaces. To be precise let $\Gamma$ be an oriented composed locally rectifiable curve (see \cite[Section~1]{BK}) and $p \in (1,\infty)$. Given a weight function $r \colon \Gamma \to \R$, $r \geq 0$, define the Banach space $L^p(\Gamma, r)$ of all measurable functions $f$ on $\Gamma$, such that the norm
\begin{align*}
 \Vert f \Vert_{L^p(\Gamma, r)} = \bigg( \int_\Gamma |f(s)|^p r(s)^p |{\rm d}s| \bigg)^{\frac{1}{p}}
\end{align*}
remains finite. Note that there is a $p$-th power of $r$ in the integral. We say that $r$ is a \emph{Muckenhoupt weight} if $r \in L^p_{\rm loc}(\Gamma)$, $1/r \in L^q_{\rm loc}(\Gamma)$ and
\begin{align*}
 \sup_{s\in \Gamma} \sup_{\rho > 0} \bigg(\frac{1}{\rho} \int_{\Gamma \cap D(s, \rho)} r(s')^p |{\rm d}s'| \bigg)^{\frac{1}{p}}
 \bigg( \frac{1}{\rho}\int_{\Gamma \cap D(s, \rho)} r(s')^{-q}|{\rm d}s'| \bigg)^{\frac{1}{q}} < \infty,
\end{align*}
where $D(s, \rho)$ is the open disc around $s$ of radius $\rho$ and $1/p + 1/q = 1$. For any $p \in (1, \infty)$ we denote the set of all Muckenhoupt weights by $A_p(\Gamma)$. The following results holds.

\begin{Theorem}\label{Muckenhoupt}
Let $p\in(1, \infty)$ and let $\Gamma$ be an oriented composed locally rectifiable curve. Assume $r \colon \Gamma \to \R$, $r \geq 0$ is a given weight. Then the mappings
\begin{align}\label{CauchyMapping}
 f \mapsto \lim_{z\to s\pm} \int_\Gamma \frac{f(s)}{s-z} {\rm d}s
\end{align}
define bounded operators from $L^p(\Gamma, r) \to L^p(\Gamma, r)$ if and only if $r$ is a Muckenhoupt weight, i.e., $r\in A_p(\Gamma)$.
\end{Theorem}
The proof can be found in \cite[Theorem 4.15]{BK}, for more material on this topic with emphasis on RH theory see \cite{JL}. We will abuse notation and denote the mapping \eqref{CauchyMapping} by $\mathcal C^\pm_\Gamma$ irrespectively of the choice of domain $L^p(\Gamma, r)$.

Next, let us define the operator
\begin{align*}
 \Cautv \colon \ L^2(\Sigma) \to L^2(\Sigma), \qquad f \mapsto \Cau^-(f (\tv-\Id)),
\end{align*}
and consider the following singular integral equation in $L^2(\Sigma)$
\begin{align}\label{LegEq}
 (1-\Cautv)\tmu = \Id.
\end{align}
Note that as the contour is bounded, $\Id$ is indeed an element of $L^2(\Sigma)$. Equation \eqref{LegEq} is, in fact, equivalent to the Legendre RH problem. More explicitly, any solution $\tmu$ will give rise to a~solution $L = \Id + \Cau(\tmu(\tv -\Id))$ of the Legendre RH problem with the condition $L_\pm \in L^2(\Sigma)$ instead of condition (iv), as can be verified by direct computation. By Proposition~\ref{LegUniq}, the solution to the Legendre RH problem (see Section~\ref{LegRHP}) exists and is unique, hence we must have $\tQ = L$ implying
\begin{align}\label{Qmurelation}
 \tQ = \Id + \Cau(\tmu(\tv -\Id)).
\end{align}
Moreover, from the Sokhotski--Plemelj formula
\begin{align*}
 \tQ = \Id + \Cau\big(\tQ_+-\tQ_-\big) = \Id + \Cau\big(\tQ_-(\tv-\Id)\big),
\end{align*}
it follows, after taking the minus limit to the contour $\Sigma$, that $\tmu := \tQ_-$ is indeed a solution to~\eqref{LegEq}. Moreover, any solution $\tmu$ of \eqref{LegEq} must be equal to $\tQ_-$ as can be seen from \eqref{Qmurelation} and
\begin{align*}
 \tQ_- = \Id + \Cau^-(\tmu(\tv-\Id)) = \Id +\Cautv \tmu = \tmu.
\end{align*}
Together these arguments imply that \eqref{LegEq} has a unique solution and hence $1-\Cautv$ must be injective. In \cite{CD}, it has been shown that $1-\Cautv$ is in fact uniformly invertible for $n \to \infty$, as described in the following result.

\begin{Theorem}[{\cite[Theorem 4.5]{CD}}]\label{UnifLeg}
The operator $1-\Cautv$ is invertible for all sufficiently large $n$. Moreover, the operator bound of $(1-\Cautv)^{-1}$ as an operator $\Ltwo \to \Ltwo$ remains uniformly bounded for $n \to \infty$.
\end{Theorem}
Theorem \ref{UnifLeg} played a central role in \cite{CD}. The uniform invertibility of the operator $1-\Cautv$ will also be crucial in the approach presented here and is the motivation for introducing the Legendre RH problem in addition to the logarithmic and model RH problems. However, in order to use Theorem~\ref{UnifLeg}, we will need to derive an explicit representation of the operator $(1-\Cautv)^{-1}$. To accomplish this, we recall the definition of an \emph{inhomogeneous RH problem of the first kind}, as introduced in \cite[Section~2.6]{DZSob}. This notion has been instrumental in the proof of Theorem \ref{UnifLeg} in~\cite{CD}. In the following, $h$ will denote a matrix-valued function on a contour $\Gamma$, with $h^{\pm 1} \in L^\infty(\Gamma)$.

\subsection*{Inhomogeneous RH problem of the first kind} For a given $g \in L^2(\Gamma)$, one seeks an $f \in L^2(\Gamma)$, such that $m_\pm = \Cau^\pm(f)+g$ satisfies the jump relation:
\begin{align}\label{HJump}
 m_+(s) &= m_-(s) h(s), \qquad s \in \Gamma.
\end{align}
A similar notion of an inhomogeneous RH problem of the second kind can be found in \cite[Section~2]{DZSob} but will not be needed here.

The importance of the above inhomogeneous RH problem comes from the following result proven in \cite[Proposition~2.6]{DZSob}.

\begin{Theorem}\label{TheoremInh}
 The mapping $1-\mathcal C_h \colon L^2(\Gamma) \to L^2(\Gamma)$ is invertible if and only if the corresponding inhomogeneous RH problem of the first kind is uniquely solvable for each $g \in L^2(\Gamma)$. Moreover, the inverse satisfies \smash{$\big\Vert (1-\mathcal C_h)^{-1}\big\Vert_{L^2(\Gamma)\to L^2(\Gamma)} \leq c$} if and only if $\Vert m_- \Vert_{L^2(\Gamma)} \leq c \Vert g \Vert_{L^2(\Gamma)}$ for all~${g\in L^2(\Gamma)}$ and the same constant $c > 0$.
 \end{Theorem}
 \begin{Remark}
 Note that by \eqref{HJump}, $m_+ = m_- h$ on $\Gamma$, hence $\Vert m_- \Vert_{L^2(\Gamma)} \leq c \Vert g \Vert_{L^2(\Gamma)}$ implies $\Vert m_\pm \Vert_{L^2(\Gamma)} \leq c'\Vert g \Vert_{L^2(\Gamma)}$ with $c'= c\Vert h \Vert_{L^\infty(\Gamma)}$.
 \end{Remark}
 As a corollary we obtain.
 \begin{Corollary}\label{CorHn}
 Given a sequence $h_n$ of matrix-valued functions with $h_n^{\pm 1} \in L^\infty(\Gamma)$, the operator $(1-\mathcal C_{h_n})^{-1}$ is uniformly bounded if and only if the corresponding inhomogeneous RH problems are uniquely solvable with $\Vert m_- \Vert_{L^2(\Gamma)} \leq c \Vert g \Vert_{L^2(\Gamma)}$, and $c > 0$ independent of $n \in \mathbb{N}$ and $g \in L^2(\Gamma)$.
\end{Corollary}
Note that Theorem \ref{UnifLeg} is proven in \cite[Section~4.2]{CD} via Corollary~\ref{CorHn}. We will use Theorems~\ref{UnifLeg} and~\ref{TheoremInh} to derive an explicit expression for the inverse $(1-\Cautv)^{-1}$ in Proposition~\ref{ExpForm} below. This allows us to identify locally the contribution of the logarithmic singularity to the uniform boundedness of $(1-\Cautv)^{-1}$, which is central to our approach as it enables us to prove Theorem~\ref{TheoremUniInv}.
\begin{Proposition}\label{ExpForm}
The inverse of $1-\Cautv$ has the explicit form
\begin{align}\label{ExpInv}
 (1-\Cautv)^{-1}\colon\ \Ltwo \to \Ltwo, \qquad g \mapsto g + \Cau^-\big(g (\tv-\Id) \tQ_+^{-1}\big)\tQ_-.
\end{align}

\end{Proposition}
\begin{proof}
From Theorem \ref{TheoremInh}, we know that for $g \in \Ltwo$, the (unique) solvability of the equation
\begin{align}\label{SingIE}
(1-\Cautv)\psi = g, \quad g\in \Ltwo,
\end{align}
in $\Ltwo$ is equivalent to the (unique) solvability of the following inhomogeneous RH problem.

\subsection*{Inhomogeneous Legendre RH problem}
For a given $g \in \Ltwo$, one seeks an $f \in \Ltwo$, such that $m_\pm = \Cau^\pm(f)+g$ satisfies the jump relation:
\begin{align*}
 m_+(s) &= m_-(s) \tv(s), \qquad s \in \Sigma.
\end{align*}

Note that Theorem \ref{UnifLeg} together with Corollary~\ref{CorHn} imply that there exists a constant $c$ independent of $n\in \mathbb{N}$ and $g\in \Ltwo$ such that the inhomogeneous Legendre RH problem has a unique solution $m_\pm$ with $\Vert m_- \Vert_{\Ltwo} \leq c\Vert g \Vert_{\Ltwo}$. We shall briefly recall the exact relation between \eqref{SingIE} and the inhomogeneous Legendre RH problem (cf.\ \cite[Section~2]{DZSob}), as it will be needed later in the proof. First, note that if we have a solution $m_\pm$ to the inhomogeneous Legendre RH problem, we must have $f = m_+ - m_- = m_-(\tv-\Id)$. If we now set $\psi := m_- = \Cau^-(f)+g \in \Ltwo$, it follows that
\begin{align*}
 (1-\Cautv)\psi &= \psi - \Cau^-(\psi (\tv-\Id))= \Cau^-(f) + g - \Cau^-(m_-(\tv-\Id)))= g.
\end{align*}
On the other hand, having a solution $\psi$ to the integral equation \eqref{SingIE}, we can define $m_\pm := \Cau^\pm(\psi (\tv-\Id)) + g$ and compute, using the Sokhotski--Plemelj (SP) formula,
\begin{align}\label{proofm}
 m_+ &= \Cau^+(\psi (\tv-\Id)) + g\overset{\text{SP}}{=} \psi (\tv-\Id) + \Cautv(\psi) + g= \psi \tv= m_- \tv,
\end{align}
as $m_- = \Cautv \psi + g = \psi$.
We can now derive an expression for the operator $(1 - \Cautv)^{-1}$. Assume $g \in \Ltwo$ is given and let $\psi \in \Ltwo$ be the unique solution of $(1-\Cautv)\psi = g$. Then $m_\pm = \Cau^\pm(f)+g$ with $f = \psi (\tv-\Id)$ solves the corresponding inhomogeneous RH problem, as we have seen in equation~\eqref{proofm}. We want to find an expression for $\psi = m_-$ in terms of $g$. To derive \eqref{ExpInv}, we start with
\begin{align}
& m_+ = m_- \tv, \qquad
 \Cau^+(f) + g = (\Cau^-(f) + g)\tilde v,\qquad \big(\Cau^+(f) + g\big)\tQ_+^{-1} = (\Cau^-(f) + g)\tQ_-^{-1},\nonumber
 \\
& \Cau^+(f)\tQ_+^{-1} - \Cau^-(f)\tQ_-^{-1} = g\big(\tQ_-^{-1}-\tQ_+^{-1}\big).\label{derInv}
 \end{align}
Observe that the left- and right-hand sides in the last line might not lie in $L^2(\Sigma)$. However, using property (iv) of $\tQ$ from the RH problem (see Section~\ref{LegRHP}), we see that they lie in $L^{2-\epsilon}(\Sigma)$ for any $\epsilon \in (0,1)$. Hence, if we define $H = \Cau(f) \tQ^{-1}$, we see that $H$ is analytic in $\C \setminus \Sigma$, satisfies $H_\pm \in L^{2-\epsilon}(\Sigma)$ and vanishes at $\infty$. By the Sokhotski--Plemelj formula, it follows that $H = \Cau(H_+-H_-)$. Thus, applying $\Cau$, which is understood to act on the space $L^{2-\epsilon}(\Sigma)$, in the last line of \eqref{derInv}, we obtain
 \begin{align*}
& \Cau(f) \tQ^{-1} = \Cau\big(\underbrace{g\big(\tQ_-^{-1}-\tQ_+^{-1}\big)}_{g(\tilde v-\Id) \tQ_+^{-1}}\big),\qquad
 \Cau(f) = \Cau\big(g (\tv-\Id) \tQ_+^{-1}\big)\tQ,
 \\
& \Cau^-(f) + g = g + \Cau^{-}\big(g (\tv-\Id) \tQ_+^{-1}\big)\tQ_-,\qquad
 \psi = g + \Cau^-\big(g (\tv-\Id) \tQ_+^{-1}\big)\tQ_-.
\end{align*}
Note here that $\Cau^-\big(g (\tv-\Id) \tQ_+^{-1}\big)\tQ_-$ is a priori a function in $L^{2-\epsilon}(\Sigma)$, however as $\psi = (1-\Cautv)^{-1}g \in L^2(\Sigma)$ by Theorem \ref{UnifLeg}, we conclude that indeed $\Cau^-\big(g (\tv-\Id) \tQ_+^{-1}\big)\tQ_- = \psi - g \in \Ltwo$.

We have thus proved that $(1-\Cautv)^{-1}$ must indeed have the formed stated in \eqref{ExpInv}.
\end{proof}

\section[Local parametrices around the point z = -1]{Local parametrices around the point $\boldsymbol{z = -1}$}\label{LocPar}
In the following, we will use appropriate local parametrices $\bP$, $\hP$ and $\tP$ to invert the three RH problems introduced in Section~\ref{RHProblems}, locally near $z = -1$. We denote the modified RH problems with ${}^\star$RH.

The explicit construction of the local parametrices is taken from \cite[equation~(6.50)]{KMVV} and can be given in terms of Bessel functions and the Szeg\H{o} functions corresponding to the three weights. To define these parametrices, we first need a local $n$-dependent change of variables~$z\to\zeta$. Following \cite[Section~6]{KMVV}, we choose a sufficiently small neighbourhood $U$ of the point $z = -1$ and define the mapping
\begin{align}\label{defXi}
 \xi\colon\ U \to \C, \qquad z \mapsto \xi(z) = \frac{\log^2(-\phi(z))}{4}.
\end{align}
Note that for $s\in U\cap[-1,1]$, $
\log^2(-\phi_+(s)) = \log^2\bigl(-\phi^{-1}_-(s)\bigr) = \log^2(-\phi_-(s))$ implying that $\xi$ is indeed well-defined and holomorphic.

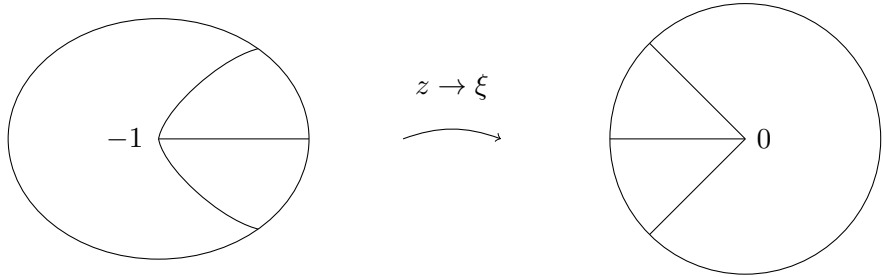
\begin{figure}[th]
 \centering\begin{tikzpicture}
 \draw (-4.2,0) circle [x radius=2cm, y radius=1.6cm];
 \draw [black] plot [smooth, tension=1.3] coordinates {(-4.2,0) (-3.7,0.7) (-2.875,1.2)};
 \draw [black] plot [smooth, tension=1.3] coordinates {(-4.2,0) (-3.7,-0.7) (-2.875,-1.2)};
 \draw [black] (-4.2,0) -- (-2.2,0);
 \node at (-4.65,0) {$-1$};

 \node at (-0.3,0.7) {$z \to \xi$};
 \draw [->] (-0.95,0) to [out=20,in=160] (0.35,0);

 \draw (3.6,0) circle [x radius=1.8cm, y radius=1.8cm];
 \draw [black] (3.6,0) -- (2.33,1.27);
 \draw [black] (3.6,0) -- (2.33,-1.27);
 \draw [black] (3.6,0) -- (1.8,0);
 \node at (3.85,0) {$0$};

 \end{tikzpicture}
 \caption{The change of variables $z \to \xi$.}\label{zToXi}
\end{figure}

Now using \eqref{phiEndpoints}, we see that $\xi'(-1) = -1/2$, meaning that for $U$ sufficiently small, $\xi$ will define a biholomorphic mapping between $U$ and its image~$\xi(U)$. Introduce now $\zeta = \zeta^{(n)}(z) = n^2 \xi(z)$ for $z \in U$ together with \smash{$\Sigma_\Psi^{(n)} = n^2 \xi(U\cap\Sigma)$}. We can assume that $\Sigma$ has been chosen such that~\smash{$\Sigma_\Psi^{(n)}$} can be extended to \smash{$\Sigma_\Psi\supset \Sigma_\Psi^{(n)}$} consisting of three straight line segments $\gamma_i$, $i = 1,2,3$, originating from $\zeta = 0$ at the angles $\pm \frac{2\pi}{3}$ and $\pi$, see Figure~\ref{besselcontour}. Accordingly, we will regard $\zeta$ as a~variable in the whole complex plane.

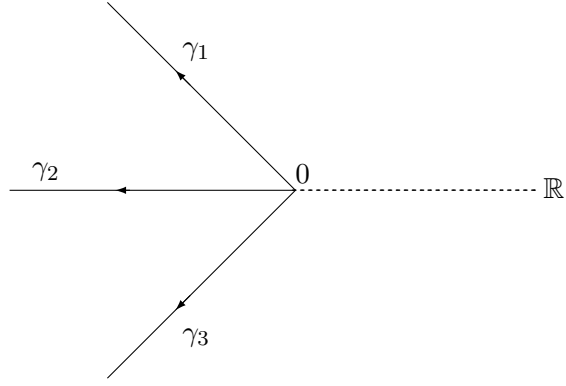
\begin{figure}[th]
\centering
\begin{picture}(7,5.2)
\put(-0.3,2.5){\line(1,0){3.8}}
\put(1.3,2.5){\vector(-1,0){0.2}}

\put(1,0){\line(1,1){2.5}}
\put(2.1,1.1){\vector(-1,-1){0.2}}

\put(1,5){\line(1,-1){2.5}}
\put(2.1,3.9){\vector(-1,1){0.2}}

\put(2,4.3){$\gamma_1$}
\put(0,2.7){$\gamma_2$}
\put(2,0.5){$\gamma_3$}

\put(6.8,2.4){$\R$}
\put(3.5,2.6){$0$}

\curvedashes{0.05,0.05}
\curve(3.4,2.5, 6.7,2.5)
\end{picture}
\caption{Contour for the local parametrix problems in the $\zeta$-plane.}\label{besselcontour}
\end{figure}

Next, we shall define two piecewise holomorphic functions $\Psi_\nu\colon \C \setminus \Sigma_\Psi \to \C^{2\times2}$ for $\nu = 0, 1$ (see \cite[equation~(6.51)]{KMVV}):
\begin{align*}
 \Psi_\nu(\zeta) =
 \begin{cases}
 \begin{pmatrix}
 I_\nu\big(2\zeta^{1/2}\big) & -\frac{\I}{\pi} K_\nu\big(2\zeta^{1/2}\big)
 \\
 -2\pi \I \zeta^{1/2} I_\nu'\big(2\zeta^{1/2}\big) & -2 \zeta^{1/2} K_\nu'\big(2\zeta^{1/2}\big)
 \end{pmatrix}, \qquad |\arg \zeta | < \frac{2\pi}{3},\vspace{1mm}
 \\
 \begin{pmatrix}
 \frac{1}{2}H^{(1)}_\nu\big(2(-\zeta)^{1/2}\big) & -\frac{1}{2}H^{(2)}_\nu\big(2(-\zeta)^{1/2}\big)
 \\
 -\pi\zeta^{1/2}(H^{(1)}_\nu)'\big(2(-\zeta)^{1/2}\big) & \pi\zeta^{1/2}(H^{(2)}_\nu)'\big(2(-\zeta)^{1/2}\big)
 \end{pmatrix}\E^{\frac{1}{2}\pi\I\nu \sigma_3},
 \\
 \qquad\frac{2\pi}{3} < \arg \zeta < \pi,\vspace{1mm}
 \\
 \begin{pmatrix}
 \frac{1}{2}H^{(2)}_\nu\big(2(-\zeta)^{1/2}\big) & \frac{1}{2}H^{(1)}_\nu\big(2(-\zeta)^{1/2}\big)
 \\
 \pi\zeta^{1/2}(H^{(2)}_\nu)'\big(2(-\zeta)^{1/2}\big) & \pi\zeta^{1/2}(H^{(1)}_\nu)'\big(2(-\zeta)^{1/2}\big)
 \end{pmatrix}\E^{-\frac{1}{2} \pi\I\nu \sigma_3},
 \\
 \qquad-\pi < \arg \zeta < -\tfrac{2\pi}{3}.
 \end{cases}
\end{align*}
Here the functions $I_\nu$, $K_\nu$ with $\nu \in \C$ are the familiar \emph{modified Bessel functions}. Generally, these are holomorphic functions in the domain $z\in \C\setminus (-\infty, 0]$ and have a branch cut along the negative real axis. In the special case $\nu \in \Z$, $I_\nu$ is entire.

Analogously, the functions $H^{(1)}_\nu$, $H^{(2)}_\nu$ with $\nu \in \C$ are holomorphic for $z\in\C \setminus (-\infty, 0]$ and have a branch cut on the negative real axis. They are the \emph{Bessel functions of the third kind}, also known as the \emph{Hankel functions}. Properties of these function can be found in \cite[Section~10]{DLMF}. In the following, we will be interested in the behaviour of $\Psi_\nu$ as $\zeta \to 0$, which can be deduced from the following lemma:
\begin{Lemma}[{\cite[Section~10]{DLMF}}]
The following asymptotic formulas hold uniformly for any path $\zeta \to 0$:
\begin{align}
& I_0\big(2\zeta^{1/2}\big), K_0\big(2\zeta^{1/2}\big), H^{(1)}_0\big(2(-\zeta)^{1/2}\big), H^{(2)}_0\big(2(-\zeta)^{1/2}\big) = O(\log |\zeta|),\nonumber
 \\
 &I_0'\big(2\zeta^{1/2}\big), K_0'\big(2\zeta^{1/2}\big), \big(H^{(1)}_0\big)'\big(2(-\zeta)^{1/2}\big), \big(H^{(2)}_0\big)'\big(2(-\zeta)^{1/2}\big) = O\left(\frac{1}{|\zeta|^{1/2}}\right),\label{BesselAsym0}
\end{align}
 and
\begin{align}
 &I_1\big(2\zeta^{1/2}\big), K_1\big(2\zeta^{1/2}\big), H^{(1)}_1\big(2(-\zeta)^{1/2}\big), H^{(2)}_1\big(2(-\zeta)^{1/2}\big) = O\left(\frac{1}{|\zeta|^{1/2}}\right),\nonumber
 \\
 &I_1'\big(2\zeta^{1/2}\big), K_1'\big(2\zeta^{1/2}\big), \big(H^{(1)}_1\big)'\big(2(-\zeta)^{1/2}\big), \big(H^{(2)}_1\big)'\big(2(-\zeta)^{1/2}\big) = O\left(\frac{1}{|\zeta|}\right).\label{BesselAsym1}
\end{align}
The following asymptotic formulas hold uniformly for $\zeta \to \infty$ in the prescribed sectors for any~$\delta >0$ and $\nu = 0,1$:
\begin{align}
& I_\nu\big(2\zeta^{1/2}\big)\E^{-2\zeta^{1/2}}, I_\nu'\big(\zeta^{1/2}\big)\E^{-2\zeta^{1/2}} = O\left(\frac{1}{|\zeta|^{1/4}} \right),\nonumber
 \\
 &K_\nu\big(2\zeta^{1/2}\big)\E^{2\zeta^{1/2}}, K_\nu'\big(2\zeta^{1/2}\big)\E^{2\zeta^{1/2}} = O\left(\frac{1}{|\zeta|^{1/4}} \right),
 \qquad|\arg \zeta|<\pi-\delta,\nonumber
 \\
 &H^{(1)}_\nu\big(2(-\zeta)^{1/2}\big)\E^{-2\I(-\zeta)^{1/2}}, (H^{(1)}_\nu)'\big(2(-\zeta)^{1/2}\big)\E^{-2\I(-\zeta)^{1/2}} = O\left(\frac{1}{|\zeta|^{1/4}} \right), \qquad \arg \zeta \not = 0,\nonumber
 \\
 &H^{(2)}_\nu\big(2(-\zeta)^{1/2}\big)\E^{2\I(-\zeta)^{1/2}}, (H^{(2)}_\nu)'\big(2(-\zeta)^{1/2}\big)\E^{2\I(-\zeta)^{1/2}} = O\left(\frac{1}{|\zeta|^{1/4}} \right), \qquad \arg \zeta \not = 0.\label{BesselInfinity}
\end{align}
In the formulas \eqref{BesselAsym0}--\eqref{BesselInfinity}, $(\cdot)^{1/2}$ denotes the principal branch.
\end{Lemma}
In \cite[Section~6]{KMVV}, it is shown that $\Psi_\nu$ satisfies the following jump conditions across the contours~$\gamma_i$:
\begin{align*}
 \Psi_{\nu,+}(\zeta) = \Psi_{\nu,-}(\zeta)
 \begin{cases}
 \begin{pmatrix}
 1 & 0
 \\
 \E^{\nu\pi\I} & 1
 \end{pmatrix}, &\zeta \in \gamma_1,\vspace{1mm}
 \\
 \begin{pmatrix}
 0 & 1
 \\
 -1 & 0
 \end{pmatrix}, &\zeta \in \gamma_2,\vspace{1mm}
 \\
 \begin{pmatrix}
 1 & 0
 \\
 \E^{-\nu\pi\I} & 1
 \end{pmatrix}, &\zeta \in \gamma_3.
 \end{cases}
\end{align*}
In the following, we will use $\Psi_\nu$, $\nu = 0,1$ to write down three local parametrices $P$, $\hP$, $\tP$ around~$z = -1$ for the three RH problems defined in Section \ref{RHProblems} (see \cite[equation~(6.52)]{KMVV}):
\begin{align*}
& \bP(z)= E(z)(2\pi n)^{\sigma_3/2} \Psi_1\big(n^2 \xi(z)\big) [-\phi(z)]^{-n\sigma_3}\left(\frac{F(z)}{W(z)}\right)^{\sigma_3},
 \\
& \hP(z)= \widehat E(z)(2\pi n)^{\sigma_3/2} \Psi_1\big(n^2 \xi(z)\big) [-\phi(z)]^{-n\sigma_3}\left(\frac{\hF(z)}{\widehat W(z)}\right)^{\sigma_3}, \qquad z \in U \setminus \Sigma,
 \\
& \tP(z)= \widetilde E(z)(2\pi n)^{\sigma_3/2} \Psi_0\big(n^2 \xi(z)\big) [-\phi(z)]^{-n\sigma_3}.
\end{align*}
Here $W = \sqrt{-w}$, $\widehat W = \sqrt{-\hw}$ are chosen to have a branch cut on $(-1,\infty)\cap U$ and to be positive on $(-\infty, -1)\cap U$. The matrix-valued functions $E$, $\widehat E$ and $\widetilde E$ are in fact holomorphic for $z$ in $U$. More explicitly, we have (see \cite[equation~(6.53)]{KMVV})
\begin{align}\label{formulaE}
 E(z) = N(z)\left(\frac{W(z)}{F(z)}\right)^{\sigma_3}\frac{1}{\sqrt{2}}
 \begin{pmatrix}
 1 & \I
 \\
 \I & 1
 \end{pmatrix}\xi(z)^{\sigma_3/4}
\end{align}
which is in fact holomorphic for $z \in U$.
 Here $N$ is the outer parametrix solution
 \begin{align}\label{OuterParametrix}
 N(z) = \begin{pmatrix}
\dfrac{a(z)+a(z)^{-1}}{2} && \dfrac{a(z)- a(z)^{-1}}{2\I}
\\
\dfrac{a(z)- a(z)^{-1}}{-2\I} && \dfrac{a(z)+a(z)^{-1}}{2}
\end{pmatrix},
\end{align}
where
\begin{align*}
a(z) = \left( \dfrac{z-1}{z+1} \right)^{1/4}
\end{align*}
with a branch cut on $(-1,1)$ and $a(\infty) = 1$. Similar formulae can be obtained for $\widehat E$ and $\widetilde E$ by substituting in \eqref{formulaE} $\hF$, $\widehat W$ and $\widetilde F \equiv 1$, $\widetilde W \equiv 1$, respectively. Crucially however, the outer parametrix $N$ is the same for all three problems. One can check that the determinants of all three parametrices are constant equal to $1$ inside $U$, cf.\ \cite[Section~7]{KMVV}. Furthermore, $E$, $\widehat E$ and~$\widetilde E$ are analytic and bounded in $U$, the singularity of $a(z)$ at $z =-1$ being compensated by the factor $\xi(z)^{\sigma_3/4}$.
\begin{Lemma}\label{PropP}
The matrix-valued functions $\bP$, $\hP$ and $\tP$ defined in $U \setminus \Sigma$, satisfy the following conditions:
\begin{enumerate}[$(i)$]\itemsep=0pt
 \item For $s \in U\cap \Sigma$,
 \begin{align} \label{ParametrixJump}
 \bP_+(s) = \bP_-(s) \bv(s), \qquad
 \hP_+(s) = \hP_-(s) \hv(s), \qquad
 \tP_+(s) = \tP_-(s) \tv(s).
 \end{align}
 \item For $s \in \partial U$ $($see \emph{\cite[equation~(6.41)]{KMVV}}$)$,
 \begin{align}\label{ParaMatchingCond}
 \bP(s), \hP(s), \tP(s) = N(s) + O\big(n^{-1}\big).
 \end{align}
\item For $z \in U$, we have uniformly $($see \emph{\cite[equation~(7.10)]{KMVV}}$)$
\begin{align} \label{ParaConv}
 \tQ(z) \big[\tP(z)\big]^{-1} = \Id + O\big(n^{-1}\big), \qquad \hQ(z) \big[\hP(z)\big]^{-1} = \Id + O\big(n^{-1}\big).
\end{align}
\item For $z \in U$,
\begin{align}\label{ParaBehaviour}
 \bP(z), \hP(z) = O\big(\max\big\lbrace |z+1|^{-1/4}, n^{-1/2} |z+1|^{-1/2}\big\rbrace\big),
\end{align}
and
\begin{align}\label{LegendreBehaviour}
 \tP(z) = O\big(|z+1|^{-1/4}\big)
\end{align}
uniformly as $n \to \infty$.
\end{enumerate}
\end{Lemma}
\begin{proof}
A detailed derivation of the local parametrices can be found in \cite[Section~6]{KMVV} together with a proof of properties (i), (ii), for property (iii) see \cite[Section~7]{KMVV}. Note that while the weight function $w$ does not fall into the class of weight functions considered in \cite{KMVV} due to the logarithmic singularity at $z = +1$, the local construction and estimation of the left parametrix near $z = -1$ found therein remains unchanged.

Regarding point (iv), we start with the properties of $\bP$ and $\hP$. Noting that $E(z)$ and~$\widehat E(z)$ are holomorphic, $n$-independent and have unit determinants, hence it is enough to consider $E^{-1}(z)\bP(z)$ and $\widehat E^{-1}(z)\hP(z)$ instead of $\bP(z)$ and $\hP(z)$ in \eqref{ParaBehaviour}. Analogously, it follows from~\eqref{F^2/w at -1} and \eqref{hat F^2/w at -1} that \smash{$\big(\frac{F}{W}\big)^{\sigma_3}$} and \smash{$\big(\frac{\hF}{\widehat W}\big)^{\sigma_3}$} are $n$-independent and bounded in a neighbourhood of $z = -1$, hence they also do not contribute in \eqref{ParaBehaviour}.

It remains to study
\begin{align*}
 (2\pi n)^{\sigma_3/2}\Psi_1\big(n^2 \xi(z)\big)[-\phi(z)]^{-n\sigma_3}
\end{align*}
which is equal to both $E^{-1}(z)\bP(z)\big(\frac{F}{W}\big)^{-\sigma_3}$ and $\widehat E^{-1}(z)\hP(z)\big(\frac{\hF}{\widehat W}\big)^{-\sigma_3}$. It follows from the definition of $\xi$ in \eqref{defXi} that \smash{$[-\phi(z)]^{-n\sigma_3} =\E^{-2\sqrt{n^2 \xi(z)}\sigma_3}$} where the square root has a branch cut along~$z > -1$. Writing $\zeta=n^2 \xi(z)$ and assuming $|z+1| \gtrsim O\big(n^{-2}\big)$, we see that $|\zeta| \gtrsim O(1)$ and using the estimates in \eqref{BesselInfinity} we conclude that
\begin{align}\nonumber
& (2\pi n)^{\sigma_3/2}\Psi_1\big(n^2 \xi(z)\big)[-\phi(z)]^{-n\sigma_3} \\
&\quad= (2\pi n)^{\sigma_3/2}\begin{pmatrix}
 O\big(|\zeta|^{-1/4}\big) & O\big(|\zeta|^{-1/4}\big)
 \\
 O\big(|\zeta|^{1/4}\big) & O\big(|\zeta|^{1/4}\big)
 \end{pmatrix}\nonumber
 \\
 &\quad=
 \begin{pmatrix}
 O\big(|z+1|^{-1/4}\big) & O\big(|z+1|^{-1/4}\big)
 \\
 O\big(|z+1|^{1/4}\big) & O\big(|z+1|^{1/4}\big)
 \end{pmatrix}, \qquad |z+1| \gtrsim O\big(n^{-2}\big).\label{geqn2}
\end{align}
For $|z+1| \lesssim O\big(n^{-2}\big)$, we use the estimates \eqref{BesselAsym1} instead to conclude
\begin{align}
 &(2\pi n)^{\sigma_3/2}\Psi_1\big(n^2 \xi(z)\big)[-\phi(z)]^{-n\sigma_3} \nonumber\\
 &\quad= (2\pi n)^{\sigma_3/2}\begin{pmatrix}
 O\big(|\zeta|^{-1/2}\big) & O\big(|\zeta|^{-1/2}\big)
 \\
 O\big(|\zeta|^{-1/2}\big) & O\big(|\zeta|^{-1/2}\big)
 \end{pmatrix}\nonumber\\
&\quad= \begin{pmatrix}
 O\big(n^{-1/2}|z+1|^{-1/2}\big) & O\big(n^{-1/2}|z+1|^{-1/2}\big)
 \\
 O\big(n^{-3/2}|z+1|^{-1/2}\big) & O\big(n^{-3/2}|z+1|^{-1/2}\big)
 \end{pmatrix}, \qquad z \in |z+1| \lesssim O\big(n^{-2}\big).\label{seqn2}
\end{align}
Note that in this case $[-\phi(z)]^{-n\sigma_3} = O(1)$, hence this term does not contribute. One checks that indeed for $|z+1| \sim n^{-2}$ the bounds in \eqref{geqn2} and \eqref{seqn2} are of the same order.

The proof of \eqref{LegendreBehaviour} works in a similar fashion. Again the holomorphic prefactor $\widetilde E$ can be ignored. For $|z+1|\gtrsim O\big(n^{-2}\big)$, we get as before
\begin{align*}
& (2\pi n)^{\sigma_3/2}\Psi_0\big(n^2 \xi(z)\big)[-\phi(z)]^{-n\sigma_3}
 \\
 &\quad=\begin{pmatrix}
 O\big(|z+1|^{-1/4}\big) & O\big(|z+1|^{-1/4}\big)
 \\
 O\big((|z+1|^{1/4}\big) & O\big((|z+1|^{1/4}\big)
 \end{pmatrix}, \qquad |z+1| \gtrsim O\big(n^{-2}\big).
\end{align*}
However, for $|z+1|\lesssim O\big(n^{-2}\big)$, we get different asymptotics after applying \eqref{BesselAsym0}. We obtain%
\begin{align} \nonumber
 &(2\pi n)^{\sigma_3/2}\Psi_0\big(n^2 \xi(z)\big) [-\phi(z)]^{-n\sigma_3}
 \\\label{LegendreGrowth}
 &\quad= \begin{pmatrix}
 O\big(n^{1/2} \big|\log\big(n^2(z+1)\big)\big|\big) & O\big(n^{1/2} \big|\log \big(n^2(z+1)\big)\big|\big)
 \\
 O\big(n^{-1/2}\big) & O\big(n^{-1/2}\big)
 \end{pmatrix}, \quad |z+1|\lesssim O\big(n^{-2}\big).
\end{align}
Now observe that for $|z+1|\lesssim O\big(n^{-2}\big)$, we have trivially $n^2|z+1|\lesssim O(1)$ and thus
\begin{align*}
 \big|\log\big(n^2(z+1)\big)\big| \lesssim \big|n^2(z+1)\big|^{-1/4}.
\end{align*}
This estimate, together with $n^{-1/2} \lesssim |z+1|^{-1/4}$ \big(in fact $n^{1/2} \lesssim |z+1|^{-1/4}$ holds\big), implies that the matrix entries in \eqref{LegendreGrowth} can be bounded by $O\big(|z+1|^{-1/4}\big)$ uniformly as $n \to \infty$, showing~\eqref{LegendreBehaviour} and finishing the proof.
\end{proof}

The fact that all three parametrices display the same asymptotic behaviour for $|z+1|\gtrsim O\big(n^{-2}\big)$ is consistent with the matching condition \eqref{ParaMatchingCond} which is the same in all three cases. Note that for a fixed $n$, $\tP(z)$ has only a logarithmic singularity near $z = -1$, but the order in~\eqref{LegendreBehaviour} is necessary to obtain a uniform bound for $n\to\infty$.
\begin{Corollary}
For $z$ in a neighbourhood $U_{+1}$ of $+1$, the matrix-valued function $\tQ$ satisfies the asymptotics
\begin{align}\label{LegendreQBehaviour}
 \tQ(z) = O\big(|z-1|^{-1/4}\big)
\end{align}
uniformly for $n\to\infty$. Moreover, $\tQ$, and its boundary values $\tQ_\pm$ on $\Sigma$, are bounded in $\overline{\C} := \C \cup \lbrace \infty \rbrace$ away from small neighbourhoods of $\pm 1$, uniformly for $n\to\infty$.
\end{Corollary}
\begin{proof}
Analogously to the local parametrix for the Legendre problem at $z\! =\! -1$, ${\tP_{-1}(z)\! := \!\tP(z)}$, one can construct a local parametrix $\tP_{+1}(z)$ for the Legendre problem near the point $z = +1$ with local jumps inside an open neighbourhood $U_{+1}$ of $+1$ as depicted in Figure~\ref{contour1}.

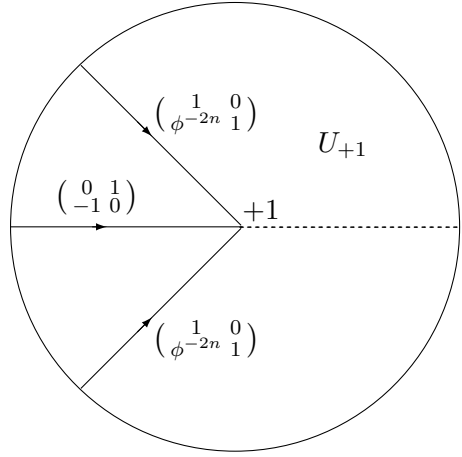
\begin{figure}[th]
\centering\vspace*{3mm}
\begin{picture}(7,5.2)
\put(0.42,2.5){\line(1,0){3.08}}
\put(1.5,2.5){\vector(1,0){0.2}}

\put(1.35,0.35){\line(1,1){2.13}}
\put(2.1,1.1){\vector(1,1){0.2}}

\put(1.35,4.65){\line(1,-1){2.13}}
\put(2.1,3.9){\vector(1,-1){0.2}}

\put(2.3,3.9){$\big(\begin{smallmatrix}
1 & 0
\\
\phi^{-2n} & 1
\end{smallmatrix}\big)$}
\put(1,2.8){$\big(\begin{smallmatrix}
0 & 1
\\
-1 & 0
\end{smallmatrix}\big)$}
\put(2.3,0.9){$\big(\begin{smallmatrix}
1 & 0
\\
\phi^{-2n} & 1
\end{smallmatrix}\big)$}

\put(3.5,2.6){$+1$}
\put(4.5,3.5){$U_{+1}$}

\curvedashes{0.05,0.05}
\curve(3.4,2.5, 6.4,2.5)

\put(3.4,2.5){\circle{5.97}}
\end{picture}
\vspace{12pt}
\caption{The jump conditions for $\tP_{+1}$.}\label{contour1}
\end{figure}

In fact, we have $\tP_{+1}(z) = \sigma_3 \tP_{-1}(-z)\sigma_3$, see \cite[equation~B.10]{CD}. Hence, it follows from Lemma~\ref{PropP}\,(iv) that
\begin{align}\label{Legendre+1Esti0}
 \tP_{+1}(z) = O\big(|z-1|^{-1/4}\big)
\end{align}
uniformly for $n\to\infty$. After deforming the local contour such that it matches locally with $\Sigma$ as depicted in Figure~\ref{contour2}, we can analytically continue $\tP_{+1}$ as necessary to obtain a \emph{deformed} local parametrix $\tP_{+1}^{\text{def}}$, which would satisfy locally the same jump conditions as $\tQ$.

\begin{figure}[th]
\centering\vspace{2mm}
\begin{picture}(7,5.2)
\put(0.42,2.5){\line(1,0){4.58}}
\put(5, 2.5){\line(-1,0.59){3.65}}
\put(5, 2.5){\line(-1,-0.59){3.65}}
\put(2.8,3.8){\vector(1,-0.6){0.2}}
\put(2.8,1.2){\vector(1,0.6){0.2}}
\put(1.5,2.5){\vector(1,0){0.2}}
\put(4.4,2.5){\vector(-1,0){0.2}}

\put(2.3,4.3){$\big(\begin{smallmatrix}
1 & 0
\\
\phi^{-2n} & 1
\end{smallmatrix}\big)$}
\put(1,2.8){$\big(\begin{smallmatrix}
0 & 1
\\
-1 & 0
\end{smallmatrix}\big)$}
\put(2.3,0.5){$\big(\begin{smallmatrix}
1 & 0
\\
\phi^{-2n} & 1
\end{smallmatrix}\big)$}

\put(4.6,1.4){$\big(\begin{smallmatrix}
1 & 0
\\
2\phi^{-2n} & 1
\end{smallmatrix}\big)$}

\put(3.5,2.6){$+1$}
\put(3.4, 2.415){$\boldsymbol{\cdot}$}
\put(5,2.6){$1+\delta$}
\put(4.5,3.5){$U_{+1}$}


\curvedashes{0.05,0.05}
\curve(5,2.5, 6.4,2.5)
\curve(4.6,1.7, 4.1, 2, 4, 2.5)
\put(4,2.3){\vector(0,1){0.2}}

\put(3.4,2.5){\circle{5.97}}
\end{picture}
\vspace{15pt}
\caption{The jump conditions for $\tP_{+1}^\text{def}$.}\label{contour2}
\end{figure}
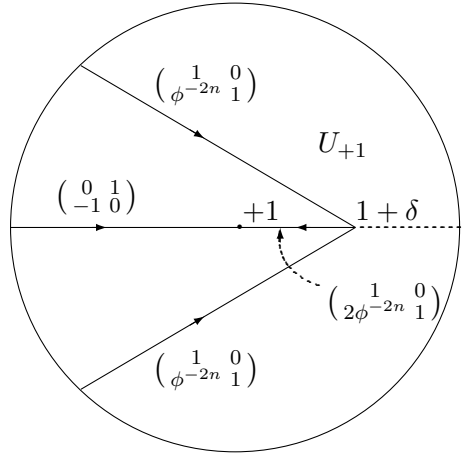

Because the jump matrices $\tv$ and their analytic continuations are uniformly bounded near $z = +1$, $\tP_{+1}^{\text{def}}$ would continue to satisfy the estimate \eqref{Legendre+1Esti0}
\begin{align}\label{Legendre+1Esti1}
 \tP_{+1}^{\text{def}}(z) = O\big(|z-1|^{-1/4}\big)
\end{align}
uniformly for $n\to\infty$. As the contour deformations are local, we also know that the matching condition \eqref{ParaConv} remains unchanged, at least on $\partial U_{+1}$:
\begin{align}\label{Legendre+1Esti2}
\tQ(s)\big[\tP_{+1}^{\text{def}}(s)\big]^{-1} = \Id + O\big(n^{-1}\big), \qquad s\in \partial U_{+1}.
\end{align}
However, as $\tQ\big[\tP_{+1}^{\text{def}}\big]^{-1}$ has no jumps inside $U_{+1}$, we can extend \eqref{Legendre+1Esti2} to all of $U_{+1}$ by the maximum modulus principle for holomorphic functions. Together with \eqref{Legendre+1Esti1}, the estimate~\eqref{LegendreQBehaviour} follows.

Regarding the second statement, it has been shown in \cite[Section~7]{KMVV} that for a wide class of weight functions including $\tw$, the outer parametrix solution $N$ is an approximation to the exact RH solution, in our case $\tQ$, uniformly in $z$ as long we stay away from the points $\pm 1$:
\begin{align*}
 \big|\tQ(z)-N(z)\big| = O\big(n^{-1}\big), \qquad z \mbox{ staying away from } \pm 1.
\end{align*}
As can be seen from \eqref{OuterParametrix}, $N$ is $n$-independent and bounded away from the points $\pm 1$. This finishes the proof.
\end{proof}

\section{Modified RH problems}\label{ModifiedRHP}
We are now in a position to define modified versions of the three RH problems found in Section~\ref{RHProblems}, which will be referred to as the respective ${}^\star$-analogs. The three ${}^\star$RH problems are introduced implicitly by defining their respective solutions. Here $U$ is a small neighbourhood of $-1$, as before.
\begin{align*}
 \sbQ(z) &= \begin{cases}
 \bQ(z), & z \in \mathbb{C} \setminus (\Sigma \cup U),
 \\
 \bQ(z) [\bP(z)]^{-1}, & z \in U,
 \end{cases}
 \\
 \shQ(z) &= \begin{cases}
 \hQ(z), & z \in \mathbb{C} \setminus (\Sigma \cup U),
 \\
 \hQ(z) [\hP(z)]^{-1}, & z \in U,
 \end{cases}
 \\
 \stQ(z) &=
 \begin{cases}
 \tQ(z), & z \in \mathbb{C} \setminus (\Sigma \cup U),
 \\
 \tQ(z) [\tP(z)]^{-1}, & z \in U.
 \end{cases}
\end{align*}
Note that because of \eqref{ParametrixJump}, the jumps on $U \cap \Sigma$ cancel out, which means that all three solutions can be uniquely defined on all of $U$. We will denote the new contour, which is the same for all three problems, by $\Sigma^\star = (\Sigma \setminus U) \cup \partial U$ and assume that $\partial U$ is oriented clockwise. The corresponding jump matrices are denoted by $\sbv$, $\shv$ and $\stv$, hence
\begin{align*}
 \sbQ_+(s) = \sbQ_-(s)\sbv(s), \qquad s\in \Sigma^\star,
\end{align*}
and so on. Note also that the normalization at infinity remains unchanged and the jump matrices on $\partial U$ are just the corresponding local parametrices:
\begin{align*}
 \sbv(s) = \bP(s), \qquad \shv(s) = \hP(s), \qquad \stv(s) = \tP(s), \qquad s \in \partial U.
\end{align*}

Now consider the singular integral operator
\begin{align*}
 1 - \sCautv \colon\ L^2(\Sigma^\star) \to L^2(\Sigma^\star), \qquad f \mapsto f - \sCau^-(f (\stv-\Id)).
\end{align*}
We claim that this operator is invertible and that the inverse is given by
\begin{align*}
 &(1-\sCautv)^{-1} \colon\ L^2(\Sigma^\star) \to L^2(\Sigma^\star), \qquad f \mapsto f + \sCau^-\big(f (\stv-\Id) [\stQ_+]^{-1}\big)\stQ_-.
\end{align*}
At the moment, it is not even clear whether the above operator is well-defined as a map from~$L^2(\Sigma^\star)$ to itself, as $\stQ_\pm$ has a logarithmic singularity near $+1$.

To prove the claim, we proceed as follows: First, we partition the jump contour $\Sigma^\star = \underbrace{\partial U}_{\Sigl} \, \cup \,  \underbrace{\Sigma \setminus U}_{\Sigr}$ as shown in Figure~\ref{SigmaStar}.

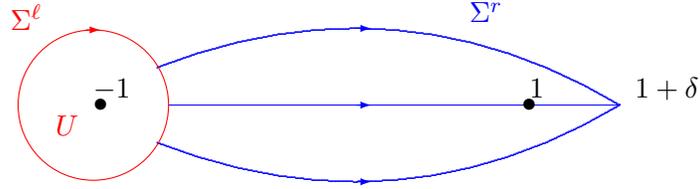
\begin{figure}[th]
\centering\vspace{-12mm}
\begin{picture}(7,5.2)
{\color{blue}\put(1,2.5){\line(1,0){6.0}}}%
{\color{blue}\put(3.5,2.5){\vector(1,0){0.2}}}%

{\color{blue}\put(3.5,3.52){\vector(1,0){0.2}}}%
{\color{blue}\put(3.5,1.48){\vector(1,0){0.2}}}%
{\color{red}\put(-0.1,3.5){\vector(1,0){0.2}}}%

{\color{blue}\curve(0.85,2, 4,1.5, 7,2.5)}%
{\color{blue}\curve(0.85,3, 4,3.5, 7,2.5)}%


\put(7.2, 2.6){$1+\delta$}
\put(0, 2.6){$-1$}
\put(-0.5,2.1){$\textcolor{red}{U}$}
\put(5.8,2.6){$1$}
\put(5.7,2.415){$\bullet$}
\put(0,2.415){$\bullet$}
{\color{red}\put(-1.1, 3.5){$\Sigl$}}%
{\color{blue}\put(5,3.6){$\Sigr$}}%

{\color{red}\put(0,2.5){\circle{2}}}%
\end{picture}
\vspace{-40pt}
\caption{The jump contour $\Sigma^\star = \Sigl \cup \Sigr$.}\label{SigmaStar}
\end{figure}

Next, we decompose the operator $\sCau^-\big(\ \cdot \ (\stv-\Id) [\stQ_+]^{-1}\big)\stQ_-$. Setting $\stu = \stv-\Id$, we obtain
\begin{align*}
 \sCau^-\big(\ \cdot \ \stu [\stQ_+]^{-1}\big)\stQ_- ={}&
 \sCau^-\big(\ \cdot \ \chi_{\Sigl} \stu [\stQ_+]^{-1}\big)\stQ_- + \sCau^-\big(\ \cdot \ \chi_{\Sigr} \stu [\stQ_+]^{-1}\big)\stQ_-
 \\
 ={}& \sCau^-\big(\ \cdot \ \chi_{\Sigl} \stu [\stQ_+]^{-1}\big)\stQ_- + \sCau^-\big(\ \cdot \ \chi_{\Sigr} \stu [\stQ_+]^{-1}\big)\stQ_-\chi_{\Sigr}
 \\
 & + \sCau^-\big(\ \cdot \ \chi_{\Sigr} \stu [\stQ_+]^{-1}\big)\stQ_-\chi_{\Sigl},
\end{align*}
where $\chi_{\Sigma^j}$ is the characteristic function of $\Sigma^j$ for $j = \ell, r$. By definition, we have $\stQ_\pm \chi_{\Sigma^r} = \tQ_\pm \chi_{\Sigma^r}$. Note that the mapping
\begin{align*}
f \mapsto f\chi_{\Sigl} \stu \big[\stQ_+\big]^{-1}
\end{align*}
is an operator uniformly bounded in $n$ from $L^2(\Sigma^\star) \to L^2\big(\Sigma^\star, |z-1|^{-1/4}\big)$, as $\big[\stQ_+\big]^{-1}$ converges uniformly on $\Sigl = \partial U$ to the outer parametrix $N$, see \eqref{ParaMatchingCond} and \eqref{ParaConv}. Here and in the following, we refer to \emph{uniform boundedness in $n$} of some operator $T = T_n$, to the operator norm~$\Vert T_n \Vert$ being bounded by an $n$-independent constant, cf.\ Theorem~\ref{UnifLeg}. As $|z-1|^{-1/4} \in A_2(\Sigma^\star)$ (see Theorem~\ref{Muckenhoupt}), we have that the mapping
\begin{align*}
 f \mapsto \sCau^{-}\big(f\chi_{\Sigl} \stu \big[\stQ_+\big]^{-1}\big)
\end{align*}
defines a uniformly bounded operator from $L^2(\Sigma^\star) \to L^2\big(\Sigma^\star, |z-1|^{-1/4}\big)$. Finally, using $\stQ = \tQ$ in $\C \setminus U$ and the estimate \eqref{LegendreQBehaviour}, we conclude that
\begin{align*}
 f \mapsto \sCau^{-}\big(f\chi_{\Sigl} \stu [\stQ_+]^{-1}\big)\stQ_-
\end{align*}
defines a uniformly bounded operator from $L^2(\Sigma^\star) \to L^2(\Sigma^\star)$.

The uniform boundedness of the mapping
\begin{align*}
 f \mapsto \sCau^-\big(f\chi_{\Sigr} \stu [\stQ_+]^{-1}\big)\stQ_-\chi_{\Sigr}
\end{align*}
as an operator from $L^2(\Sigma^\star) \to L^2(\Sigma^\star)$ follows directly from Proposition~\ref{ExpForm} together with Theorem \ref{UnifLeg}.

Finally, the mapping
\begin{align*}
 f \mapsto f\chi_{\Sigr} \stu \big[\stQ_+\big]^{-1}
\end{align*}
defines a uniformly bounded operator from $L^2(\Sigma^\star)$ to $L^2\big(\Sigma^\star, |z-1|^{1/4}\big)$. As $|z-1|^{1/4} \in A_2(\Sigma^\star)$, the mapping
\begin{align*}
 f \mapsto \sCau^-\big(f\chi_{\Sigr} \stu \big[\stQ_+\big]^{-1}\big)
\end{align*}
is a uniformly bounded operator from $L^2(\Sigma^\star) \to L^2\big(\Sigma^\star, |z-1|^{1/4} \big)$ as well. Moreover, the multiplication with $\stQ_-\chi_{\Sigl}$ defines a uniformly bounded operator from $L^2\big(\Sigma^\star, |z-1|^{1/4}\big) \to L^2(\Sigma^\star)$, implying that
\begin{align*}
 f \mapsto \sCau^-\big(f\chi_{\Sigr} \stu \big[\stQ_+\big]^{-1}\big)\stQ_-\chi_{\Sigma^\ell}
\end{align*}
is a uniformly bounded operator from $L^2(\Sigma^\star)\to L^2(\Sigma^\star)$. We obtain:
\begin{Theorem}\label{UniformBoundedness}
The inverse of the operator $1-\sCautv$ on $L^2(\Sigma^\star)$ exists and is given by
\begin{align*}
 (1-\sCautv)^{-1} \colon\ L^2(\Sigma^\star) \to L^2(\Sigma^\star), \ \ f \mapsto f + \sCau^-\big(f (\stv-\Id)\big[\stQ_+\big]^{-1}\big)\stQ_-.
\end{align*}
Moreover, the operator norm is uniformly bounded as $n \to \infty$.
\end{Theorem}
\begin{proof}
That the mapping is indeed uniformly bounded follows from the preceding argument. The fact that it coincides with the inverse of $1-\sCautv$ follows from a computation analogous to the one given in the proof of Proposition~\ref{ExpForm}.
\end{proof}

\subsection[Uniform invertibility of the star resolvents]{Uniform invertibility of the $\boldsymbol{{}^\star}$resolvents}\label{Comparison}
The uniform boundedness stated in Theorem \ref{UniformBoundedness} can now be extended to the uniform invertibility of the operators $1-\sCautv$, $1-\sCauhv$. Note that the jump matrices satisfy
\begin{align}\label{UniformConv}
\Vert \sbv - \stv \Vert_{L^\infty(\Sigma^\star)}, \ \Vert \shv - \stv \Vert_{L^\infty(\Sigma^\star)} \to 0
\end{align}
for $n \to \infty$. In fact, we have
\begin{align}
 \sbv - \stv =
 \begin{cases}
 \begin{pmatrix}
 0 & 0
 \\
 \big(\Fw(s)-1\big)\phi^{-2n}(s) & 0
 \end{pmatrix}, & s\in (\Sigma_1 \cup \Sigma_2)\cap\Sigma^r,\vspace{1mm}
 \\
 \begin{pmatrix}
 0 & 0
 \\
 \big(\Fw_+(s) +\Fw_-(s)-2\big)\phi^{-2n}(s) & 0
 \end{pmatrix}, & s\in (1,1+\delta),
 \\
 0, & s\in(-1,1)\cap \Sigma^r,
 \\
 \bP(s)-\tP(s), & s\in\Sigma^\ell,
 \end{cases}\label{DiffJumpMatrices}
\end{align}
and the corresponding claim in \eqref{UniformConv} for the difference $\sbv - \stv$ follows from \eqref{phiGreater1}, Corollary~\ref{CorF+1} and \eqref{ParaMatchingCond}. An analog formula of the form \eqref{DiffJumpMatrices} can be written down for $\shv - \stv$:
\begin{align*}
 \shv - \stv =
 \begin{cases}
 \begin{pmatrix}
 0 & 0
 \\
 \big(\hFw(s)-1\big)\phi^{-2n}(s) & 0
 \end{pmatrix}, & s\in (\Sigma_1 \cup \Sigma_2)\cap\Sigma^r,\vspace{1mm}
 \\
 \begin{pmatrix}
 0 & 0
 \\
 2\big(\hFw(s)-1\big)\phi^{-2n}(s) & 0
 \end{pmatrix}, & s\in (1,1+\delta),
 \\
 0, & s\in(-1,1)\cap \Sigma^r,
 \\
 \hP(s)-\tP(s), & s\in\Sigma^\ell.
 \end{cases}
\end{align*}
Here \eqref{UniformConv} follows as before after using estimate \eqref{hat F^2/w at +1} instead of \eqref{F^2/w at +1}.

Thus it follows that
\begin{align*}
 \Vert (1-\sCauv) - (1-\sCautv) \Vert_{L^2(\Sigma^\star) \to L^2(\Sigma^\star)} \to 0,
 \qquad
 \Vert (1-\sCauhv) - (1-\sCautv) \Vert_{L^2(\Sigma^\star) \to L^2(\Sigma^\star)} \to 0
\end{align*}
as $n \to \infty$. As the operator $1-\sCautv$ is uniformly invertible, a standard argument (see, e.g., \cite[Theorem~4.7]{CD}) shows that the operators $1-\sCauv$, $1-\sCauhv$ are also uniformly invertible for $n$ large enough. We summarize:
\begin{Theorem}\label{TheoremUniInv}
The operators $1-\sCauv$ and $1-\sCauhv$ are invertible for $n$ large enough and the operator norms of their inverses are uniformly bounded as $n \to \infty$.
\end{Theorem}
Note that Theorem \ref{TheoremUniInv} is the analog of Theorem \ref{UnifLeg} initially proved in \cite[Theorem~4.5]{CD}, but for the logarithmic weight function and on a contour avoiding the problematic point $z = -1$. Interestingly Theorem \ref{UnifLeg} played a crucial role in the proof of Theorem \ref{TheoremUniInv}.

\section{Asymptotic Analysis}
\label{SectAsymptoticAnalysis}
The following section will be largely based on \cite[Section~5]{CD} and culminates in the proof of Theorem \ref{MainTheorem}.

\subsection{Some norm estimates}
Let us define $\smu_{n} = Q_{-}^{\star (n)}$ and $\shmu_{n} = \widehat Q_{-}^{\star (n)}$, where in the following we will make the $n$-dependence explicit. In particular, we will denote the jump matrices with a subscript $n$ and the Cauchy operator with a superscript $(n)$ for notational convenience, i.e., $v^\star_n$, \smash{$\Cau^{(n)}$} and so on. We will now prove an analog of \cite[Proposition~5.1]{CD} in the ${}^\star$-case.
\begin{Proposition}\label{PropFourInt}
The following estimates hold for $n \to \infty$:
\begin{enumerate}[$(i)$]\itemsep=0pt
 \item $\Vert \shmu_n (\sbv_n - \shv_n) \Vert_{L^2(\Sigma^\star)} = O\big(\frac{1}{n^{1/2}\log^2 n}\big)$,
 \item $\Vert \smu_n - \shmu_n \Vert_{L^2(\Sigma^\star)} = O\big(\frac{1}{n^{1/2}\log^2 n}\big)$,
 \item $\Vert \shmu_n (\sbv_n - \shv_n) - \shmu_{n+1} (\sbv_{n+1} - \shv_{n+1}) \Vert_{L^2(\Sigma^\star)} = O\big(\frac{1}{n^{3/2}\log^2 n}\big)$,
 \item $\Vert (\smu_n - \shmu_n) - (\smu_{n+1} - \shmu_{n+1}) \Vert_{L^2(\Sigma^\star)} = O\big(\frac{1}{n^{3/2}\log^2 n}\big)$.
\end{enumerate}
\end{Proposition}
\begin{proof}
The proof is for the most part taken from \cite[Proposition~B.2]{CD}, with the only difference being the contribution from $\Sigl =\partial U$ instead of $\Sigma \cap U$.

For (i), let us write
\begin{align}\label{IntDecomp}
 \Vert \shmu_n (\sbv_n - \shv_n) \Vert_{L^2(\Sigma^\star)}^2 = \Vert \shmu_n (\sbv_n - \shv_n) \Vert_{L^2(\Sigl)}^2 + \Vert \shmu_n (\sbv_n - \shv_n) \Vert_{L^2(\Sigr)}^2.
\end{align}
As $\hmu_n|_{s\in \Sigr} = \shmu_n|_{s\in\Sigr}$, $\bv_n|_{s\in \Sigr} = \sbv_n|_{s\in\Sigr}$ and $\hv_n|_{s\in \Sigr} = \shv_n|_{s\in\Sigr}$, one can conclude from \cite[equation~5.1]{CD} that
\begin{align} \label{Estimate05}
 \Vert \shmu_n (\sbv_n - \shv_n) \Vert_{L^2(\Sigr)} = O\left(\frac{1}{n^{1/2}\log^2 n}\right).
\end{align}
Hence it remains to consider the term $\Vert \shmu_n (\sbv_n - \shv_n) \Vert_{L^2(\Sigl)}$. Recall that for $s \in \Sigl$ we have $\sbv_n(s) = \bP^{(n)}(s)$ and $\shv_n(s) = \hP^{(n)}(s)$ and so it follows from Lemma~\ref{PropP}\,(ii) that
\begin{align}\label{jumpsComparison}
 \Vert \sbv_n - \shv_n \Vert_{L^\infty(\Sigl)} = O\big(n^{-1}\big).
\end{align}
Moreover, from Lemma~\ref{PropP}\,(iii) it also follows that $\shmu_n(s) = \hQ^{(n)}(s) \big[\hP^{(n)}(s)\big]^{-1}$ is uniformly bounded for $s \in \Sigl$. Hence we conclude
\begin{align*}
 \Vert \shmu_n (\sbv_n - \shv_n) \Vert_{L^2(\Sigl)} = O\bigl(n^{-1}\bigr)
\end{align*}
which together with \eqref{Estimate05} proves (i).

Point (ii) follows from (i) by considering
\begin{align*}
 \smu_n - \shmu_n &= \big(1-\sCauv^{(n)}\big)^{-1} \Id - \big(1-\sCauhv^{(n)}\big)^{-1}\Id
= \big(1-\sCauv^{(n)}\big)^{-1}\big(\sCauv^{(n)}-\sCauhv^{(n)}\big) \big(1-\sCauhv^{(n)}\big)^{-1}\Id
 \\
 &= \big(1-\sCauv^{(n)}\big)^{-1}\big(\sCauv^{(n)}-\sCauhv^{(n)}\big)\shmu_n= \big(1-\sCauv^{(n)}\big)^{-1}\sCau^-(\shmu_n(\sbv_n-\shv_n)).
\end{align*}
As $\big(1-\sCauv^{(n)}\big)^{-1}$, $\sCau^-$ are bounded operators (uniformly in $n$), it follows that{\samepage
\begin{align*}
 \Vert \smu_n - \shmu_n \Vert_{L^2(\Sigma^\star)} \lesssim \Vert \shmu_n(\sbv_n-\shv_n) \Vert_{L^2(\Sigma^\star)} = O\left(\frac{1}{n^{1/2}\log^2 n}\right)
\end{align*}
showing (ii).}

For (iii), we will again decompose the norm as in \eqref{IntDecomp}. As before, following the arguments found in \cite[Proposition~B.2]{CD}, we conclude that
\begin{align*}
\Vert \shmu_n (\sbv_n - \shv_n) - \shmu_{n+1} (\sbv_{n+1} - \shv_{n+1}) \Vert_{L^2(\Sigr)} = O\left(\frac{1}{n^{3/2}\log^2 n}\right).
\end{align*}
For the remaining term, we will use the fact that the local parametrices $\bP^{(n)}(s)$, $\hP^{(n)}(s)$ have an infinite series expansion in powers of $n^{-1}$ which is uniform on $\Sigl$, see \cite[equation~(8.2)]{KMVV}:
\begin{align*}
& v_n^\star(s)= P^{(n)}(s) \sim \left( \Id + \sum_{k=1}^\infty \frac{\Delta_k(s)}{n^k}\right) N(s),
 \\
& \shv_n(s)= \hP^{(n)}(s) \sim \left( \Id + \sum_{k=1}^\infty \frac{\widehat \Delta_k(s)}{n^k}\right) N(s),
\end{align*}
where $\Delta_k$ and $\widehat \Delta_k$ are $n$ independent and can be explicitly computed. In particular, \begin{align}\label{vinftyEst}
 \sbv_n(s) - \sbv_{n+1}(s) = O\big(n^{-2}\big), \qquad \shv_n(s) - \shv_{n+1}(s) = O\big(n^{-2}\big)
\end{align}
uniformly for $s \in \Sigl$. Hence, using the fact $\shmu_n(s) = \Id + O\big(n^{-1}\big)$ uniformly for $s \in \Sigl$ (see~\eqref{ParaConv}), we can write
\begin{align*}
 \Vert \shmu_n (\sbv_n - \shv_n) - \shmu_{n+1} (\sbv_{n+1} - \shv_{n+1}) \Vert_{L^2(\Sigl)}
 \lesssim \Vert (\shmu_n-\shmu_{n+1}) (\sbv_n-\shv_n) \Vert_{L^2(\Sigl)} + O\big(n^{-2}\big).
\end{align*}
Additionally, it follows that $\shmu_n(s)-\shmu_{n+1}(s) = O\big(n^{-1}\big)$, (even $O\big(n^{-2}\big)$ see \cite[equation~(8.7)]{KMVV}) uniformly for $s \in \Sigl$. Together with \eqref{jumpsComparison}, we can conclude that
\begin{align*}
 \Vert (\shmu_n-\shmu_{n+1}) (\sbv_n-\shv_n) \Vert_{L^2(\Sigl)} = O\big(n^{-2}\big),
\end{align*}
which proves (iii).

In order to prove (iv), we first write
\begin{align*}
 (\smu_n - \shmu_n)-(\smu_{n+1} - \shmu_{n+1})={}&\big(\big(1-\sCauv^{(n)}\big)^{-1} \Id - \big(1-\sCauhv^{(n)}\big)^{-1}\Id\big)
 \\
 &-\big(\big(1-\sCauv^{(n+1)}\big)^{-1} \Id - \big(1-\sCauhv^{(n+1)}\big)^{-1}\Id\big)
 \\
={}& \big(1-\sCauv^{(n)}\big)^{-1}\sCau^-(\shmu_n(\sbv_n-\shv_n))
 \\
 & -\big(1-\sCauv^{(n+1)}\big)^{-1}\sCau^-(\shmu_{n+1}(\sbv_{n+1}-\shv_{n+1}))
 \\
={}& \big(\big(1-\sCauv^{(n)}\big)^{-1}-\big(1-\sCauv^{(n+1)}\big)^{-1}\big)\sCau^-(\shmu_n(\sbv_n-\shv_n))
 \\
 & + \big(1-\sCauv^{(n+1)}\big)^{-1}\sCau^-(\shmu_n(\sbv_n-\shv_n)-\shmu_{n+1}(\sbv_{n+1}-\shv_{n+1})).
\end{align*}
From the uniform boundedness of $\big(1-\sCauv^{(n+1)}\big)^{-1}$, $\sCau^-$ and point (iii), it follows that
\begin{align*}
 \big\Vert \big(1-\sCauv^{(n+1)}\big)^{-1}\sCau^-(\shmu_n(\sbv_n-\shv_n)-\shmu_{n+1}(\sbv_{n+1}-\shv_{n+1})) \big\Vert_{L^2(\Sigma^\star)} = O\left(\frac{1}{n^{3/2}\log^2 n}\right).
\end{align*}
For the remaining term, we have
\begin{align}
 &\big(\big(1-\sCauv^{(n)}\big)^{-1}-\big(1-\sCauv^{(n+1)}\big)^{-1}\big)\sCau^-(\shmu_n(\sbv_n-\shv_n))\nonumber
 \\
 &\qquad= \big(\big(1-\sCauv^{(n)}\big)^{-1}\big(\sCauv^{(n)}-\sCauv^{(n+1)}\big)\big(1-\sCauv^{(n+1)}\big)^{-1}\big)\sCau^-(\shmu_n(\sbv_n-\shv_n)).\label{CC}
\end{align}
Now observe that
\begin{align}\label{CauMatrixIneq}
 \big\Vert \sCauv^{(n)}-\sCauv^{(n+1)} \big\Vert_{L^2(\Sigma^\star)\to L^2(\Sigma^\star)} \lesssim \Vert \sbv_n - \sbv_{n+1} \Vert_{L^\infty(\Sigma^\star)}.
\end{align}
It has been shown in \cite[p.~54]{CD} that $\Vert \sbv_n - \sbv_{n+1} \Vert_{L^\infty(\Sigr)}=O\big(n^{-1}\big)$, hence with \eqref{vinftyEst} it follows that
\begin{align*}
 \Vert \sbv_n - \sbv_{n+1} \Vert_{L^\infty(\Sigma^\star)} = O\big(n^{-1}\big),
\end{align*}
which implies using the bound \eqref{CauMatrixIneq}
\begin{align*}
 \big\Vert \sCauv^{(n)}-\sCauv^{(n+1)} \big\Vert_{L^2(\Sigma^\star)\to L^2(\Sigma^\star)} = O\big(n^{-1}\big).
\end{align*}
Plugging this estimate together with (i) into \eqref{CC}, we conclude that
\begin{align*}
 \big\Vert \big(\big(1-\sCauv^{(n)}\big)^{-1}-\big(1-\sCauv^{(n+1)}\big)^{-1}\big)\sCau^-(\shmu_n(\sbv_n-\shv_n))\big\Vert_{L^2(\Sigma^\star)} = O\left(\frac{1}{n^{3/2}\log^2 n}\right),
\end{align*}
which implies (iv) and finishes the proof.
\end{proof}

We immediately get from Proposition~\ref{PropFourInt}:
\begin{Corollary}\label{CorFourInt}
The following estimates hold for $n \to \infty$:
\begin{enumerate}[$(i)$]\itemsep=0pt
 \item $\Vert \hmu_n (\bv_n - \hv_n) \Vert_{L^2(\Sigr)} = O\big(\frac{1}{n^{1/2}\log^2 n}\big)$,
 \item $\Vert \mu_n - \hmu_n \Vert_{L^2(\Sigr)} = O\big(\frac{1}{n^{1/2}\log^2 n}\big)$,
 \item $\Vert \hmu_n (\bv_n - \hv_n) - \hmu_{n+1} (\bv_{n+1} - \hv_{n+1}) \Vert_{L^2(\Sigr)} = O\big(\frac{1}{n^{3/2}\log^2}
 \big)$,
 \item $\Vert (\mu_n - \hmu_n) - (\mu_{n+1} - \hmu_{n+1}) \Vert_{L^2(\Sigr)} = O\big(\frac{1}{n^{3/2}\log^2 n}\big)$.
\end{enumerate}
\end{Corollary}
Note that we restricted the path of integration to $\Sigr$. The contributions coming from $\Sigma \cap U$ turn out to be of smaller order as will be shown in Lemma~\ref{FinalLemma}.

To obtain an asymptotic formula for the recurrence coefficients, we need to obtain an asymptotic formula for \eqref{QIntForm} which in our current notation reads
\begin{align}\label{QQ1}
 \bQ_1^{(n)} - \hQ_1^{(n)} = - \frac{1}{2\pi\I} \int_\Sigma \bmu_n(\bv_n-\hv_n)\hmu^{-1}_n {\rm d}s.
\end{align}
The key in proving Theorem \ref{MainTheorem} lies in the following proposition.
\begin{Proposition}\label{QQ}
For $n \to \infty$ the following estimates hold:
\begin{align*}
 \bQ_1^{(n)} - \hQ_1^{(n)} = \frac{3}{16n \log^2n} \begin{pmatrix}
-1 & \I\\
\phantom{-}\I & 1
\end{pmatrix} + O\left(\frac{1}{n\log^3 n}\right)
\end{align*}
and
\begin{align} \label{QDiffFormula2}
 \bQ_1^{(n)} - \hQ_1^{(n)} - \big(\bQ_1^{(n+1)} - \hQ_1^{(n+1)}\big) = \frac{3}{16n^2 \log^2n} \begin{pmatrix}
-1 & \I\\
\phantom{-}\I & 1
\end{pmatrix} + O\left(\frac{1}{n^2\log^3 n}\right).
\end{align}
\end{Proposition}
We will make use of two important results stated in \cite[Section~5.2]{CD}, but restricted to the contour $\Sigr$:
\begin{Proposition}\label{IntegralEst1}
The following estimates hold:
\begin{align}\label{LegendreSubstitution1}
 \int_{\Sigr} \bmu_n(\bv_n-\hv_n)\hmu_n^{-1} {\rm d}s = \int_{\Sigr} \hmu_n(\bv_n-\hv_n)\hmu_n^{-1} {\rm d}s + O\left(\frac{1}{n\log^4 n}\right),
\end{align}
and
\begin{align}
 & \int_{\Sigr} \bmu_n(\bv_n-\hv_n)\hmu^{-1}_n {\rm d}s - \int_{\Sigr} \bmu_{n+1}\big(\bv_{n+1}-\hv_{n+1}\big)\hmu^{-1}_{n+1} {\rm d}s
\nonumber \\
 &\qquad=\int_{\Sigr} \hmu_n(\bv_n-\hv_n)\hmu^{-1}_n {\rm d}s - \int_{\Sigr} \hmu_{n+1}\big(\bv_{n+1}-\hv_{n+1}\big)\hmu^{-1}_{n+1} {\rm d}s + O\left(\frac{1}{n^2\log^4 n}\right).\label{LegendreSubstitution2}
\end{align}
\end{Proposition}
\begin{proof}
For \eqref{LegendreSubstitution1}, observe that
\begin{align} \label{SubstitutionEq}
 \int_{\Sigr} \bmu_n&(\bv_n-\hv_n)\hmu_n^{-1} {\rm d}s
= \int_{\Sigr} \hmu_n(\bv_n-\hv_n)\hmu_n^{-1} {\rm d}s + \int_{\Sigr} (\bmu_n-\hmu_n)(\bv_n-\hv_n)\hmu_n^{-1} {\rm d}s.
\end{align}
As $\det \bmu = \det \hmu \equiv 1$ and $\bv-\hv$ has a nonzero entry only in the $21$-entry, it follows through explicit calculation that $\Vert (\bv_n - \hv_n)\hmu_n^{-1} \Vert_{L^2(\Sigr)} =\Vert \hmu_n (\bv_n - \hv_n) \Vert_{L^2(\Sigr)}$.
Now using estimates (i) and (ii) from Corollary~\ref{CorFourInt}, it follows that the error term in \eqref{SubstitutionEq} can be bounded by
\begin{align*}
 &\bigg| \int_{\Sigr}(\bmu_n-\hmu_n)(\bv_n-\hv_n)\hmu_n^{-1} {\rm d}s\bigg|\lesssim \Vert \mu_n - \hmu_n \Vert_{L^2(\Sigr)} \big\Vert (\bv_n - \hv_n)\hmu_n^{-1} \big\Vert_{L^2(\Sigr)} = O\left(\frac{1}{n\log^4 n}\right).
\end{align*}
In a similar fashion for \eqref{LegendreSubstitution2}, we can write
\begin{align}
 &\int_{\Sigr} \bmu_n(\bv_n-\hv_n)\hmu^{-1}_n {\rm d}s - \int_{\Sigr} \bmu_{n+1}(\bv_{n+1}-\hv_{n+1})\hmu^{-1}_{n+1} {\rm d}s\nonumber
 \\
 &\qquad=\int_{\Sigr} \hmu_n(\bv_n-\hv_n)\hmu^{-1}_n {\rm d}s - \int_{\Sigr} \hmu_{n+1}(\bv_{n+1}-\hv_{n+1})\hmu^{-1}_{n+1} {\rm d}s\nonumber
 \\
 &\qquad\phantom{=}{}+ \int_{\Sigr} (\bmu_n-\hmu_n)\big[(\bv_n-\hv_n)\hmu^{-1}_n-(\bv_{n+1}-\hv_{n+1})\hmu^{-1}_{n+1}\big] {\rm d}s\nonumber
 \\
 &\qquad\phantom{=}{}+ \int_{\Sigr} \big[(\bmu_n-\hmu_n)-(\bmu_{n+1}-\hmu_{n+1})\big](\bv_{n+1}-\hv_{n+1})\hmu^{-1}_{n+1} {\rm d}s.\label{SubstitutionEq2}
\end{align}
We can now estimate the two error terms in \eqref{SubstitutionEq2}, using Corollary~\ref{CorFourInt},
\begin{align*}
 &\bigg|\int_{\Sigr} (\bmu_n-\hmu_n)\big[(\bv_n-\hv_n)\hmu^{-1}_n-(\bv_{n+1}-\hv_{n+1})\hmu^{-1}_{n+1}\big] {\rm d}s\bigg|
 \\
 &\qquad\lesssim \Vert \mu_n - \hmu_n \Vert_{L^2(\Sigr)} \Vert \hmu_n (\bv_n - \hv_n) - \hmu_{n+1} (\bv_{n+1} - \hv_{n+1}) \Vert_{L^2(\Sigr)} = O\left(\frac{1}{n^{2}\log^4 n}\right),
\end{align*}
and
\begin{align*}
 &\bigg| \int_{\Sigr} \big[(\bmu_n-\hmu_n)-(\bmu_{n+1}-\hmu_{n+1})\big](\bv_{n+1}-\hv_{n+1})\hmu^{-1}_{n+1} {\rm d}s \bigg|
 \\
 &\qquad\lesssim \Vert (\mu_n - \hmu_n) - (\mu_{n+1} - \hmu_{n+1}) \Vert_{L^2(\Sigr)} \Vert \hmu_{n+1} (\bv_{n+1} - \hv_{n+1}) \Vert_{L^2(\Sigr)} = O\left(\frac{1}{n^{2}\log^4 n}\right),
\end{align*}
proving \eqref{LegendreSubstitution2}.
\end{proof}

The next result, stated in \cite[Section~5]{CD}, contains the leading order term in the integrals found in Proposition~\ref{IntegralEst1}:
\begin{Proposition}\label{IntegralEst2}
 The following estimates holds:
\begin{align*}
\frac{1}{2\pi\I} \int_{\Sigr} \hmu_n(\bv_n-\hv_n)\hmu_n^{-1} {\rm d}s = \frac{3}{16n \log^2n} \begin{pmatrix}
\phantom{-}1 & -\I\\
-\I & -1
\end{pmatrix} + O\left(\frac{1}{n\log^3 n}\right)
\end{align*}
and
\begin{align*}
 & \frac{1}{2\pi\I} \int_{\Sigr} \hmu_n(\bv_n-\hv_n)\hmu_n^{-1} {\rm d}s - \frac{1}{2\pi\I} \int_{\Sigr} \hmu_{n+1}(\bv_{n+1}-\hv_{n+1})\hmu^{-1}_{n+1} {\rm d}s
 \\
 &\qquad= \frac{3}{16n^2 \log^2n} \begin{pmatrix}
\phantom{-}1 & -\I\\
-\I & -1
\end{pmatrix} + O\left(\frac{1}{n^2\log^3 n}\right).
\end{align*}
\end{Proposition}
\begin{proof}
The proof is essentially given in \cite[Propositions~C.3 and C.4]{CD}. The main difference is the occurrence of $\hmu_n$ instead of $\tmu_n = \tQ^{(n)}_-$. We provide the full proof of both formulas in Proposition~\ref{IntegralEst2Ap}.
\end{proof}

Note that in the formula \eqref{QQ1} the integral is over $\Sigma$, while in the Propositions~\ref{IntegralEst1} and \ref{IntegralEst2} the integrals are over $\Sigr$. It remains to show that the remaining integral over $ \Sigma^U = \Sigma\cap U$, which is localized around the point $-1$, is negligible. This is the content of the next lemma.

\begin{Lemma}\label{FinalLemma}
For $n \to \infty$, we have
\begin{align}\label{FinalInt}
 \int_{ \Sigma^U} \bmu_n(\bv_n-\hv_n)\hmu_n^{-1} {\rm d}s = O\left(\frac{1}{n^3}\right).
\end{align}
\end{Lemma}
\begin{proof}
First, we need to show that $\sbQ(z) = \bQ(z)[\bP(z)]^{-1}$ for $z \in U$ is uniformly bounded as $n \to \infty$. Note that this statement is more intricate than the corresponding statements for $\hQ(z)\big[\hP(z)\big]^{-1}$ and $\tQ(z)\big[\tP(z)\big]^{-1}$, as the logarithmic weight function is not covered in \cite{KMVV} and hence we do not have equation~\eqref{ParaConv} for $\bQ(z)[\bP(z)]^{-1}$ at our disposal. Recall that for~${s\!\in\! \Sigma^\ell = \partial U}$, we have
\begin{align}\label{smu_on_U}
 \smu(s) = \sbQ_-(s) = \bQ(s)\big[\bP(s)\big]^{-1}.
\end{align}
From \eqref{smu_on_U}, we see that on $\partial U$, $\smu$ is in fact the restriction of the analytic function $\sbQ(z) = \bQ(z)[\bP(z)]^{-1}$, $z\in U$. Moreover, $\Vert \smu \Vert_{L^2( \Sigma^U)} = O(1)$ for $n \to \infty$, as $\smu = (1-\sCauv)^{-1}\Id$ and~${(1-\sCauv)^{-1}}$ is uniformly bounded on $L^2(\Sigma^\star)$. Thus, using Cauchy's integral formula, we get for~$z \in U$
\begin{align*}
 |\sbQ(z)| = \bigg|\frac{1}{2\pi\I} \int_{\partial U} \frac{\smu(s){\rm d}s}{s-z}\bigg| \lesssim O\left(\frac{1}{\dist(z,\partial U)}\right).
\end{align*}
We have some freedom to choose $U$, so if necessary we can shrink it and conclude
\begin{align}\label{Qstar}
 |\sbQ(z)| = O(1), \qquad z\in U,
\end{align}
uniformly as $n \to \infty$. Let us now rewrite \eqref{FinalInt}, where we drop the $n$-dependence for better readability:
\begin{align}\label{FinalInt2}
 \int_{ \Sigma^U} \bmu(\bv-\hv)\hmu^{-1} {\rm d}s = \int_{ \Sigma^U} \sbQ(s)\bP_-(s)(\bv(s)-\hv(s))\big[\hP_-(s)\big]^{-1}\big[\shQ(s)\big]^{-1} {\rm d}s.
\end{align}
We now list bounds for each of the factors in the integrand:
\begin{itemize}\itemsep=0pt
 \item \eqref{phiEndpoints} and \eqref{3/2decay} imply that for some $c > 0$
 \begin{align*}\bv(s)-\hv(s) =\begin{cases} \begin{pmatrix}
 0 & 0
 \\
 \big(\Fw(s)-\hFw(s)\big)\phi^{-2n}(s) & 0
 \end{pmatrix}
 \\
 \quad= O\big(|s+1|^{3/2}\E^{-cn|s+1|^{1/2}}\big),\qquad s \in \Sigma^U \setminus (-1,1),
 \\
  0, \qquad s \in \Sigma^U \cap (-1,1),
 \end{cases}
 \end{align*}
 \item \eqref{ParaConv} and \eqref{Qstar} implies that $[\sbQ(z)]^{\pm 1}, [\shQ(z)]^{\pm 1} = O(1)$,
 \item and \eqref{ParaBehaviour} implies that $\big[\bP_-(z)\big]^{\pm 1}, \big[\hP_-(z)\big]^{\pm 1} = O\big(|z+1|^{-1/2}\big)$.
\end{itemize}
Taking the contributions of all factors in \eqref{FinalInt2} into account we see that the integrand can be estimated by $O\big(|z+1|^{1/2}\E^{-cn|z+1|^{1/2}}\big)$, which precisely integrates to the error in \eqref{FinalInt}.
\end{proof}

Propositions~\ref{IntegralEst1} and \ref{IntegralEst2} together with Lemma~\ref{FinalLemma} now readily imply Proposition~\ref{QQ}.

\subsection{Asymptotics of the recurrence coefficients}
Next, we will derive the asymptotics for the recurrence coefficients stated in Theorem \ref{MainTheorem}. This subsection follows the same line of reasoning as \cite[Section~5.2]{CD}.

Recall the result of Corollary~\ref{AsymptoticsHatRC}, which states that the recurrence coefficients of the orthogonal polynomials with weight function $\hw$ satisfy
\begin{align}\label{asymanbn}
\widehat a_n= \frac{1}{4n^2} + O\left(\frac{1}{n^3}\right),\qquad \widehat b_n = \frac{1}{2}-\frac{1}{16n^2} + O\left(\frac{1}{n^3}\right).
\end{align}
With the help of Proposition~\ref{QQ}, we can now expand \eqref{a_n-widehat a_n}:
\begin{align*}
 a_n -\widehat a_n &= \big(Q_1^{(n)}\big)_{11}-\big(\hQ_1^{(n)}\big)_{11} - \big(\big(Q_1^{(n+1)}\big)_{11}-\big(\hQ_1^{(n+1)}\big)_{11}\big)
 \\
 &= \Bigg( \frac{3}{16n \log^2n} \begin{pmatrix}
-1 & \I\\
\phantom{-}\I & 1
\end{pmatrix} + O\left(\frac{1}{n\log^3 n}\right) \Bigg)_{11}
\\&= -\frac{3}{16n^2\log^2 n} + O\left( \frac{1}{n^2\log^3 n}\right).
\end{align*}
Now substituting the asymptotic formula \eqref{asymanbn} for $\widehat a_n$ implies \eqref{FinalEq1}.

The proof of the asymptotic formula \eqref{FinalEq2} is more involved. First recall \cite[Section~8]{KMVV}, in which it is shown that for $z$ away from $\pm 1$, we have
\begin{align} \label{hatQinfty}
 \hQ^{(n)}(z) = \left(\Id + \frac{R_1(z)}{n} + Er(z,n)\right)N(z),
\end{align}
where $R_1(z)$, $Er(z,n)$ are matrix-valued functions, holomorphic for $z \in \Omega_0$ (cf.\ Figure~\ref{lensJumpContour}), satisfying
\begin{align}\label{R1}
 |R_1(z)| \leq \frac{c_1}{|z|}, \qquad |Er(z,n)| \leq \frac{c_2}{|z|n^2}, \qquad z \to \infty
\end{align}
for some $c_1, c_2 > 0$. Importantly, $R_1$ is $n$-independent. As a consequence of \eqref{hatQinfty} and \eqref{R1}, we obtain
\begin{align*}
 \big|\hQ^{(n+1)}(z) - \hQ^{(n)}(z)\big| \lesssim \frac{c_2}{|z|n^2}, \qquad z\to \infty,
\end{align*}
from which
\begin{align}\label{hQDiffFormula}
 \hQ^{(n)}_1 - \hQ^{(n+1)}_1 = O\left(\frac{1}{n^2}\right)
\end{align}
follows. Additionally, direct computation leads to $N_{12}(z) = -\frac{1}{2\I z} + O\big(z^{-2}\big)$ implying
\begin{align}\label{hatQ12}
 \big(\hQ^{(n)}_1\big)_{12} = -\frac{1}{2\I} + O\left(\frac{1}{n}\right).
\end{align}
Now using \eqref{QDiffFormula2}, we conclude from \eqref{hQDiffFormula} that
\begin{align}
 \bQ^{(n)}_1 - \bQ^{(n+1)}_1 &= \hQ^{(n)}_1 - \hQ^{(n+1)}_1 + O\left(\frac{1}{n^2\log^2 n}\right)= O\left(\frac{1}{n^2}\right).\label{Qnn+1}
\end{align}
Regarding formula \eqref{b_n^2-widehat b_n^2}, we now obtain using \eqref{hatQ12} and \eqref{Qnn+1}, together with Proposition~\ref{QQ}:
\begin{align}\nonumber
 b_n^2-\widehat b_n^2={}& \big(\big(\bQ_1^{(n+1)}\big)_{12}-\big(\hQ_1^{(n+1)}\big)_{12}\big)\big(\big(\bQ_1^{(n+1)}\big)_{21}-\big(\bQ_1^{(n+2)}\big)_{21}\big)
 \\\nonumber
 &+ \big( \hQ_1^{(n)}\big)_{12}\big[\big(\big(\bQ_1^{(n+1)}\big)_{21}-\big(\bQ_1^{(n+2)}\big)_{21}\big)-\big(\big(\hQ_1^{(n+1)}\big)_{21}-\big(\hQ_1^{(n+2)}\big)_{21}\big)\big]\nonumber
 \\
={}& \left(\frac{3\I}{16n\log^2 n} + O\left(\frac{1}{n\log^3 n}\right)\right)O\left(\frac{1}{n^2}\right)\nonumber
 \\
 &+\left(-\frac{1}{2\I}+O\left(\frac{1}{n}\right)\right) \left(\frac{3\I}{16n^2\log^2 n} + O\left(\frac{1}{n^2\log^3 n}\right)\right)\nonumber
 \\
={}& -\frac{3}{32n^2\log^2 n} + O\left(\frac{1}{n^2\log^3 n}\right).\label{bnlong}
\end{align}
Moreover, we have
\begin{align}\nonumber
 b_n^2-\widehat b_n^2 = \big(b_n-\widehat b_n\big)\big(b_n+\widehat b_n\big) &= \big(b_n-\widehat b_n\big)\big(2\widehat b_n + b_n-\widehat b_n\big)
 \\\label{bnshort}
 &=\big(b_n-\widehat b_n\big)\big(1 + b_n-\widehat b_n + O\big(n^{-2}\big)\big)
\end{align}
As $\widehat b_n = \frac{1}{2} + O\big(n^{-2}\big)$, see Corollary~\ref{AsymptoticsHatRC}, and $b_n > 0$, we have that $\big(1 + b_n-\widehat b_n + O\big(n^{-2}\big)\big)>1/2-\epsilon$, which with \eqref{bnlong} implies that
\begin{align}\label{bnveryshort}
 b_n-\widehat b_n = O\left(\frac{1}{n^2\log^2 n}\right).
\end{align}
Substituting \eqref{bnveryshort} once again into the term $\big(1 + b_n-\widehat b_n + O\big(n^{-2}\big)\big)$, we conclude from \eqref{bnshort} that in fact
\begin{align*}
 b_n^2-\widehat b_n^2 = \big(b_n-\widehat b_n\big)\big(1+O\big(n^{-2}\big)\big),
\end{align*}
which with \eqref{bnlong} implies
\begin{align*}
 b_n-\widehat b_n = -\frac{3}{32n^2\log^2 n} + O\left(\frac{1}{n^2\log^3 n}\right).
\end{align*}
That together with the asymptotic formula \eqref{asymanbn} implies \eqref{FinalEq2} finishing the proof of Theorem~\ref{MainTheorem}.

\appendix
\section{Proofs of certain propositions}\label{appendixA}

The following appendix contains proofs of Propositions~\ref{Proprr} and~\ref{IntegralEst2}. These differ in certain details from the analogous proofs in \cite{CD}, hence are included here.
\begin{Proposition} \label{ProprrAp}
Fix $R > 0$. Let $r_n, \tilde r_n \in (0,1)$, $n \in \N$ be two sequences satisfying $r_n, \tilde r_n \to 0$, such that $n\big|\frac{r_n}{\tilde r_n}-1|< R$. Then
\begin{align*}
 &\frac{F^2}{w_+}(1+r_n)-\frac{F^2}{w_+}(1+\tilde r_n)+\frac{F^2}{w_-}(1+r_n)-\frac{F^2}{w_-}(1+\tilde r_n)
 \\
 &\qquad= O(r_n \log |\log r_n|) + O\left(\frac{1}{n\log^3 r_n}\right) + O\left(\frac{1}{n^2}\right),
\end{align*}
where the implied constants in the $O$-terms depend only on $R$.
\end{Proposition}

\begin{proof}
Let us assume without loss of generality $r_n \geq \tilde r_n$. It follows from definition \eqref{SzegoF} that
\begin{align*}
 \log F(1+r_n) &= \big((1+r_n)^2-1\big)^{1/2} \frac{1}{2\pi \I}\int_{-1}^1 \frac{\log w(s)}{\big(s^2-1\big)_+^{1/2}} \frac{{\rm d}s}{s-(1+r_n)}
 \\
 &=\big(\sqrt{2} r_n^{1/2}+O\big(r_n^{3/2}\big)\big) \frac{1}{2\pi \I}\int_{-1}^1 \frac{\log w(s)}{\big(s^2-1\big)_+^{1/2}} \frac{{\rm d}s}{s-(1+r_n)}.
\end{align*}
Moreover, by \cite[equation~A.11]{CD}, we have
\begin{align*}
 \int_{-1}^1 \frac{\log w(s)}{\big(s^2-1\big)_+^{1/2}} \frac{{\rm d}s}{s-(1+r_n)} = O\left(\frac{\log |\log r_n|}{r_n^{1/2}}\right),
\end{align*}
implying that, in fact,
\begin{align*}
 \log F(1+r_n) = \sqrt{2} r_n^{1/2} \frac{1}{2\pi \I}\int_{-1}^1 \frac{\log w(s)}{\big(s^2-1\big)_+^{1/2}} \frac{{\rm d}s}{s-(1+r_n)}+O(r_n \log |\log r_n|).
\end{align*}
In particular, we obtain
\begin{align} \nonumber
 \log \frac{F(1+r_n)}{F(1+\tilde r_n)}={}&
 \sqrt{2} r_n^{1/2} \frac{1}{2\pi \I}\int_{0}^1 \frac{\log w(s)}{\big(s^2-1\big)_+^{1/2}} \frac{{\rm d}s}{s-(1+r_n)}\\
 &-\sqrt{2} \tilde r_n^{1/2} \frac{1}{2\pi \I}\int_{0}^1 \frac{\log w(s)}{\big(s^2-1\big)_+^{1/2}} \frac{{\rm d}s}{s-(1+\tilde r_n)} \nonumber
 \\
 &+\frac{1}{2\pi \I}\int_{-1}^0 \frac{\log w(s)}{\big(s^2-1\big)_+^{1/2}} \Bigg[\frac{\sqrt{2} r_n^{1/2}}{s-(1+r_n)}-\frac{\sqrt{2} \tilde r_n^{1/2}}{s-(1+\tilde r_n)}\Bigg] {\rm d}s \nonumber \\
 &+O(r_n \log |\log r_n|).\label{FF1}
\end{align}
The term in the bracket can be estimated as follow:
\begin{align*}
 \Bigg|\frac{\sqrt{2} r_n^{1/2}}{s-(1+r_n)}-\frac{\sqrt{2} \tilde r_n^{1/2}}{s-(1+\tilde r_n)}\Bigg| &= \Bigg|\sqrt{2} \frac{r_n^{1/2}(s-(1+\tilde r_n))-\tilde r_n^{1/2}(s-(1+ r_n))}{(s-(1+r_n))(s-(1+\tilde r_n))}\Bigg|
 \\
 &= \Bigg|\sqrt{2} \frac{\big(s-1+r_n^{1/2}\tilde r_n^{1/2}\big)\big(r_n^{1/2}-\tilde r_n^{1/2}\big)}{(s-(1+r_n))(s-(1+\tilde r_n))}\Bigg|
 \\
 &= \underbrace{\Bigg|\sqrt{2} \frac{\big(s-1+r_n^{1/2}\tilde r_n^{1/2}\big) \tilde r_n^{1/2}}{(s-(1+r_n))(s-(1+\tilde r_n))}\Bigg|}_{= O\big(r_n^{1/2}\big)} \cdot \underbrace{\left|\left(\frac{r_n}{\tilde r_n}\right)^{1/2} -1 \right|}_{< \frac{R}{n}}
 \\
 &= O\left(\frac{r_n^{1/2}}{n}\right),
\end{align*}
where we used that $s\in (-1,0)$. Thus \eqref{FF1} can be rewritten as
\begin{align*}
 \log \frac{F(1+r_n)}{F(1+\tilde r_n)}={}&
 \sqrt{2} r_n^{1/2} \frac{1}{2\pi \I}\int_{0}^1 \frac{\log w(s)}{\big(s^2-1\big)_+^{1/2}} \frac{{\rm d}s}{s-(1+r_n)}\\
 &-\sqrt{2} \tilde r_n^{1/2} \frac{1}{2\pi \I}\int_{0}^1 \frac{\log w(s)}{\big(s^2-1\big)_+^{1/2}} \frac{{\rm d}s}{s-(1+\tilde r_n)}
\\&+O(r_n \log |\log r_n|) + O\left(\frac{r_n^{1/2}}{n}\right),
\end{align*}
Now, after the change of variables $t = 1- s$ and some algebraic manipulation, we get (cf.\ proof of \cite[Proposition~A.4]{CD})
\begin{align}\nonumber
 \log \frac{F(1+r_n)}{F(1+\tilde r_n)}={}& \frac{r_n^{1/2}}{2\pi} \int_0^1 \frac{\log w(1-t)}{\sqrt{t}} \frac{{\rm d}t}{t+r_n} - \frac{\tilde r_n^{1/2}}{2\pi} \int_0^1 \frac{\log w(1-t)}{\sqrt{t}} \frac{{\rm d}t}{t+\tilde r_n} \\\label{logFF}
 &+ \frac{1}{2\pi} H(r_n, \tilde r_n) +O(r_n \log |\log r_n|) + O\left(\frac{r_n^{1/2}}{n}\right),
\end{align}
where
\begin{align*}
 H(r_n, \tilde r_n)={}& \big(r_n^{1/2}-\tilde r_n^{1/2}\big)\int_0^1 \log w(1-t) \left[\frac{\sqrt{t}}{\sqrt{2}\sqrt{2-t}+(2-t)} \right]\frac{{\rm d}t}{t+r_n}
 \\
 &+ \tilde r_n^{1/2} \int_0^1 \log w(1-t) \left[\frac{\sqrt{t}}{\sqrt{2}\sqrt{2-t}+(2-t)} \right]\left(\frac{1}{t+r_n}-\frac{1}{t+\tilde r_n}\right){\rm d}t.
\end{align*}
Thus, we can estimate
\begin{align*}
 |H(r_n, \tilde r_n)|\leq{}& \big|r_n^{1/2}-\tilde r_n^{1/2}\big| \int_0^1 |\log w(1-t)|\left[\frac{\sqrt{t}}{\sqrt{2}\sqrt{2-t}+(2-t)} \right]\frac{{\rm d}t}{t}
 \\
 &+\tilde r_n^{1/2} \int_0^1 |\log w(1-t)|\left[\frac{\sqrt{t}}{\sqrt{2}\sqrt{2-t}+(2-t)} \right]\frac{|\tilde r_n - r_n|{\rm d}t}{|t+r_n|\cdot |t+\tilde r_n|}
 \\
\leq{}& c \tilde r_n^{1/2} \left|\frac{r_n^{1/2}}{\tilde r_n^{1/2}}-1\right| + \tilde r_n^{1/2}\left|\frac{\tilde r_n - r_n}{\tilde r_n}\right| \int_0^1 |\log w(1-t)|\left[\frac{\sqrt{t}}{\sqrt{2}\sqrt{2-t}+(2-t)} \right]\frac{{\rm d}t}{t}
 \\
\leq{}& c \tilde r_n^{1/2} \underbrace{\left|\frac{r_n^{1/2}}{\tilde r_n^{1/2}}-1\right|}_{<\frac{R}{n}} + c \tilde r_n^{1/2} \underbrace{\left|\frac{r_n}{\tilde r_n}-1\right|}_{<\frac{R}{n}} = O\left(\frac{r_n^{1/2}}{n}\right).
\end{align*}
So, we see that $H(r_n, \tilde r_n)$ can be included in the error term $O\big(\frac{r_n^{1/2}}{n}\big)$.

For the remaining integrals in \eqref{logFF}, we obtain after performing the change of variables $t \to r_n t$ and $t \to \tilde r_n t$ in the first and second integral, respectively,
\begin{align}\nonumber
 &r_n^{1/2}\int_0^1 \frac{\log w(1-t)}{\sqrt{t}} \frac{{\rm d}t}{t+r_n} - \tilde r_n^{1/2} \int_0^1 \frac{\log w(1-t)}{\sqrt{t}} \frac{{\rm d}t}{t+\tilde r_n}
 \\
 &\qquad= \int_0^{1/r_n} \frac{\log w(1-r_n t)}{\sqrt{t}} \frac{{\rm d}t}{t+1} - \int_0^{1/\tilde r_n} \frac{\log w(1-\tilde r_n t)}{\sqrt{t}} \frac{{\rm d}t}{t+1}\nonumber
 \\
&\qquad = \int_0^{1/r_n}\frac{\big[\log w(1-r_n t) - \log w(1-\tilde r_n t)\big]}{\sqrt{t}} \frac{{\rm d}t}{t+1}-\int_{1/ r_n}^{1/\tilde r_n} \frac{\log w(1-\tilde r_n t)}{\sqrt{t}} \frac{{\rm d}t}{t+1}.\label{t_to_rt}
\end{align}
Using the fact that $|\log w(1-\tilde r_n t)|$ is uniformly bounded for $t \in \big[\frac{1}{r_n}, \frac{1}{\tilde r_n}\big]$, the last integral in~\eqref{t_to_rt} can be estimated via
\begin{align*}
 \left| \int_{1/r_n}^{1/\tilde r_n} \frac{\log w(1-\tilde r_n t)}{\sqrt{t}} \frac{{\rm d}t}{t+1} \right| &\leq \left \Vert \frac{1}{\sqrt{t}} \frac{1}{t+1} \right \Vert_{L^\infty(1/r_n, 1/\tilde r_n)} \int_{1/ r_n}^{1/\tilde r_n} |\log w(1-\tilde r_n t)| {\rm d}t \\
 &\leq c r_n^{3/2} \left| \frac{1}{\tilde r_n} - \frac{1}{r_n}\right| = c r_n^{1/2} \left| \frac{r_n}{\tilde r_n}-1\right| = O\left(\frac{r_n^{1/2}}{n}\right),
\end{align*}
and hence can be again included in the $O\big(\frac{r_n^{1/2}}{ n}\big)$-term in \eqref{logFF}.

Let us now consider the remaining integral in the last line of \eqref{t_to_rt}:
\begin{align}\nonumber
& \int_0^{1/r_n} \big[\log w(1-r_n t)- \log w(1-\tilde r_n t)\big] \frac{{\rm d}t}{t^{3/2}+t^{1/2}}
 \\\label{logDiff}
 &\qquad= \int_0^{1/r_n} \log \left(\frac{w(1-r_n t)}{w(1-\tilde r_n t)}\right) \frac{{\rm d}t}{t^{3/2}+t^{1/2}}
\end{align}
Now define $a = a(r_n, \tilde r_n ; n) := n\big(\frac{r_n}{\tilde r_n}-1\big) \in [-R, R]$. Note that for $t \in \big(0, \frac{1}{r_n}\big)$, we have $\log 2 \leq \log \frac{2}{\tilde r_n t}$, hence
\begin{align*}
 \frac{w(1-r_n t)}{w(1-\tilde r_n t)} &= \frac{\log \frac{2}{r_nt}}{\log \frac{2}{\tilde r_nt}} = 1 + \frac{\log \frac{2}{r_nt}-\log \frac{2}{\tilde r_nt}}{\log \frac{2}{\tilde r_nt}} =1 + \frac{\log \frac{\tilde r_n}{r_n}}{\log \frac{2}{\tilde r_n t}}
 \\
 &= 1 + \frac{\log (1+\frac{a}{n})}{\log \frac{2}{\tilde r_n t}}= 1 + \frac{a}{n\log \frac{2}{\tilde r_n t}} + O\left(\frac{1}{n^2}\right).
\end{align*}
Thus, we can estimate the integrand of \eqref{logDiff} by
\begin{align*}
 \log \frac{w(1-r_n t)}{w(1-\tilde r_n t)} = \frac{a}{n\log \frac{2}{\tilde r_n t}} + O\left(\frac{1}{n^2}\right),
\end{align*}
where the $O\big(\frac{1}{n^2}\big)$-term is uniform for $t\in\big[0, \frac{1}{r_n}\big]$. Substituting this into \eqref{logDiff}, we obtain
\begin{align}\nonumber
 &\int_0^{1/r_n} \left(\frac{a}{n\log \frac{2}{\tilde r_n t}} + O\left(\frac{1}{n^2}\right)\right) \frac{{\rm d}t}{t^{3/2}+t^{1/2}}
 \\\label{PartitionThreeIntegrals}
 &\qquad= \left(\int_{0}^{r_n^{1/2}}+\int_{r_n^{1/2}}^{1/r_n^{1/2}}+\int_{1/r_n^{1/2}}^{1/r_n}\right) \frac{a}{n\log \frac{2}{\tilde r_n t}} \frac{{\rm d}t}{t^{3/2}+t^{1/2}}+O\left(\frac{1}{n^2}\right).
\end{align}
Note that here we use the assumption $r_n < 1$.
Two of the integrals in \eqref{PartitionThreeIntegrals} can be estimated by
\begin{align*}
 \left|\int_{0}^{r_n^{1/2}} \frac{a}{n\log \frac{2}{\tilde r_n t}} \frac{{\rm d}t}{t^{3/2}+t^{1/2}} \right| \leq \left|\int_{0}^{r_n^{1/2}} \frac{c_1}{n \log r_n} \frac{{\rm d}t}{t^{1/2}}\right| = \frac{2c_1 r_n^{1/4}}{n |\log r_n|} = O\left(\frac{r_n^{1/4}}{n \log r_n}\right)
\end{align*}
and
\begin{align*}
 \left|\int_{1/r_n^{1/2}}^{1/r_n} \frac{a}{n\log \frac{2}{\tilde r_n t}} \frac{{\rm d}t}{t^{3/2}+t^{1/2}} \right| \leq \left|\int_{1/r_n^{1/2}}^{1/r_n} \frac{c_2}{n} \frac{{\rm d}t}{t^{3/2}}\right| \leq \frac{c_2 r_n^{1/4}}{n} = O\left(\frac{r_n^{1/4}}{n}\right).
\end{align*}
For the remaining integral, we have
\begin{align}\nonumber
 \int_{r_n^{1/2}}^{1/r_n^{1/2}} \frac{a}{n\log \frac{2}{\tilde r_n t}} \frac{{\rm d}t}{t^{3/2}+t^{1/2}} &=
 \int_{r_n^{1/2}}^{1/r_n^{1/2}} \frac{a}{n\log \frac{2}{\tilde r_n }-n \log t} \frac{{\rm d}t}{t^{3/2}+t^{1/2}}
 \\\label{iteratedFractions}
 &= \int_{r_n^{1/2}}^{1/r_n^{1/2}} \frac{a}{n\log \frac{2}{\tilde r_n }} \Bigg(\frac{1}{1-\frac{\log t}{\log \frac{2}{\tilde r_n}}} \Bigg) \frac{{\rm d}t}{t^{3/2}+t^{1/2}}.
\end{align}
Note that because $t \in \big[r_n^{1/2}, r_n^{-1/2}\big]$ we have $|\log t| \leq \frac{1}{2} \log{\frac{1}{r_n}} < \frac{1}{2} \log{\frac{2}{\tilde r_n}}$ for $n$ sufficiently large depending on $R$, so the last integral in \eqref{iteratedFractions} can be estimated by
\begin{align*}
 \frac{a}{n\log \frac{2}{\tilde r_n }} \int_{r_n^{1/2}}^{1/r_n^{1/2}} \left(1+\frac{\log t}{\log \frac{2}{\tilde r_n}} + O \left( \frac{\log^2 t}{\log^2 \frac{2}{\tilde r_n}} \right)\right) \frac{{\rm d}t}{t^{3/2}+t^{1/2}}.
\end{align*}
Making the change of variables $\gamma = t^{1/2}$, this can be rewritten as
\begin{align*}
 &\frac{2a}{n\log \frac{2}{\tilde r_n }} \int_{r_n^{1/4}}^{1/r_n^{1/4}} \left(1+\frac{2\log \gamma}{\log \frac{2}{\tilde r_n}} + O \left( \frac{\log^2 \gamma}{\log^2 \frac{2}{\tilde r_n}} \right)\right) \frac{{\rm d}\gamma}{\gamma^{2}+1}
 \\
 &\qquad= \frac{2a}{n\log \frac{2}{\tilde r_n }} \left(\int_{r_n^{1/4}}^{1/r_n^{1/4}} \frac{{\rm d}\gamma}{\gamma^{2}+1} + \int_{r_n^{1/4}}^{1/r_n^{1/4}} \frac{2\log \gamma}{\log \frac{2}{\tilde r_n}} \frac{{\rm d}\gamma}{\gamma^{2}+1} \right)+ O\left(\frac{1}{n \log^3 r_n}\right)
 \\
 &\qquad= \frac{2a}{n\log \frac{2}{\tilde r_n }} \left(\int_{0}^{\infty} \frac{{\rm d}\gamma}{\gamma^{2}+1} + \int_{0}^{\infty} \frac{2\log \gamma}{\log \frac{2}{\tilde r_n}} \frac{{\rm d}\gamma}{\gamma^{2}+1}\right) + O\left(\frac{1}{n \log^3 r_n}\right).
\end{align*}
Note that
\begin{align*}
 \int_{0}^{\infty} \frac{{\rm d}\gamma}{\gamma^{2}+1} = \arctan \gamma \Big|^\infty_0 = \frac{\pi}{2},
\end{align*}
while the substitution $\eta = \gamma^{-1}$ yields
\begin{align*}
 \int_{0}^{\infty} \frac{\log(\gamma){\rm d}\gamma}{\gamma^{2}+1} = - \int_{0}^{\infty} \frac{\log(\eta){\rm d}\eta}{\eta^{2}+1}
\end{align*}
implying $\int_{0}^{\infty} \frac{\log(\gamma){\rm d}\gamma}{\gamma^{2}+1} = 0$.

Summarizing, we have shown that
\begin{align*}
& \int_0^{1/r_n} \big[\log w(1-r_n t)- \log w(1-\tilde r_n t)\big] \frac{{\rm d}t}{t^{3/2}+t^{1/2}}
 \\
 &\qquad= \frac{a\pi}{n\log \big(\frac{2}{\tilde r_n }\big)} + O\left(\frac{1}{n^2}\right) + O\left(\frac{1}{n \log^3 r_n} \right).
\end{align*}
We can substitute this estimate in \eqref{t_to_rt} and then \eqref{logFF} to obtain
\begin{align}\label{ratioFsquared}
 \log\frac{F^2(1+r_n)}{F^2(1+\tilde r_n)} = 2 \log \frac{F(1+r_n)}{F(1+\tilde r_n)} = \frac{a}{ n\log \frac{2}{\tilde r_n} } + \Theta(r_n, n),
\end{align}
where $\Theta(r_n, n)$ is short for $O(r_n \log |\log r_n|) + O\big(\frac{1}{n\log^3 r_n}\big) +O\big(\frac{1}{n^2}\big)$. Exponentiating the expression~\eqref{ratioFsquared} leads to
\begin{align*}
 \frac{F^2(1+r_n)}{F^2(1+\tilde r_n)} = 1 + \frac{a}{ n\log \frac{2}{\tilde r_n} } + \Theta(r_n, n).
\end{align*}
Moreover,
\begin{align}
& \frac{F^2}{w_\pm}(1+r_n) - \frac{F^2}{w_\pm}(1+\tilde r_n)
\nonumber \\
 &\qquad = \frac{F^2(1+r_n)-F^2(1+\tilde r_n)}{w_\pm(1+r_n)} + F^2(1+\tilde r_n)\left(\frac{1}{w_\pm(1+r_n)}-\frac{1}{w_\pm(1+\tilde r_n)}\right)
\nonumber \\
 &\qquad = \frac{F^2(1+\tilde r_n)}{w_\pm(1+r_n)}\left(\frac{F^2(1+r_n)}{F^2(1+\tilde r_n)}-1\right) + \frac{F^2(1+\tilde r_n)}{w_\pm(1+r_n)}\left(1-\frac{w_\pm(1+r_n)}{w_\pm(1+\tilde r_n)}\right)
\nonumber \\
 &\qquad = \frac{F^2(1+\tilde r_n)}{w_\pm(1+r_n)}\left(\frac{a}{n\log \frac{2}{\tilde r_n}} + \Theta(r_n, n)+ 1-\frac{w_\pm(1+r_n)}{w_\pm(1+\tilde r_n)}\right).\label{FTheta}
\end{align}
For the ratio of the weight functions, we have (here $\pm$ refers to the limit from $\C_\pm$):
\begin{align*}
 \frac{w_\pm(1+r_n)}{w_\pm(1+\tilde r_n)} &= \frac{\log \frac{2}{r_n}\pm \pi \I}{\log \frac{2}{\tilde r_n}\pm \pi \I} = 1 + \frac{\log \frac{2}{r_n}-\log\frac{2}{\tilde r_n}}{\log \frac{2}{\tilde r_n}\pm \pi \I}
 \\
 &= 1+ \frac{\log \frac{\tilde r_n}{r_n}}{\log \frac{2}{\tilde r_n}\pm \pi \I} = 1 + \frac{a}{n\big(\log \frac{2}{\tilde r_n}\pm \pi \I\big)}+ O\left(\frac{1}{n^2}\right),
\end{align*}
so
\begin{align}\nonumber
 \frac{a}{ n\log \frac{2}{\tilde r_n}} + 1-\frac{w_\pm(1+r_n)}{w_\pm(1+\tilde r_n)} &= \frac{a}{ n\log \frac{2}{\tilde r_n}} - \frac{a}{n\big(\log \frac{2}{\tilde r_n}\pm \pi \I\big)} + O\left(\frac{1}{n^2}\right)
 \\\nonumber
 &= \pm \frac{\pi\I a }{n \log \frac{2}{\tilde r_n}\big(\log \frac{2}{\tilde r_n}\pm \pi \I\big)} + O\left(\frac{1}{n^2}\right)
 \\\label{wratio}
 &= \pm\frac{\pi\I a}{n \log^2\frac{2}{\tilde r_n}} + O\left(\frac{1}{n \log^3 r_n}\right)+ O\left(\frac{1}{n^2}\right).
\end{align}
Note that both error terms in \eqref{wratio} already appear in $\Theta(r_n, n)$. Substituting \eqref{wratio} into \eqref{FTheta} and using \eqref{F^2/w at +1} leads to
\begin{align*}
 \frac{F^2}{w_\pm}(1+r_n) - \frac{F^2}{w_\pm}(1+\tilde r_n) &= \frac{F^2(1+\tilde r_n)}{w_\pm(1+r_n)}\left( \pm \frac{\pi\I a}{n \log^2\frac{2}{\tilde r_n}} + \Theta(r_n, n)\right)
 \\
 &= \left(1+ O\left(\frac{1}{\log r_n}\right)\right)\left(\pm\frac{\pi\I a}{n \log^2\frac{2}{\tilde r_n}} + \Theta(r_n, n)\right)
 \\
 &= \pm\frac{\pi\I a}{n \log^2\frac{2}{\tilde r_n}} + \Theta(r_n, n).
\end{align*}
In particular,
\begin{align*}
 &\frac{F^2}{w_+}(1+r_n)-\frac{F^2}{w_+}(1+\tilde r_n)+\frac{F^2}{w_-}(1+r_n)-\frac{F^2}{w_-}(1+\tilde r_n)
 \\
 &\qquad= \Theta(r_n ,n)= O(r_n \log |\log r_n|) + O\left(\frac{1}{n\log^3 r_n}\right) + O\left(\frac{1}{n^2}\right),
\end{align*}
finishing the proof.
\end{proof}

\begin{Proposition}\label{IntegralEst2Ap}
The following estimates hold:
\begin{align}\label{Int}
\frac{1}{2\pi\I} \int_{\Sigr} \hmu_n(\bv_n-\hv_n)\hmu_n^{-1} {\rm d}s = \frac{3}{16n \log^2n} \begin{pmatrix}
\phantom{-}1 & -\I\\
-\I & -1
\end{pmatrix} + O\left(\frac{1}{n\log^3 n}\right)
\end{align}
and
\begin{align}\nonumber
 & \frac{1}{2\pi\I} \int_{\Sigr} \hmu_n(\bv_n-\hv_n)\hmu_n^{-1} {\rm d}s - \frac{1}{2\pi\I} \int_{\Sigr} \hmu_{n+1}(\bv_{n+1}-\hv_{n+1})\hmu^{-1}_{n+1} {\rm d}s
 \\\label{differenceInt}
 &\qquad= \frac{3}{16n^2 \log^2n} \begin{pmatrix}
\phantom{-}1 & -\I\\
-\I & -1
\end{pmatrix} + O\left(\frac{1}{n^2\log^3 n}\right).
\end{align}
\end{Proposition}
\begin{proof}
 Following \cite[Proposition~C.3]{CD}, one can show that the contributions to the integrals from the contour $\Sigr \setminus (1,1+1/n)$ are exponentially small for $n \to \infty$. Thus, we see that for $n$ large enough, up to an exponentially small error, the integral in \eqref{Int} takes on the form
\begin{align}\label{localizedIntegral}
 - \frac{1}{2\pi \I}\int_1^{1+1/n} \hmu_n(\bv_n-\hv_n)\hmu_n^{-1} {\rm d}s,
\end{align}
as $\Sigr \cap (1,1+1/n)= (1,1+1/n)$ is oriented right to left. Hence, we need to consider the behaviour of $\hmu_n$ near the point $+1$. In fact, as $\bv_n-\hv_n$ is nonzero only in the $21$-entry, see \eqref{jumpMatrixV} and \eqref{jumpMatrixhV}, we just need to consider the second column of $\hmu_n$, which we will denote by $\hmu_{2,n}$:
\begin{align*}
 \hmu_{2,n} = \begin{pmatrix}
 \hmu_{12,n}
 \\
 \hmu_{22,n}
 \end{pmatrix}.
\end{align*}
It follows from \cite[Section~6, 7]{KMVV} (cf.\ \cite[equation~B.12]{CD}), that for $s \in (1, 1+1/n)$, $\hmu_{2,n}(s)$ takes on the form
\begin{align}\label{explicithmu}
 \hmu_{2,n}(s) = \cR^{(n)}(s)
 \cE(s)(2\pi n)^{\sigma_3/2} \mathcal K\big(n^2 f(s)\big)\phi^n(s)\Big(\frac{\hF}{\cW}(s)\Big)^{-1}.
\end{align}
Here
\begin{align*}
 \mathcal K(\zeta) = \begin{pmatrix}
 \cK_1(\zeta)
 \\
 \cK_2(\zeta)
 \end{pmatrix} = \begin{pmatrix}
 \frac{\I}{\pi}K_0\big(2\zeta^{1/2}\big)
 \\
 -2\zeta^{1/2} K_0'\big(2\zeta^{1/2}\big)
 \end{pmatrix},
\end{align*}
where $K_0$ is a special solution to the modified Bessel differential equation, see \cite[Section~10.25]{DLMF}, characterized by the condition
\begin{align*}
 K_0(u) \sim \sqrt{\frac{\pi}{2u}}\E^{-u} \qquad \text{for} \quad u \to \infty \quad \text{with} \quad |\arg u | < \frac{3\pi}{2} - \varepsilon, \quad \varepsilon > 0,
\end{align*}
and $f(z) = \frac{\log^2 \phi(z)}{4}$ locally around $z = +1$. The matrix-valued function $\cR^{(n)}$ is holomorphic in a fixed neighbourhood $U_{+1}$, of $+1$, where it satisfies uniformly
\begin{align*}
\cR^{(n)}(z) = \Id + O\left(\frac{1}{n} \right), \qquad z \in U_{+1},\quad n \to \infty.
\end{align*}
For $\cE$, we have
\begin{align*}
 \cE(z) = N(z)\left(\frac{ \cW}{\hF}(z)\right)^{\sigma_3}\frac{1}{\sqrt{2}}
 \begin{pmatrix}
 \phantom{-} 1 & -\I
 \\
 -\I & \phantom{-} 1
 \end{pmatrix}f(z)^{\sigma_3/4}, \qquad z \in U_{+1},
\end{align*}
where $\cW = \sqrt{\hw}$ is holomorphic in $U_{+1}$ as $\hw$ is nonvanishing close to $+1$ and $N$ is taken from~\eqref{OuterParametrix}. One sees easily that $\cE$ is holomorphic in $U_{+1}$ and satisfies
\begin{align*}
 \cE(1) = \frac{1}{\sqrt{2}}
 \begin{pmatrix}
 \phantom{-} 1 & *
 \\
 -\I & *
 \end{pmatrix},
\end{align*}
cf.\ \cite[Proposition~C.2]{CD} and \eqref{hat F^2/w at +1}. Using the explicit form of $\hmu_{2,n}$ in \eqref{explicithmu} and abbreviating $\cK_j = \cK_j\big(n^2 f(s)\big)$, $j = 1,2$, the integral in \eqref{localizedIntegral} without the $\frac{1}{2\pi\I}$-prefactor can be written as
\begin{align}
 &-\int_1^{1+1/n} \hmu_n(\bv_n-\hv_n)\hmu_n^{-1} {\rm d}s\nonumber
 \\
 &\qquad= -\int_1^{1+1/n} \cR^{(n)}(s)\cE(s)(2\pi n)^{\sigma_3/2}
 \begin{pmatrix}
 * & \cK_1
 \\
 * & \cK_2
 \end{pmatrix}
 \phi^{-n\sigma_3}(s)\left(\frac{\hF}{\cW}(s)\right)^{\sigma_3}\nonumber
 \\
 &\phantom{\qquad=}{}\times \begin{pmatrix}
 0 & 0
 \\
 \big(\frac{F^2}{w_+}(s)+\frac{F^2}{w_-}(s)-2\hFw(s)\big) \phi^{-2n}(s) & 0
 \end{pmatrix}
 \left(\frac{\hF}{\cW}(s)\right)^{-\sigma_3}\phi^{n\sigma_3}(s)\nonumber
 \\
 &\phantom{\qquad=}{}\times \begin{pmatrix}
 \cK_2 & -\cK_1
 \\
 * & *
 \end{pmatrix}(2\pi n)^{-\sigma_3/2}\cE^{-1}(s) \big[\cR^{(n)}(s)\big]^{-1} {\rm d}s\nonumber
 \\ &\qquad= - \int_1^{1+1/n} \left( \Id + O\left(\frac{1}{n}\right)\right)\cE(s)(2\pi n)^{\sigma_3/2}
 \begin{pmatrix}
 * & \cK_1
 \\
 * & \cK_2
 \end{pmatrix}
 \begin{pmatrix}
 0 & 0
 \\
 1 & 0
 \end{pmatrix} \begin{pmatrix}
 \cK_2 & -\cK_1
 \\
 * & *
 \end{pmatrix}\nonumber
 \\
 &\phantom{\qquad=}{}\times (2\pi n)^{-\sigma_3/2}\cE^{-1}(s) \left( \Id + O\left(\frac{1}{n}\right)\right) \left(\frac{\cW}{\hF}(s)\right)^2\left(\frac{F^2}{w_+}(s)+\frac{F^2}{w_-}(s)-2\hFw(s)\right) {\rm d}s\nonumber
 \\
 &\qquad= -\int_1^{1+1/n} \left( \cE(1) + O\left( \frac{1}{n} \right)\right)(2\pi n)^{\sigma_3/2}
 \begin{pmatrix}
 \cK_1\cK_2 & -\cK^2_1
 \\
 \cK^2_2 & -\cK_1\cK_2
 \end{pmatrix}\nonumber
 \\
 &\phantom{\qquad=}{}\times (2\pi n)^{-\sigma_3/2} \left( \cE^{-1}(1) + O\left( \frac{1}{n} \right)\right)\left(\frac{\cW}{\hF}(s)\right)^2\left(\frac{F^2}{w_+}(s)+\frac{F^2}{w_-}(s)-2\hFw(s)\right) {\rm d}s\nonumber
 \\
 &\qquad=- 2\pi n \int_1^{1+1/n} \left( \cE(1) + O\left( \frac{1}{n} \right)\right)
 \begin{pmatrix}
 1 & 0
 \\
 0 & \frac{1}{2\pi n}
 \end{pmatrix}
 \begin{pmatrix}
 \cK_1\cK_2 & -\cK^2_1
 \\
 \cK^2_2 & -\cK_1\cK_2
 \end{pmatrix} \label{veryLongEq}\\
 &\phantom{\qquad=}{}\times
 \begin{pmatrix}
 \frac{1}{2\pi n} & 0
 \\
 0 & 1
 \end{pmatrix}\left( \cE^{-1}(1) + O\left( \frac{1}{n} \right)\right)\left(\frac{\cW}{\hF}(s)\right)^2\left(\frac{F^2}{w_+}(s)+\frac{F^2}{w_-}(s)-2\hFw(s)\right) {\rm d}s.\nonumber
\end{align}
Here we used that the matrix \smash{$\left(\begin{smallmatrix}
* & \cK_1
\\
* & \cK_2
\end{smallmatrix}\right)$} has determinant equal to $1$, cf.\ \cite[Remark 7.1]{KMVV}.
Note that all $O\big(\frac{1}{n}\big)$-terms in \eqref{veryLongEq} for $s \in (1,1+\frac{1}{n})$ are bounded by $\frac{c}{n}$, where $c > 0$ is fixed. Hence, we obtain using \eqref{hat F^2/w at +1} and \eqref{Diff F^2/w at +1} in the second last line
\begin{align}
 &-\int_1^{1+1/n} \hmu_n(\bv_n-\hv_n)\hmu_n^{-1} {\rm d}s\nonumber
 \\
 &\qquad=-2\pi n \int_1^{1+1/n} \cE(1) \begin{pmatrix}
 1 & 0
 \\
 0 & 0
 \end{pmatrix}
 \begin{pmatrix}
 \cK_1\cK_2 & -\cK^2_1
 \\
 \cK^2_2 & -\cK_1\cK_2
 \end{pmatrix}
 \begin{pmatrix}
 0 & 0
 \\
 0 & 1
 \end{pmatrix} \cE^{-1}(1)\nonumber
 \\
 &\phantom{\qquad=}{}\times \left(\frac{\cW}{\hF}(s)\right)^2\left(\frac{F^2}{w_+}(s)+\frac{F^2}{w_-}(s)-2\hFw(s)\right) {\rm d}s\nonumber
 \\
 &\phantom{\qquad=\times}{}+ \underbrace{O \left\Vert \begin{pmatrix}
 \cK_1\cK_2 & -\cK^2_1
 \\
 \cK^2_2 & -\cK_1\cK_2
 \end{pmatrix} \left(- \frac{3\pi^2}{\log^2 \frac{2}{x-1}} + O\left(\frac{1}{\log^3(x-1)}\right)\right)\right\Vert_{L^1(1,1+1/n)}}_{=O\big(\frac{1}{n^2\log^2 n}\big), \ \text{by arguments as in \cite[Proposition~C.1]{CD}}}\nonumber
 \\
 &\qquad=- 2\pi n \int_1^{1+1/n} \cE(1)
 \begin{pmatrix}
 0 & -\cK^2_1
 \\
 0 & 0
 \end{pmatrix}
 \cE^{-1}(1)
 \left(\frac{\cW}{\hF}(s)\right)^2\left(\frac{F^2}{w_+}(s)+\frac{F^2}{w_-}(s)-2\hFw(s)\right) {\rm d}s\nonumber
 \\
 &\phantom{\qquad=\times}{}+ O\left(\frac{1}{n^2\log^2 n}\right)\nonumber
 \\
 &\qquad= 2\pi n \cE(1)
 \begin{pmatrix}
 0 & 1
 \\
 0 & 0
 \end{pmatrix}
 \cE^{-1}(1)
 \int_1^{1+1/n} \cK^2_1\left(\frac{\cW}{\hF}(s)\right)^2\left(\frac{F^2}{w_+}(s)+\frac{F^2}{w_-}(s)-2\hFw(s)\right) {\rm d}s\nonumber
 \\
 &\phantom{\qquad=\times}{}+ O\left(\frac{1}{n^2\log^2 n}\right)\nonumber
 \\
 &\qquad= \pi n \begin{pmatrix}
 \I & 1
 \\
 1 & -\I
 \end{pmatrix}
 \underbrace{\int_1^{1+1/n} \cK^2_1 \left(-\frac{3\pi^2}{\log^2 \frac{2}{s-1}} + O\left(\frac{1}{\log^3(s-1)}\right)\right) {\rm d}s }_{= \frac{3}{8 n^2\log^2 n} + O\big(\frac{1}{n^2\log^3 n}\big) \ \text{by \cite[Proposition~C.1]{CD}}}
 + O\left(\frac{1}{n^2\log^2 n}\right)\nonumber
 \\
 &\qquad= \frac{3\pi\I}{8n \log^2n} \begin{pmatrix}
1 & -\I\\
-\I & -1
\end{pmatrix}+ O\left(\frac{1}{n\log^3 n}\right),\label{firstFormula}
\end{align}
which after dividing by $2\pi \I$ is equal to \eqref{Int}.

To show \eqref{differenceInt} we proceed just a before, but use a more precise expansion of $\cR^{(n)}$, see \cite[equation~(8.7)]{KMVV}:
\begin{align*}
 \cR^{(n)}(z) = \Id + \frac{\cR_1(z)}{n} + O\left(\frac{1}{n^2}\right),
\end{align*}
where the $O\big(\frac{1}{n^2}\big)$-term is uniform in $z$.
We use this expansion in \eqref{veryLongEq} to obtain
\begin{align}
 &-\int_1^{1+1/n} \hmu_n(\bv_n-\hv_n)\hmu_n^{-1} {\rm d}s\nonumber
 \\
 &\qquad= -2\pi n \int_1^{1+1/n} \left( \Id + \frac{\cR_1(s)}{n}+ O\left( \frac{1}{n^2} \right)\right)\cE(s)
 \begin{pmatrix}
 1 & 0
 \\
 0 & \frac{1}{2\pi n}
 \end{pmatrix}\nonumber
 \\
 &\phantom{\qquad= }{}\times \begin{pmatrix}
 \cK_1\cK_2 & -\cK^2_1
 \\
 \cK^2_2 & -\cK_1\cK_2
 \end{pmatrix}
 \begin{pmatrix}
 \frac{1}{2\pi n} & 0
 \\
 0 & 1
 \end{pmatrix}\cE^{-1}(s)\left( \Id - \frac{\cR_1(s)}{n}+ O\left( \frac{1}{n^2} \right)\right)\nonumber
 \\
 &\phantom{\qquad= }{}\times \left(\frac{\cW}{\hF}(s)\right)^2\left(\frac{F^2}{w_+}(s)+\frac{F^2}{w_-}(s)-2\hFw(s)\right) {\rm d}s.\label{IntwithR}
\end{align}
A similar expression holds for the second integral in \eqref{differenceInt} but with $n+1$ instead of $n$. Next define $y_1 = n^2f(s)$ and $y_2 = (n+1)^2f(s)$. Note that for $s \in (1,1+1/n)$, we have $0<y_1, y_2 < cn$ for some $c > 0$. Let us denote by $f_1^{-1}$ the local inverse of $f$ around $1$. Then we have with some constant $a$
\begin{align*}
 {\rm d}s &= \frac{1}{n^2}\big(f_1^{-1}\big)'\left(\frac{y_1}{n^2}\right) {\rm d}y_1 = \frac{2{\rm d}y_1}{n^2}\left(1 + \frac{ay_1}{n^2} + O\left(\frac{1}{n^2}\right)\right),
 \\
 {\rm d}s &= \frac{1}{(n+1)^2}\big(f_1^{-1}\big)'\left(\frac{y_2}{(n+1)^2}\right) {\rm d}y_2 = \frac{2{\rm d}y_2}{(n+1)^2}\left(1 + \frac{ay_2}{(n+1)^2} + O\left(\frac{1}{n^2}\right)\right).
\end{align*}
Now performing the substitution $y_1 = n^2f(s)$ in the expression \eqref{IntwithR}, we get
\begin{align}
 &-\int_1^{1+1/n} \hmu_n(\bv_n-\hv_n)\hmu_n^{-1} {\rm d}s\nonumber
 \\
 &\qquad=- \frac{4\pi}{n} \int_0^{n^2f(1+1/n)} \left( \Id + \frac{\cR_1\big(f_1^{-1}\big(\frac{y_1}{n^2}\big)\big)}{n}+ O\left( \frac{1}{n^2} \right)\right)\cE\left(f_1^{-1}\left(\frac{y_1}{n^2}\right)\right)\nonumber
 \\
 &\phantom{\qquad=}{}\times
 \begin{pmatrix}
 1 & 0
 \\
 0 & \frac{1}{2\pi n}
 \end{pmatrix}
 \begin{pmatrix}
 \cK_1(y_1)\cK_2(y_1) & -\cK^2_1(y_1)
 \\
 \cK^2_2(y_1) & -\cK_1(y_1)\cK_2(y_1)
 \end{pmatrix}
 \begin{pmatrix}
 \frac{1}{2\pi n} & 0
 \\
 0 & 1
 \end{pmatrix}\nonumber
 \\
 &\phantom{\qquad=}{}\times
\cE^{-1}\left(f_1^{-1}\left(\frac{y_1}{n^2}\right)\right)\left( \Id - \frac{\cR_1\big(f_1^{-1}\big(\frac{y_1}{n^2}\big)\big)}{n}+ O\left( \frac{1}{n^2} \right)\right)\nonumber
 \\
 &\phantom{\qquad=}{}\times
 \left(\frac{\cW}{\hF}\left(f_1^{-1}\left(\frac{y_1}{n^2}\right)\right)\right)^2\left(\frac{F^2}{w_+}\left(f_1^{-1}\left(\frac{y_1}{n^2}\right)\right)+\frac{F^2}{w_-}\left(f_1^{-1}\left(\frac{y_1}{n^2}\right)\right)-2\hFw\left(f_1^{-1}\left(\frac{y_1}{n^2}\right)\right)\right)\nonumber
 \\
 &\phantom{\qquad=}{}\times \left(1 + \frac{ay_1}{n^2} + O\left(\frac{1}{n^2}\right)\right){\rm d}y_1.\label{LongExp1}
\end{align}
We get a similar expression for the second integral in \eqref{differenceInt} after the change of variables $s = f_1^{-1}\big(\frac{y_2}{(n+1)^2}\big)$:
\begin{align}
&-\int_1^{1+1/(n+1)} \hmu_{n+1}(\bv_{n+1}-\hv_{n+1})\hmu_{n+1}^{-1} {\rm d}s\nonumber
 \\
 &\qquad=- \frac{4\pi}{n+1} \int_0^{(n+1)^2f(1+1/(n+1))} \left( \Id + \frac{\cR_1\big(f_1^{-1}\big(\frac{y_2}{(n+1)^2}\big)\big)}{n+1}+ O\left( \frac{1}{n^2} \right)\right)\nonumber
 \\
 &\phantom{\qquad=}{}\times\cE\left(f_1^{-1}\left(\frac{y_2}{(n+1)^2}\right)\right)
 \begin{pmatrix}
 1 & 0
 \\
 0 & \frac{1}{2\pi (n+1)}
 \end{pmatrix}
 \begin{pmatrix}
 \cK_1(y_2)\cK_2(y_2) & -\cK^2_1(y_2)
 \\
 \cK^2_2(y_2) & -\cK_1(y_2)\cK_2(y_2)
 \end{pmatrix}
\nonumber
 \\
 &\phantom{\qquad=}{}\times \begin{pmatrix}
 \frac{1}{2\pi (n+1)} & 0
 \\
 0 & 1
 \end{pmatrix}
\cE^{-1}\left(f_1^{-1}\left(\frac{y_2}{(n+1)^2}\right)\right)\left( \Id - \frac{\cR_1\big(f_1^{-1}\big(\frac{y_2}{(n+1)^2}\big)\big)}{n+1}+ O\left( \frac{1}{n^2} \right)\right)\nonumber
 \\
 &\phantom{\qquad=}{}\times
 \left(\frac{\cW}{\hF}\left(f_1^{-1}\left(\frac{y_2}{(n+1)^2}\right)\right)\right)^2\bigg(\frac{F^2}{w_+}\left(f_1^{-1}\left(\frac{y_2}{(n+1)^2}\right)\right)+\frac{F^2}{w_-}\left(f_1^{-1}\left(\frac{y_2}{(n+1)^2}\right)\right)
\nonumber \\
 &\phantom{\qquad=\times}{}-2\hFw\left(f_1^{-1}\left(\frac{y_2}{(n+1)^2}\right)\right)\bigg) \left(1 + \frac{ay_2}{(n+1)^2} + O\left(\frac{1}{n^2}\right)\right){\rm d}y_2.\label{LongExp2}
 \end{align}
Now for $0< y < cn$, we have uniformly
\begin{align*}
& \frac{\cR_1\big(f_1^{-1}\big(\frac{y}{n^2}\big)\big)}{n} - \frac{\cR_1\big(f_1^{-1}\big(\frac{y}{(n+1)^2}\big)\big)}{n+1}= O\left(\frac{1}{n^2}\right),
 \\
& \cE\left(f_1^{-1}\left(\frac{y}{n^2}\right)\right) - \cE\left(f_1^{-1}\left(\frac{y}{(n+1)^2}\right)\right)= O\left(\frac{y}{n^3}\right) \lesssim O\left(\frac{1}{n^2}\right),
 \\
& \begin{pmatrix}
 1 & 0
 \\
 0 & \frac{1}{2\pi n}
 \end{pmatrix} - \begin{pmatrix}
 1 & 0
 \\
 0 & \frac{1}{2\pi (n+1)}
 \end{pmatrix}= O\left(\frac{1}{n^2}\right),
 \\
 &\frac{ay}{n^2}-\frac{ay}{(n+1)^2}= O\left( \frac{y}{n^3}\right) \lesssim O\left(\frac{1}{n^2}\right),
 \\
& \frac{\cW}{\hF}\left(f_1^{-1}\left(\frac{y}{n^2}\right)\right) - \frac{\cW}{\hF}\left(f_1^{-1}\left(\frac{y}{(n+1)^2}\right)\right)= O\left(\frac{y^{1/2}}{n^2}\right),
\end{align*}
where we used \cite[Lemma~6.4]{KMVV} in the last estimate. Note that all these error terms can be uniformly bounded by \smash{$O\big(\frac{1+y^{1/2}}{n^2}\big)$}. Additionally, we can choose $r_n = f_1^{-1}\big(\frac{y}{n^2}\big)-1 = O\big( \frac{y}{n^2}\big)$, $\tilde r_n = f_1^{-1}\big(\frac{y}{(n+1)^2}\big)-1 = O\big( \frac{y}{n^2}\big)$ in Proposition~\ref{Proprr} to obtain
\begin{align*}
 &\frac{F^2}{w_+}\left(f_1^{-1}\left(\frac{y}{n^2}\right)\right)+\frac{F^2}{w_-}\left(f_1^{-1}\left(\frac{y}{n^2}\right)\right) - \frac{F^2}{w_+}\left(f_1^{-1}\left(\frac{y}{(n+1)^2}\right)\right)
 \\
 &\qquad{}-\frac{F^2}{w_-}\left(f_1^{-1}\left(\frac{y}{(n+1)^2}\right)\right) = O\left(\frac{y}{n^2}\log \left| \log \left( \frac{y}{n^2} \right)\right| \right)+O\left(\frac{1}{n\log^3 n}\right).
\end{align*}
Note that indeed $n \big|\frac{r_n}{\tilde r_n} - 1\big| < R$ for an appropriate $R > 0$, and $r_n, \tilde r_n \in (0,1)$ for $n$ large enough due to $0 < y < cn$.

From \cite[equation (10.40)]{DLMF}, we have that
\begin{align}\label{KBesselInfty}
 K_0(u) \sim \sqrt{\frac{\pi}{2u}}\E^{-u}, \qquad K_0'(u) \sim -\sqrt{\frac{\pi}{2u}}\E^{-u}
\end{align}
for $u \to \infty$ with $|\arg u | < \frac{3\pi}{2} - \varepsilon$ and $\varepsilon > 0$, implying that the $\cK_j(y)$ decay exponentially for $y \to \infty$. Therefore, changing the limit of integration from $(n+1)^2f(1+1/(n+1))$ to $n^2f(1+1/n)$ in \eqref{LongExp2} will only introduce an exponentially small error which we will neglect. We are now in a position to evaluate \eqref{differenceInt}, by taking the difference of \eqref{LongExp1} and \eqref{LongExp2}:
\begin{align}
 &= - \int_{\Sigr} \hmu_n(\bv_n-\hv_n)\hmu_n^{-1} {\rm d}s + \int_{\Sigr} \hmu_{n+1}(\bv_{n+1}-\hv_{n+1})\hmu^{-1}_{n+1} {\rm d}s\nonumber
 \\
 &=\left(-\frac{4\pi}{n} + \frac{4\pi}{n+1}\right) \int_0^{n^2 f(1+1/n)} \left(\Id + O\left(\frac{1}{n}\right)\right)\cE\left(f_1^{-1}\left(\frac{y}{n^2}\right)\right)\nonumber
 \\
 &\qquad\times
 \begin{pmatrix}
 1 & 0
 \\
 0 & \frac{1}{2\pi n}
 \end{pmatrix}
 \begin{pmatrix}
 \cK_1(y)\cK_2(y) & -\cK^2_1(y)
 \\
 \cK^2_2(y) & -\cK_1(y)\cK_2(y)
 \end{pmatrix}
 \begin{pmatrix}
 \frac{1}{2\pi n} & 0
 \\
 0 & 1
 \end{pmatrix}\nonumber
 \\
 &\qquad\times
\cE^{-1}\left(f_1^{-1}\left(\frac{y}{n^2}\right)\right)\left(\Id + O\left(\frac{1}{n}\right)\right)
 \left(\frac{\cW}{\hF}\left(f_1^{-1}\left(\frac{y}{n^2}\right)\right)\right)^2\nonumber
 \\
 &\qquad\times
 \left(\frac{F^2}{w_+}\left(f_1^{-1}\left(\frac{y}{n^2}\right)\right)+\frac{F^2}{w_-}\left(f_1^{-1}\left(\frac{y}{n^2}\right)\right)-2\hFw\left(f_1^{-1}\left(\frac{y}{n^2}\right)\right)\right)\frac{1}{2} (f_1^{-1})'\left(\frac{y}{n^2}\right){\rm d}y\nonumber
 \\
 &\qquad\phantom{\times}+ \frac{4\pi}{n}O\left(\left\Vert\begin{pmatrix}
 \cK_1(y)\cK_2(y) & -\cK^2_1(y)
 \\
 \cK^2_2(y) & -\cK_1(y)\cK_2(y)
 \end{pmatrix}\left( \frac{1+y^{1/2}}{n^2}\right)\right\Vert_{L^1(0, n^2 f(1+1/n))}\right)\nonumber
 \\
 &\qquad\phantom{\times}+\frac{4\pi}{n} \bigg\Vert \left[O\left(\frac{y}{n^2}\log \left|\log \left( \frac{y}{n^2}\right)\right| \right)+O\left(\frac{1}{n\log^3 n}\right)\right]\nonumber
 \\
 &\qquad\times\begin{pmatrix}
 \cK_1(y)\cK_2(y) & -\cK^2_1(y)
 \\
 \cK^2_2(y) & -\cK_1(y)\cK_2(y)
 \end{pmatrix} \bigg\Vert_{L^1(0, n^2 f(1+1/n))}.\label{L1norms}
\end{align}
Note that the $y^{1/2}$ in the first $L^1$-norm is absorbed by the exponential decay of $\cK_j(y)$, which are in $L^2(\mathbb{R}_+)$ because of \eqref{BesselAsym0} and \eqref{KBesselInfty}, implying that this norm is finite and of order $O\big(\frac{1}{n^2}\big)$. Now observe that for $\E^{-1} \leq y < cn$, we have
\begin{align*}
 \log \left| \log \left( \frac{y}{n^2} \right)\right|= \log\big(-\log y + \log n^2\big) \leq \log\big(1+ \log n^2\big) \lesssim O(\log \log n),
\end{align*}
by the monotonicity of the logarithm, while for $0<y < \E^{-1}$ we have
\begin{align*}
 \log \left| \log \left( \frac{y}{n^2} \right)\right|= \log\big(-\log y + \log n^2\big) &\leq \log(-2\log y ) +\log \big(2\log n^2\big)
 \\
 &\lesssim O(\log |\log y|) + O(\log \log n),
\end{align*}
again by the monotonicity of the logarithm.

Hence, for $0 < y < cn$, we have
\begin{align*}
 O\left(\frac{y}{n^2}\log \left|\log \left( \frac{y}{n^2}\right)\right| \right)&\lesssim O\left( \frac{y}{n^2} \left[\log | \log y |\chi_{(0, \E^{-1})}+ \log \log n\right]\right)
 \\
 &\lesssim O\left(\frac{\log \log n}{n^2} \right) \big(O(y)+1\big).
\end{align*}
We see that the growth of $O(y)$ can be again absorbed into the exponential decay of $\cK_j(y)$, implying that the second $L^1$-norm in \eqref{L1norms} can be bounded by $O\big(\frac{1}{n \log^3 n}\big)$. We summarize:
\begin{align*}
 &= -\int_{\Sigr} \hmu_n(\bv_n-\hv_n)\hmu_n^{-1} {\rm d}s + \int_{\Sigr} \hmu_{n+1}(\bv_{n+1}-\hv_{n+1})\hmu^{-1}_{n+1} {\rm d}s
 \\
 &=-\frac{4\pi}{n(n+1)} \int_0^{n^2 f(1+1/n)} \left(\Id + O\left(\frac{1}{n}\right)\right)\cE\left(f_1^{-1}\left(\frac{y}{n^2}\right)\right)
 \\
 &\quad\times
 \begin{pmatrix}
 1 & 0
 \\
 0 & \frac{1}{2\pi n}
 \end{pmatrix}
 \begin{pmatrix}
 \cK_1(y)\cK_2(y) & -\cK^2_1(y)
 \\
 \cK^2_2(y) & -\cK_1(y)\cK_2(y)
 \end{pmatrix}
 \begin{pmatrix}
 \frac{1}{2\pi n} & 0
 \\
 0 & 1
 \end{pmatrix}
\\
&\quad\times\cE^{-1}\left(f_1^{-1}\left(\frac{y}{n^2}\right)\right)\left(\Id + O\left(\frac{1}{n}\right)\right)
 \left(\frac{\cW}{\hF}\left(f_1^{-1}\left(\frac{y}{n^2}\right)\right)\right)^2
 \\
 &\quad\times
 \left(\frac{F^2}{w_+}\left(f_1^{-1}\left(\frac{y}{n^2}\right)\right)+\frac{F^2}{w_-}\left(f_1^{-1}\left(\frac{y}{n^2}\right)\right)-2\hFw\left(f_1^{-1}\left(\frac{y}{n^2}\right)\right)\right)\frac{1}{2}\big(f_1^{-1}\big)'\left(\frac{y}{n^2}\right){\rm d}y
 \\
 &\phantom{\quad\times}{}+O\left(\frac{1}{n^2\log^3 n}\right).
\end{align*}
 Inverting the change of variables $y = n^2 f(s)$ and Taylor expanding $\cE$, this becomes equal to
\begin{align}
 &-\frac{2\pi n}{n+1} \int_1^{1+1/n} \left( \cE(1) + O\left( \frac{1}{n} \right)\right)
 \begin{pmatrix}
 1 & 0
 \\
 0 & \frac{1}{2\pi n}
 \end{pmatrix}\nonumber
 \\
 &\qquad\times \begin{pmatrix}
 \cK_1\cK_2 & -\cK^2_1
 \\
 \cK^2_2 & -\cK_1\cK_2
 \end{pmatrix}
 \begin{pmatrix}
 \frac{1}{2\pi n} & 0
 \\
 0 & 1
 \end{pmatrix}
 \left( \cE^{-1}(1) + O\left( \frac{1}{n} \right)\right)\nonumber
 \\
 &\qquad\times
 \left(\frac{\cW}{\hF}(s)\right)^2
 \left(\frac{F^2}{w_+}(s)+\frac{F^2}{w_-}(s)-2\hFw(s)\right){\rm d}s
+O\left(\frac{1}{n^2\log^3 n}\right).\label{n+1term}
\end{align}
However, \eqref{n+1term} is precisely $\frac{1}{n+1}$-times the integral in the last line of \eqref{veryLongEq}, which was shown to be equal to $\frac{3\pi\I}{8n \log^2n} \left(\begin{smallmatrix}
\phantom{-}1 & -\I\\
-\I & -1
\end{smallmatrix}\right)+ O\big(\frac{1}{n\log^3 n}\big)$ in equation~\eqref{firstFormula}. Thus, the expression in \eqref{n+1term} is equal to
\begin{align*}
 &\frac{1}{n+1}\left(\frac{3\pi\I}{8n \log^2n} \begin{pmatrix}
\phantom{-}1 & -\I\\
-\I & -1
\end{pmatrix}+ O\left(\frac{1}{n\log^3 n}\right)\right)+ O\left(\frac{1}{n^2\log^3 n}\right)
\\
&\qquad=\frac{3\pi\I}{8n^2 \log^2n} \begin{pmatrix}
\phantom{-}1 & -\I\\
-\I & -1
\end{pmatrix}+ O\left(\frac{1}{n^2\log^3 n}\right).
\end{align*}
Taking into account the $\frac{1}{2\pi \I}$-prefactor, this is seen to be equal to the right hand side of \eqref{differenceInt}, finishing the proof.
\end{proof}

\subsection*{Acknowledgments}
The work of P.D.~was supported in part by a Silver Grant from NYU. The work of M.P.\ was supported by the Methusalem grant METH/21/03 -- long term structural funding of the Flemish Government and the National Science Foundation under Grant No.~1440140. The authors also gratefully acknowledge the support of MSRI (now the Simons Laufer Mathematical Sciences Institute) where the work in this paper was begun when the authors attended a semester program in Fall 2021.

\pdfbookmark[1]{References}{ref}
\LastPageEnding

\end{document}